\documentclass[10pt]{amsart}
\usepackage{listings,graphicx,amsmath,varioref,amscd,ulem, amssymb,color,bm,stmaryrd,amsthm,amsfonts,graphics,geometry,latexsym,pgf,pst-all} 
\theoremstyle{plain}
\usepackage{esint}
\usepackage{color}

\usepackage{amsthm}
\theoremstyle{plain}
\newtheorem{theorem}{Theorem}[section]
\newtheorem{proposition}[theorem]{Proposition}
\newtheorem{lemma}[theorem]{Lemma}

\usepackage{geometry}
\usepackage[colorlinks=false]{hyperref}

\theoremstyle{definition}
\newtheorem{defin}[theorem]{Definition}

\newtheorem{remark}[theorem]{Remark}

\theoremstyle{remark}

\def\div{{\rm div}}

\def\huz{H^{1}_{0}(\Omega)}
\def\huzu{H^{1}_{0}(0,L)}

\def\bk{\color{black}}
\def\ou{\overline{u}}
\def\re{\mathbb{R}}
\numberwithin{equation}{section}
 \makeatletter
\@namedef{subjclassname@2020}{%
  \textup{2020} Mathematics Subject Classification}
\makeatother

\def\dys{\displaystyle}

\def\vare{\varepsilon}

\DeclareMathOperator{\R}{\mathbb{R}}

\newcommand{\car}[1]{\raise1pt\hbox{$\chi$}_{#1}}

\def\og{\leavevmode\raise.3ex\hbox{$\scriptscriptstyle\langle\!\langle$~}}
\def\fg{\leavevmode\raise.3ex\hbox{~$\!\scriptscriptstyle\,\rangle\!\rangle$}}

\author[D. Giachetti]{Daniela Giachetti}
\address[D. Giachetti]{Dipartimento di Scienze di Base e Applicate per l'Ingegneria, Sapienza Universit\`a di Roma, Via Scarpa 16, 00161 Roma, Italy 
	\\ daniela.giachetti@sbai.uniroma1.it}
	\author[P. J. Mart\'inez-Aparicio]{Pedro J. Mart\'inez-Aparicio}
\address[P. J. Mart\'inez-Aparicio]{Departamento de Matem\'aticas,
Universidad de Almer\'ia, Ctra. Sacramento s/n, La Ca\^{n}ada de San Urbano, 04120 Almer\'ia, Spain, pedroj.ma@ual.es}
	\author[F. Murat]{Fran\c{c}ois Murat}
\address[F. Murat]{Laboratoire Jacques-Louis Lions, Bo\^ite courrier 187, Sorbonne Universit\'e, 4 place Jussieu, 75252 Paris cedex 05, France, francois.murat@sorbonne-universite.fr}
\author[F. Petitta]{Francesco Petitta}
\address[F. Petitta]{Dipartimento di Scienze di Base e Applicate per l'Ingegneria, Sapienza Universit\`a di Roma, Via Scarpa 16, 00161 Roma, Italy 
	\\ francesco.petitta@uniroma1.it}

\begin{document}

\title[Unexpected phenomena  in a one-dimensional singular elliptic equation]{Unexpected phenomena \\ in a one-dimensional elliptic equation \\with a singular first order divergence term}

\begin{abstract}
In this paper we study the possible solutions $u$ of the one-dimensional non-linear singular problem which formally reads as 
\begin{equation}\label{intro}\tag{S}
\begin{cases}
\dys -\frac{d}{dx}\left(a(x) \frac{d u}{dx}\right)  = - \frac{d \phi (u) }{dx}- \frac{d g(x)  }{dx}& \text{in}\;(0,L),\\
u(0)=u(L)=0\,, & 
\end{cases}
\end{equation}
where $L>0$, and where the data ($a$, $g$, $\phi$) are as follow:   $a$ is a function of $L^\infty(0,L)$ which is bounded between two positive constants,   $g$ is a function of $L^2(0,L)$, and the singular function $\phi:\mathbb{R}\mapsto \mathbb{R}\cup \{+\infty\}$  is   continuous as a function with values in  $\mathbb{R}\cup \{+\infty\}$, and satisfies $\phi(0)=+\infty$ and $\phi(s)<+\infty$ for every $s\in\mathbb{R}$, $s\not=0$; the model example for the singular function $\phi$ is 
%$\phi_\gamma(s)=\frac{1}{|s|^\gamma}$ 
$\phi_\gamma(s)={|s|^{-\gamma}}$
with $\gamma>0$.

We first study the behaviour of the solutions of approximating problems  (S$_n$) involving  non-singular functions $\phi_n$ which converge to $\phi$ in a sense that we specify, and we prove that these solutions have subsequences which either converge to weak solutions of \eqref{intro} (for a definition of weak solution that we specify), or converge to zero. We then prove that for a large class of  data ($a$, $g$, $\phi$) it does not exist any weak solution of \eqref{intro}, while for another large class of  data ($a$, $g$, $\phi$) it exists at least one weak solution of \eqref{intro}.

Thanks to the study of an associated singular ODE (this study is of independent interest), we prove that under additional assumptions which are satisfied by the model example 
$\phi_\gamma(s)={|s|^{-\gamma}}$ when $0<\gamma<1$, if for some data ($a$, $g$, $\phi$) there exists one weak solution of \eqref{intro}, then for the same data it also exists an infinity of weak solutions of \eqref{intro} which are parametrized by $c\in(-\infty,c^*]$ for some finite $c^*$.

We finally prove that for any given data ($a$, $g$, $\phi$)  and for any weak solution $u$ of \eqref{intro} corresponding to  these data, there exist sequences of data ($a$, $g_n$, $\phi_n$) with non-singular functions $\phi_n$ which converge to ($a$, $g$, $\phi$) for which the solutions  converge to $u$, while there also exist other sequences of data ($a$, $g_n$, $\phi_n$) with non-singular functions $\phi_n$ which converge to ($a$, $g$, $\phi$) for which the solutions converge to zero.

Most of these results are unexpected.

\end{abstract}

\maketitle

\tableofcontents

\section{Introduction}

\noindent{\bf Setting of the problem} 

In the present paper we deal with a  singular one-dimensional problem. Our main aim  consists in finding a function $u$ which formally satisfies

\begin{equation}
\label{pb1i}
\begin{cases}
\displaystyle-\frac{d}{dx}\left(a(x) \frac{d u}{dx}\right)  = - \frac{d \phi (u)}{dx} - \frac{d g }{dx} & \text{in}\;(0,L),\\
u(0)=u(L)=0\,. & 
\end{cases}
\end{equation}

We will soon give a mathematically correct   and natural definition of solutions of this problem, but let us first give the main assumptions on its data.  \\

 We will assume that the data ($a$, $g$, $\phi$) satisfy 
\begin{equation}\label{a1i}
a\in L^{\infty}(0,L)\,,\ \ \exists\ \alpha>0: \ a(x)\geq \alpha\,,\ \ \text{a.e. in}\ \ (0,L)\,;
\end{equation}
\begin{equation}\label{g1i}
g\in L^2(0,L)\,;
\end{equation}
\begin{equation}
\label{condphi1i}
\begin{cases}
\phi:\R\mapsto\R\cup\{+\infty\},
\,\,\phi\,\, \mbox{is continuous with values in } \R\cup\{+\infty\},\\
 \phi(s)<+\infty \ \ \forall s\in\R,\,\, \ s\neq 0\,.
\end{cases}
\end{equation}

Our main purpose is to study the case where 
\begin{equation} \label{p3i}\phi(0)=+\infty\,\ \ \ \end{equation}
even though some results will be true (and new) also in the case where $\phi(0)<+\infty.$\\

The model case for the function $\phi$ is 
\begin{equation*}
\dys\phi_\gamma (s)=\frac{c}{|s|^{\gamma}} +\varphi (s),
\mbox{ with } c>0, \,\gamma>0,\, \varphi\in C^0(\R)\,.
\end{equation*}

Assuming $ \phi(0)=-\infty$ in place of $ \phi(0)=+\infty$  is just a variant of problem \eqref{pb1i} by a simple change of variable (see Remark~\ref{rm-infty} below) and we will not treat this case.
 \bigskip
 
 \medskip 
 \noindent{\bf Definition of a weak solution}  
 
 We introduce  the following definition of weak solution of problem \eqref{pb1i}. For more details see Definition~\ref{defin} and Subsection~\ref{definitionws} below.

\begin{defin}\label{defini}
We say that $u$ is a {\it weak solution} of problem \eqref{pb1i} if $u$ satisfies 
\begin{equation}\label{form-defi}
\begin{cases}
\displaystyle u\in \huzu\,, &\phi(u)\in L^2(0,L),  \\ \\ 
\displaystyle-\frac{d}{dx}\left(a(x) \frac{d u}{dx}\right) & = \displaystyle- \frac{d \phi (u)}{dx} \displaystyle- \frac{d g }{dx} \ \  \text{in} \ \ \mathcal{D}'(0,L)\,.
\end{cases}
\end{equation} \qed
\end{defin}

Note that \eqref{form-defi} holds true
if and only if $u$ solves  (see Proposition~\ref{remequiv})
\begin{equation}\label{equivi}
\begin{cases}
 u\in \huzu \quad \phi(u)\in L^2(0,L), &\\
\dys a (x)\frac{d u}{dx}    =  \phi (u) + g +c & \text{in}\;(0,L)\,,
\end{cases}
\end{equation}
for 
\begin{equation} 
\displaystyle c= \dys -  \frac{\dys\int_0^L \frac{\phi(u)}{a(x)} dx +\dys\int_0^L\frac{g}{a(x)} dx}{\dys\int_0^L\frac{1}{a(x)} dx}\,.
\label{oce}\end{equation}

 \medskip 
\noindent{\bf   Three  other singular problems   and some  references}

 If one has $ \phi\in C^0(\R) $  and $N\geq 1$, it is proved in \cite{BGDM2} and \cite{BGDM} that there exists a {\it renormalized solution} of the $N$-dimensional problem which formally reads as
\begin{equation}
\label{pb2i}
\begin{cases}
\displaystyle -\div (A(x) \nabla u) = - \div (\phi (u))- \div g& \text{in}\,\,  \Omega,\\
u(x)=0  & \text{on} \,\,\partial \Omega,
\end{cases}
\end{equation}
where $\Omega$ is a bounded domain of $\re^N$ and  $A(x)$ is a bounded matrix satisfying, for some $\alpha>0$,  
$$
A(x)\xi\cdot\xi\geq \alpha |\xi|^2, \ \ \forall \ \xi\in \re^N.
$$
These solutions turn out to be also weak solutions if the growth at infinity of $\phi(s)$ is sufficiently low. 

\bigskip

In the special case  $N=1$ with $ \phi\in C^0(\R)$  this notion coincides with the notion of classical weak solution:
\begin{equation*}
\label{vf2i}
\begin{cases}
\dys u\in \huzu,  \\ 
\dys\int_0^L a(x)\frac{du}{dx}\frac{dz}{dx}\, dx= \int_0^L \phi(u)\frac{dz}{dx}\,dx+\int_0^L g(x)\frac{dz}{dx}\,dx\quad \forall z\in H_0^1(0,L),
\end{cases}
\end{equation*}
and there exists a classical   weak solution of this problem (see Proposition \ref{propexis} below).

\bigskip
 Concerning semilinear problems involving singular lower order terms $h(x,u)$, namely
 \begin{equation}\label{prob:g=0h}
    \left\{\begin{array}{ll}
   \displaystyle-\div(A(x) \nabla u) = h(x,u) &\text{ in } \Omega,\\
   u=0 &\text{ on } \partial\Omega,
\end{array}\right.
\end{equation} 
 let us quote several important papers. The cases $\dys h(x,s)=f(x)e^{1/s}$ or $\dys h(x,s)=\frac{f(x)}{s^\gamma}$ for a regular function $f(x)$ are treated in~\cite{Fulks-Maybee} where the authors prove the existence of a classical solution when  $A(x)$ is  the identity matrix.  In~\cite{Crandall, Stuart} similar results are proved for a regular matrix $A(x)$ and a regular function $h(x,s)$ uniformly bounded for $s>1$ with $\lim_{s\to 0} h(x,s) = +\infty$ uniformly for $x\in\overline{\Omega}$. Moreover,  continuity properties of the solution are proved in \cite{Crandall} if $h(x,s)$ does not depend on $x$.

The case where the nonlinearity is of the form $\dys h(x,s)=\frac{f(x)}{s^\gamma}$ with $f(x)$ a positive H\"older continuous function in $\overline{\Omega}$ is studied in  \cite{LM} where it is proved that problem \eqref{prob:g=0h} has a classical solution which does not always belong to $H_0^1(\Omega)$. More precisely the authors prove in \cite{LM} that the solution belongs to $H_0^1(\Omega)$ if and only if $\gamma<3$. Furthermore, they demonstrate that for $\gamma>1$ the solution is not in ${C}^1(\overline{\Omega})$. 
In the case $h(x,s)=f(x)\Tilde{h}(s)$, some extensions may be found, among others,   in~\cite{Lair-1,Lair-2} for $\Omega=\mathbb{R}^N$ and  in~\cite{Z} for bounded domains. In the latter case $f(x)$ may also be singular   at the boundary of $\Omega$.

Let us highlight the paper~\cite{BO}, in which the authors extensively study the semi-linear problem  in  the case $\dys h(x,s)=\frac{f(x)}{s^\gamma}$ with $f\geq 0$, $f\in L^m(\Omega)$ for $m\geq 1$,   where they prove existence results depending on $\gamma$ and on the summability of $f$. For $\gamma=1$ and $f\in L^1(\Omega)$, they prove the existence of a solution belonging to $H_0^1(\Omega)$.  They also prove a similar result when   $f\in L^m(\Omega)$ with $m\geq  C(N, \gamma)>1$. Finally, for the case $\gamma>1$ and $f\in L^1(\Omega)$ they prove the existence of a solution $u$ belonging to $H^1_\mathrm{loc}(\Omega)$ satisfying $u^{\frac{\gamma+1}{2}}\in H_0^1(\Omega)$.

In~\cite{OP} (see also \cite{AMM}),  the authors prove the existence of a   solution in $H_0^1(\Omega)$ if $f\in L^m(\Omega)$, $f$ positive,   $1<\gamma<3-\frac{2}{m}$. These results are optimal  for $f\in L^\infty(\Omega)$, and they  fit  with the result  of  \cite{LM}, i.e. $u\in H_0^1(\Omega)$ for all $\gamma<3$ if one  formally takes $m=+\infty$.

In \cite{GMM4, GMM2,GMM,GMM1,GMM3} the authors introduce, in the case of strong singularities (which in the model case corresponds to $\gamma>1$), a new definition of the solution, with a space for the solution which is unconventional, and test functions which are reminiscent of  the notion of solution defined by transposition. In this framework they prove results of existence, stability and uniqueness. They also prove \cite{GMM2,GMM1} results of homogenization in this framework. Another step in this direction is \cite{CD-M} where the right-hand side of the equation can change sign. Other homogenization results can be found in \cite{BC} and \cite{D-G}.

Let us also  point out that the cases $\dys h(x,s)= \frac{f(x)}{s^\gamma}+\mu$ and $h(x,s)=\mu \Tilde{h}(s)$ with $\mu$ a nonnegative Radon measure have been studied in \cite{OP1}. Moreover, the case of a variable exponent , i.e. $\dys h(x,s)=\frac{f(x)}{s^{\gamma(x)}}$ is considered in \cite{CMA}. For more details  and a gentle introduction to singular elliptic problems we refer to the recent survey  \cite{op}.

\bigskip 

Finally, the case of singularities in first order terms with natural growth in the gradient has been also extensively studied. For a far to be complete account on these problems see \cite{bo1,pe,a6, b4,bop}, and references therein. 
 
 \medskip 
\noindent{\bf Originality of problem \eqref{pb1i}}

  Let us emphasize some  special features which are peculiar of problem \eqref{pb1i}. Both in the singular case (i.e. $\phi(0)=+\infty$) and in  the non-singular case (i.e. $ \phi\in C^0(\R)$), for any (possible) weak solution $u$ of \eqref{pb1i} it can be proven  that  $\int_0^L \phi(u)\frac{du}{dx}\,dx=0$, since one has \begin{equation*}
  \mbox{if } z\in H_0^1(0,L) \mbox{ with } \phi(z)\in L^2(0,L), \mbox{ then }\dys \int_0^L \phi(z)\frac{dz}{dx}\,dx=0,\end{equation*} 
(see Lemma \ref{prop24ter} below, where the hypothesis $\phi(z)\in L^2(0,L)$ is essential). This implies for the weak solutions of \eqref{pb1i} defined above by \eqref{form-defi} the following important a priori estimate
\begin{equation}
\label{propduii}
\left\|\frac{du}{dx}\right\|_{L^2(0,L)}\leq \frac{1}{\alpha} \|g\|_{L^2(0,L)},\end{equation}  
and, by Morrey's embedding   (here $N=1$ is crucial)
\begin{equation}\label{propdu1i}\left\|u\right\|_{L^{\infty}(0,L)}\leq \frac{\sqrt{L}}{\alpha} \|g\|_{L^2(0,L)}\,.
\end{equation}

\bigskip 
 
In order to  (try to) prove the existence of a weak solution of problem \eqref{pb1i}, we proceed as usual by approximation.

 We consider sequences  of  $a_n$, $g_n$ and $\phi_n$ which satisfy 
\begin{equation}
\label{sequencei}
\begin{cases}
a_n\in L^\infty(0,L),\,\, \alpha\leq a_n\leq \beta \mbox{ for some fixed constant } \beta,\\ a_n(x)\to a(x) \mbox{ a.e. } x\in(0,L),
\end{cases}
\end{equation}
\begin{equation}
\label{propgn}
g_n\in L^2(0,L),\,\, g_n\to g \,\, \mbox{weakly in} \,\, L^2(0,L),
\end{equation}
\begin{equation}
\label{phinconvi}
\phi_n\in C^0(\R) \,\,\mbox{ for every } \,\, n\in\mathbb{N}\ \  \text{which satisfy}
\end{equation}
\begin{equation}
\label{phini}
\mbox{if } \,\, s_n\to s\,\, \mbox{in}\,\, \R\,\mbox{ then } \phi_n(s_n)\to\phi(s) \mbox{ in } \R\cup\{+\infty\}; \end{equation}

\noindent  the latest property \eqref{phini} is equivalent to say that the sequence $\phi_n$ (locally) uniformly converges to $\phi$, even in the case in which $\phi(0)=+\infty$ (see Proposition~\ref{rem29} below).

Examples of such approximations of $\phi$
are the truncations (i.e. $
\phi_n(s)=T_n(\phi(s))$) and the  homographic approximations (i.e. $\dys\phi_n(s)=\frac{\phi(s)}{1+\frac 1n \phi(s)}).$

For $\phi_n$ satisfying \eqref{phinconvi}, \eqref{phini} it is quite easy to show  that, for every $n\in\mathbb{N}$,  there exists at least one function $u_n$ which satisfies 
\begin{equation}
\label{intseqi}
\begin{cases}
u_n\in \huzu,\\
\dys \int_0^L a_n(x) \frac{d u_n}{dx} \frac{dz}{dx}\,dx=\int_0^L \phi_n(u_n)\frac{dz}{dx}\,dx +\int_0^L g_n \frac{dz}{dx}\, dx \quad \forall z\in\huzu.
\end{cases}
\end{equation}

Therefore, due to \eqref{propduii}, there exists $u\in\huzu$ such that, for a subsequence
\begin{equation}
\label{convweaki}
u_n\rightharpoonup u \,\,\mbox{ weakly in } \,\, \huzu \,\,\mbox{ and a.e. }\,\, x\in(0,L). 
\end{equation}

 \medskip 
\noindent{\bf An alternative} 

 Our first main result (see Theorem~\ref{maintheo} below) consists in the following   alternative   for the weak limit $u$:  if $\phi(0)=+\infty$,  then

\begin{itemize}
\item either $u\equiv 0$,
\item or $u$ is a weak solution in the sense of Definition~\ref{defini} of problem \eqref{pb1i}. 
\end{itemize}

In the   proof of this alternative we show an original new estimate on the sequence $\phi_n(u_n)$ in $L^2(0,L).$ This estimate, which is specific to the one-dimensional setting, can only be obtained when the sequence $u_n$ weakly converges to a function $u$ which is not identically zero.

Note that the result is a true alternative, in the sense that the two possible situations  are mutually exclusive since $0$ is not a weak solution of \eqref{pb1i} in the sense of Definition~\ref{defini} when $\phi(0)=+\infty$.

\bigskip 
Moreover the two cases of the alternative effectively happen depending on the data $g$ and $\phi$. Indeed on the first hand there is a large class of functions $g\in L^2(0,L)$ (see Theorem \ref{36} below), and a large class of functions $\phi$ which satisfy \eqref{condphi1i} and \eqref{p3i} (see Theorem \ref{45} below), such that the limit of the sequence $u_n$ is always $u\equiv0$. On the other hand there exists (see Section \ref{Sec8} below) a large class of data $g$ and $\phi$ such that \eqref{pb1i} has a weak solution, and for every weak solution, there exists (see Section \ref{Sec7} below) (many) sequences of approximations of the data which converges to this weak solution.

Let us emphasize that hypothesis $\phi(0)=+\infty$ is essential in order to get the alternative. When  $\phi(0)<+\infty$, then for any sequence $\phi_n$ of reasonable approximations of $\phi$, one has $\phi_n(u_n)$ bounded in $L^\infty(0,L)$ and $\phi_n(u_n)$ converges strongly in $L^2(0,L)$ to $\phi(u)$. Moreover, in this setting,  the case $u\equiv 0$ occurs if $g\equiv c\in \R$.

\bigskip 
The results of Theorem \ref{36} and Theorem \ref{45} proved in Section \ref{Sec4} are in fact results of non-existence of weak solutions of problem \eqref{pb1i} in the sense of Definition~\ref{defini}.

The first one,   Theorem \ref{36}, states that when the datum $ g$ is bounded from below, there exists no weak solution of problem \eqref{pb1i}  in the sense of Definition~\ref{defini}.
This is essentially due to the fact that in this case  the datum $g$ cannot compensate the singular behaviour of the function $\phi(u(x))$ when  $u(x)$ tends to zero. 

The second one,   Theorem \ref{45}, states that when the function $\phi(s)$ is not integrable around $s=0$, there exists no weak solution of problem \eqref{pb1i} in the sense of Definition~\ref{defini}.
This is proved by showing that the strong behaviour of the singular function $\phi(s)$ near $s=0$ implies that  the class of functions $ u\in \huzu$, with  $\phi(u)\in L^2(0,L)$ is empty. 

 \medskip 
\noindent{\bf An associated ODE}

  If we look at \eqref{equivi}-\eqref{oce},   which is an equivalent formulation of the definition of weak solution of problem \eqref{pb1i} in the sense of Definition \ref{defini}, we are naturally induced to consider, for $h\in L^2(0,L)$, the ODE 

\begin{equation}\label{formhi}
\begin{cases}
v\in H^1(0,L), \quad \phi(v)\in L^2(0,L), \\
\dys a (x)\frac{d v}{dx}   =  \phi (v) + h ,\,\, \text{ in}\;[0,L],\\
\dys v(0)=0 ,& 
\end{cases}
\end{equation}
which is singular when $\phi(0)=+\infty$.

This problem is  clearly  related to our problem \eqref{pb1i}. Indeed, if $h=g+c$, any solution $u$ of this ODE is a weak solution of \eqref{pb1i} when $c$ is given by \eqref{oce}, and conversely.

In Section \ref{Sec5} we study the ODE \eqref{formhi}; this is original when  $\phi(0)=+\infty$. We obtain existence, positivity, comparison,  and uniqueness results for the solution $v$ of \eqref{formhi}, under further assumptions on $\phi(s)$ and $g$; the assumptions on $\phi$ are satisfied in the model case $\phi_\gamma (s)=\frac{c}{|s|^{\gamma}} +\varphi (s)$ when $c>0$, $0<\gamma<1$,  and $\varphi\in C^0_b(\R)$ with $\phi_\gamma (s)$ non-increasing for $s>0$.

In particular, in order to get existence of solutions $v$ to \eqref{formhi}, we use the integrability of $\phi(s)$ in $s=0$ and its  boundedness at infinity (see Theorem \ref{shoot} below), while for the comparison and uniqueness results (see Proposition \ref{compar} below) we further require that $\phi(s)$ is monotone non-increasing for $s>0$.

A synthesis of the results we proved for \eqref{formhi} is given in Subsection~\ref{sub53}. These results are new and, in our opinion, of  independent interest, because of the singular behaviour of the function $\phi(s)$ in  $s=0.$
 
\indent These results will be  strongly used when we will prove the  existence of weak solutions of problem \eqref{pb1i} (see Section \ref{Sec8} below) as well as the multiplicity result we describe now.

 \medskip 
\noindent{\bf A multiplicity result for the solutions of \eqref{pb1i}}

 Another interesting consequence of these results is a multiplicity result for the weak solutions of problem \eqref{pb1i}  in the sense of Definition~\ref{defini} (see Section~\ref{Sec6} below). This  multiplicity result stated in Theorem \ref{cara} below is quite unexpected: it says that, whenever a solution of problem \eqref{pb1i} in the sense of Definition~\ref{defini} exists for some given data ($a$, $g$, $\phi$), then infinitely many solutions exist;  these solutions can be  indexed by a real  parameter $c$ which varies in an interval $(-\infty, c^*]$ where $c^*$ is finite.  
These infinitely many solutions are strictly ordered with respect to $c$ and any possible solution of \eqref{pb1i} for these data correspond to some $c\in (-\infty , c^*]$.

 \medskip 
\noindent{\bf Stability and instability of the approximations}

  We also show, in Section~\ref{Sec7} below, two quite remarkable results concerning the stability of the weak solutions of \eqref{pb1i} for the data ($a$, $g_n$, $\phi_n$) when $g_n$ converges to $g$ in $L^2(0,L)$ and when $\phi_n$ is a sequence of reasonable approximations of $\phi$. We prove in particular  that for any weak solution $u$ of problem \eqref{pb1i} in the sense of Definition~\ref{defini} such approximations ($a$, $g_n$, $\phi_n$) can be  built whose solutions $u_n$ converge to $u$. Roughly speaking this result asserts that a weak solution of problem  \eqref{pb1i}   in the sense of Definition \ref{defini} is never isolated.  As a counterpart of this result, we also show that for any weak solution $u$ of problem \eqref{pb1i} in the sense of Definition \ref{defini} one can build such approximations ($a$, $g_n$, $\phi_n$) for which the  solutions $u_n$ of the approximating problem  converge to $0$.  This  can be viewed as  a strong instability result.   
Let us stress that in both results, the   approximating   sequence $\phi_n$ can be any reasonable approximation of $\phi$, while  the sequence $g_n$ has to be chosen   accordingly.

 \medskip 
\noindent{\bf Existence of solutions of problem \eqref{pb1i}} 

Finally,  Section \ref{Sec8} below is devoted to produce explicit large classes of data for which solutions to \eqref{pb1i} in the sense of Definition \ref{defini} do exist.

These results are essentially consequences of the fact that given any $\phi$ satisfying \eqref{condphi1i} and \eqref{p3i} as well as the fact that $\phi$ is integrable in $0$, one can construct a large class of functions $u$ such that  $ u\in \huzu$ with  $\phi(u)\in L^2(0,L)$. Defining  $g\in  L^{2}(0,L)$ by $g(x)=a (x)\frac{d u}{dx} -\phi (u)$, one has built data such that problem \eqref{pb1i}  admits $u$ as a weak solution.

A similar result, Theorem~\ref{t85},  is in some sense a ``density result'', since it asserts that for any $\phi$ satisfying \eqref{condphi1i} and \eqref{p3i} as well as  the fact that $\phi$ is integrable in $0$, for any  $g\in L^{2}(0,L)$ and for any $\delta>0$, one can construct a function $\hat{g}\in L^2(0,L)$ such that $\hat{g}\equiv g$ on $(0,L-\delta)$ such that the problem \eqref{pb1i} has a weak solution in the sense of Definition~\ref{defini} for the data  ($a$, $\hat{g}$, $\phi$).

Let us note that most of the results that we obtain and prove in this article, even if simple and obtained through elementary proofs, are new and even original.

In any case, they are unexpected.
 
   \medskip 
\noindent{\bf Concluding  remarks and comments}  

To our great regrets, our results are confined to the one-dimensional setting, since, as said before, in this case we are able to find an original new estimate on  $\phi_n(u_n)$ in $L^2(0,L)$, and since this is no more the case in the $N$-dimensional setting for $N>1$. The only situations that we were able to face are special forms of equations which can be solved by separation of variables, and the case of radial solutions under special (radial) assumptions on the data. We will  publish  (\cite{GMMP1}) these partial results   along with some  variants of  the results which are the subject of the present article, namely
\begin{itemize}
\item the case   where a zero-th order term $+b(x)u$ with $b(x)\geq 0$ (or more generally $+b(x,u)$ with $b(x,s)s\geq 0$) is added to the left hand side of problem \eqref{pb1i};
\item the case  where the linear operator 
$-\frac{d}{dx}\left(a(x) \frac{d u}{dx}\right)  
$
is replaced by a nonlinear monotone operator $-\frac{d}{dx}\left(a(x, \frac{d u}{dx})\right)$ (or possibly by a nonlinear pseudo monotone operator $-\frac{d}{dx}\left(a(x,u,\frac{d u}{dx})\right)$) where  $a: (x,\xi)\in (0,L)\times \mathbb{R}\rightarrow a(x,\xi)\in \mathbb{R}$ is a   Carath\'eodory function that satisfies, for some $p$, $\alpha$, $\beta$, $b(x)$ with $1<p<+\infty$, $\alpha>0$, $\beta\geq\alpha$, $b\in L^p (0,L)$, the classical monotonicity properties 
\begin{equation*}
\begin{cases}
\dys \alpha |\xi|^p\leq a(x,\xi)\xi,\quad |a(x,\xi)|\leq \beta (|\xi|^p+|b(x)|^p),\\
(a(x,\xi)-a(x,\eta))(\xi-\eta)>0, \quad \mbox{a.e. } x\in(0,L),\,\, \forall \xi\in\mathbb{R}, \,\, \forall \eta\in\mathbb{R},\,\, \xi\not=\eta.
\end{cases}
\end{equation*}
\end{itemize} 
\noindent In a second   forthcoming paper  (\cite{GMMP2}) we will treat the problem where the model singular function  $\phi_\gamma (s)=\frac{1}{|s|^{\gamma}}$ is replaced by  a function that is singular in $s=m$ (for $m\neq 0$)  whose model is 
$$
\phi^m_\gamma (s)=\frac{1}{|s-m|^{\gamma}}, \ \ \text{with}\ \ \gamma>0.
$$

\bigskip

{\bf Notation}

 In the present paper we will use classical notations. Here we just recall and precise some of them.

We denote by $C^0(\re)$ the space of functions $\phi:\re\rightarrow\re$ which are continuous at each point of $\re$. Observe that a function $\phi:\re\rightarrow\re\cup\{+\infty\}$ which is continuous (see an example in \eqref{pM} below) does not belong to $C^0(\re)$ when $\phi(s_0)=+\infty$ for some $s_0\in\re$.

We denote by $C^0_b(\re)$ the space of functions of $C^0(\re)$ which are bounded on $\re$, namely $C^0_b(\re)=C^0(\re)\cap L^\infty(\re)$.

We denote by $\mathcal{D}'(0,L)$ the space of distributions on the open interval $(0,L)$, namely the dual of the space $C_c^\infty(0,L)$ of functions which have derivatives of any order and which have a compact support in $(0,L)$.

We denote by $Lip(\re)$ the space of the lipschitz continuous functions on $\re$, namely the space of function $\psi\in C^0(\re)$ such that 
$$
\|\psi\|_{Lip(\re)}=\sup_{s,t\in\re,\,s\not=t} \frac{|\psi(s)-\psi(t)|}{|s-t|}<+\infty.
$$

We denote by $H^1_0 (0,L)$ the  Sobolev space of those  functions $z\in L^2(0,L)$ whose distributional first  derivative $\displaystyle \frac{dz}{dx}$ belongs to $L^2(0,L)$ and which satisfy $z(0)=z(L)=0$. The space $H^1_0 (0,L)$ will be  equipped with the norm
\begin{equation}\label{21b}\|z\|_{H^1_0 (0,L)}=\left\| \frac{d z}{dx}\right\|_{L^2 (0,L)}, \ \ \forall \ z\in H^1_0 (0,L),\end{equation}
since the Poincar\'e inequality asserts that 
\begin{equation}
\label{Poincare}
\|z\|_{L^2(0,L)}\leq {L}\left\|\frac{dz}{dx}\right\|_{L^2(0,L)}, \quad \forall z\in \huzu.
\end{equation}

\noindent Recall that the  Morrey's embedding and estimate, which are specific to the one-dimensional case,  assert  that
\begin{equation}
\label{rem210}
 \huzu\subset L^{\infty} (0,L) \quad \mbox{ with } \quad
\dys  \|z\|_{L^\infty(0,L)}\leq \sqrt{L} \left\|\frac{dz}{dx}\right\|_{L^2(0,L)}, \,\, \forall z\in\huzu,
\end{equation}
and also that
\begin{equation}
\begin{cases} \label{Morrey}
H^1(0,L)\subset C^{0,\frac 12} ([0,L]), \,\, \mbox{ with }\\ 
\displaystyle  \|z\|_{C^{0,\frac12} ([0,L])}=\sup_{\underset{x\neq y}{x,y\in [0,L]}}\frac{|z(x)-z(y)|}{\sqrt{|x-y|}}\leq \left\|\frac{dz}{dx}\right\|_{L^{2}(0,L)},\quad \forall z\in H^1(0,L)\,.
\end{cases}
 \end{equation}

Finally, for $k\in\re^+$, let us  denote by $T_k:\re\rightarrow\re$ the truncation function  at height $k$, i.e.  the function given by
\begin{equation}
\label{tk}
T_k(s)=
\begin{cases}
\dys s & \mbox{if} \,\, |s|\leq k,\\
\dys k\frac{s}{|s|} &\mbox{if} \,\, |s|\geq k,
\end{cases}
\quad \forall s\in \re.
\end{equation}

\section{Assumptions and definitions}\label{sec2}
As mentioned in the Introduction, in this paper  we will study a one-dimensional singular problem that we formally  write as
\begin{equation}
\label{pb1}
\begin{cases}
\dys -\frac{d}{dx}\left(a(x) \frac{d u}{dx}\right)  = - \frac{d \phi (u) }{dx}- \frac{d g(x)  }{dx}& \text{in}\;(0,L),\\
u(0)=u(L)=0\,, & 
\end{cases}
\end{equation}
where $u$ is the unknown and where ($a$, $g$, $\phi$) are data which will be specified below (see Subsection~\ref{assumptions}).

One of the main difficulties of the problem is to give a correct mathematical meaning to problem \eqref{pb1} (see Subsection~\ref{definitionws} below).

\subsection{Assumptions}\label{assumptions}

We will always  assume that 
\begin{equation}
\label{cond1}
N=1 \,\,\,\,\mbox{ and }\,\,\,\, L>0.
\end{equation}

As far as the data $(a, g, \phi)$ are concerned, we will assume that they satisfy 
\begin{equation}\label{a1}
a\in L^{\infty}(0,L),\,\, \exists \,\, \alpha,\beta,\,\, 0<\alpha<\beta, \ \alpha\leq a(x)\leq  \beta\, \,\, \text{ a.e.}\,\, x\in(0,L),
\end{equation}
\begin{equation}\label{g1}
g\in L^2(0,L), 
\end{equation}
\begin{equation}
\label{condphi1}
\begin{cases}
\phi:\re\mapsto\re\cup\{+\infty\},
\,\,\phi\,\, \mbox{ is continuous with values in } \re\cup\{+\infty\},\\
 \phi(s)<+\infty, \ \ \forall s\in\re,\,\, \ s\neq 0\,.
\end{cases}
\end{equation}

We are mainly interested in the case where $\phi$ is singular at $s=0$, namely when
\begin{equation} \label{p3}\phi(0)=+\infty\,,\ \ \ \end{equation}
and this will represent the originality and the difficulty of the problem.

Note however that we will also consider functions $\phi$ which do not satisfy \eqref{p3}, for instance when approximating a singular function $\phi$ which satisfies \eqref{condphi1} and \eqref{p3} by a sequence of functions $\phi_n$ which belong to $C^0(\re)$, and therefore satisfy \eqref{condphi1} but not \eqref{p3}.
\bigskip

\begin{remark}
\label{rem21}
When $\phi$ satisfies \eqref{condphi1}--\eqref{p3}, one has 
\begin{equation}
\label{rem0inf}
\phi(s)\to +\infty \quad \mbox{ as } \quad s\to 0,
\end{equation}
which implies that
\begin{equation}
\label{rem0inf2}
\phi(s)\geq 0 \quad \mbox{ for } s \mbox{ sufficiently small}.
\end{equation}

Note also that when $\phi$ satisfies \eqref{condphi1}, then $\phi$ satisfies 
\begin{equation}
\label{rem0inf3}
\phi\in C^0_b([-R,-\delta])\cap C^0_b([+\delta,+R]),\quad
\forall (\delta,R),\,\, 0<\delta<R<+\infty.
\end{equation}

Finally note that \eqref{condphi1} does not impose any behaviour of $\phi$ as $s$ tends to $+\infty$\break and $-\infty$.

\qed
\end{remark}

\begin{remark}
The {\bf model case} for the function $\phi$ is the case of the function $\phi_\gamma$ given by
\begin{equation}
\label{pM}
%\begin{cases}
\dys\phi_\gamma (s)=\frac{c}{|s|^{\gamma}} +\varphi (s),
\mbox{ with } c>0, \,\gamma>0,\, \varphi\in C^0(\re),
%\end{cases}
\end{equation}
which satisfies \eqref{condphi1}-\eqref{p3} for every $\gamma>0$. 

\qed
\end{remark}

\begin{remark}\label{rm-infty}
Assuming $ \phi(0)=-\infty$ in place of $ \phi(0)=+\infty$ in  \eqref{p3} is just a variant of problem \eqref{pb1}: indeed the problem 
\begin{equation} 
\begin{cases}
\dys -\frac{d}{dx}\left(a(x) \frac{d \hat{u}}{dx}\right)  = - \frac{d \hat{\phi} (\hat{u})}{dx} - \frac{d  \hat{g}}{dx} & \text{in}\;(0,L),\\
\hat{u}(0)=\hat{u}(L)=0\,, & 
\end{cases}
%\eqno{(2.3)^{\hat{}}}
\end{equation}
with 
\begin{equation}
\label{phihat}
\begin{cases}
\hat{\phi}:\re\rightarrow \re\cup\{-\infty\},\,\,\hat{\phi} \mbox{ is continuous with values in } \re\cup\{-\infty\},\\
\hat{\phi}(s)>-\infty,\,\, \forall s\in \re,\,s\not=0,
\end{cases}
\end{equation}
reduces to problem \eqref{pb1} with assumptions \eqref{condphi1}-\eqref{p3}  on $\phi$ by setting 
\begin{equation}
\label{datapb2}
u=-\hat{u}, \quad g=-\hat{g},\quad \phi(s)=-\hat{\phi}(-s), \quad \forall s\in\re. 
\end{equation}\qed
\end{remark}

\begin{remark}\label{importante1}

Observe that when $u$ is a solution of  \eqref{pb1}, then $u$ is still a solution of \eqref{pb1} when  $\phi$ is changed in $\phi +c_1$ and $g$ in $g+c_2$, where $c_1$ and $c_2$ are arbitrary constants. 

To avoid confusion, we then  emphasize the fact that in the whole of the present paper,   {\it the data $g$ and $\phi$ are  fixed once and for all; in other terms the data $g$ and $\phi$ are  not defined up to additive constants, but fixed}.

 \qed
\end{remark}
 
\begin{remark}\label{importante2}
When the data $(a, g, \phi)$ are given and when they satisfy the hypotheses \eqref{a1}-\eqref{condphi1} above, we claim that {\it it is always possible to assume without any loss of generality that $\phi$ also enjoys the following property} 
\begin{equation}\label{pp2}
\phi\in C^0_b(\re\backslash(-\delta, +\delta)),  \ \   \forall \ \delta>0\footnote{or, equivalently, 
$
\phi\in C^0_b(\re\backslash(-\delta, +\delta))  \,\, \mbox{ for some given } \,\, \delta>0
$.},
\end{equation}

\noindent (compare with \eqref{rem0inf3}), if one considers weak solutions of problem \eqref{pb1} in the sense of Definition~\ref{defin} below for the data $(a,g,\phi)$ and if $g$ satisfies for some $M>0$
\begin{equation}
\label{estimgvis}
\|g\|_{L^2(0,L)}\leq \frac{\alpha}{\sqrt{L}}M. 
\end{equation}
Note  that {\it condition} \eqref{estimgvis} {\it is not a restriction on} $g$, since $M$ is an arbitrary constant; on the contrary, this condition allows $g$ to vary into a given ball of $L^2(0,L)$ in the proof below.

Let us prove this claim.

 Consider data $(a,g,\phi)$ which satisfy \eqref{a1}-\eqref{condphi1} and \eqref{estimgvis}, and let $u$ be any weak solution of problem \eqref{pb1} in the sense of Definition~\ref{defin} below for these data  $(a,g,\phi)$, i.e. a function $u$ which satisfies 
\begin{equation}
\label{pbrem}
\begin{cases}\vspace{0.1cm}
\dys u\in \huzu, \,\phi(u)\in L^2(0,L),  \\ 
\dys -\frac{d}{dx}\left(a(x) \frac{d u}{dx}\right)  =\dys - \frac{d \phi (u)}{dx}- \frac{d g(x) }{dx} \ \  \text{in} \ \ \mathcal{D}'(0,L).
\end{cases}
\end{equation}
Then, in view of \eqref{propduibis} (see Proposition~\ref{28bis} below) and of condition \eqref{estimgvis} on $g$ and $M$, the function $u$ satisfies
\begin{equation}
\label{propduiter}
\left\|u\right\|_{L^{\infty}(0,L)}\leq \frac{\sqrt{L}}{\alpha} \|g\|_{L^2(0,L)}\leq M.
\end{equation}

Define the function $\hat{\phi}_M:\re\to\re\cup\{+\infty\}$ by the formula
\begin{equation}
\label{formulaphiR}
\hat{\phi}_M(s)=\begin{cases}
\dys \phi(-M) & \mbox{if} \,\, s\leq -M,\\
\dys \phi(s) &\mbox{if} \,\, -M\leq s \leq +M,\\
\phi(+M) & \mbox{if} \,\, s\geq +M.\\
\end{cases}
\end{equation}

Then the function $\hat{\phi}_M$ satisfies both \eqref{condphi1} and \eqref{pp2}.

On the other hand, in view of \eqref{propduiter} one has
\begin{equation}
\label{phicap}
\dys\hat{\phi}_M(u)=\phi(u) \in L^2(0,L) \,\,\,\, \mbox{ and} \,\,\,\,
\dys \frac{d \hat{\phi}_M(u)}{dx}=\frac{d \phi(u)}{dx}  \,\, \mbox{ in } \,\,\mathcal{D}'(0,L). 
\end{equation}

When $g$ satisfies both \eqref{g1} and \eqref{estimgvis}, this implies that, any weak solution $u$ of problem \eqref{pb1} in the sense of Definition~\ref{defin} below for the data $(a,g,\phi)$ also satisfies  
\begin{equation}
\label{pbremcap}
\begin{cases}\vspace{0.1cm}
\dys u\in \huzu, \hat{\phi}_M(u)\in L^2(0,L),  \\ 
\dys -\frac{d}{dx}\left(a(x) \frac{d u}{dx}\right)  =\dys - \frac{d \hat{\phi}_M(u)}{dx}- \frac{d g(x)}{dx}  \ \  \text{in} \ \ \mathcal{D}'(0,L),
\end{cases}
\end{equation}
or, in other terms, $u$ is a weak solution of problem \eqref{pb1} in the sense of Definition~\ref{defin} below for the data $(a,g,\hat{\phi}_M)$.

Conversely, when $g$ satisfies both  \eqref{g1} and \eqref{estimgvis}, any weak solution $u$ of problem \eqref{pb1} in the sense of Definition~\ref{defin} below for the data $(a,g,\hat{\phi}_M)$, or in other terms, any solution $u$ of \eqref{pbremcap}, is also a solution of \eqref{pbrem}, or in other terms a weak solution of problem \eqref{pb1}  in the sense of Definition~\ref{defin} below for the data $(a,g,\phi)$.

In brief, when $g$ satisfies both \eqref{g1} and \eqref{estimgvis} for some $M>0$, one can always replace the function $\phi$, which only satisfy \eqref{condphi1}, by the function $\hat{\phi}_M$, which satisfies both \eqref{condphi1}  and \eqref{pp2}, if one considers weak solutions of problem \eqref{pb1}  in the sense of Definition~\ref{defin} below. This proves the claim. 

 \qed
\end{remark}

\subsection{Definition of a weak solution of problem \eqref{pb1}}
 \label{definitionws}

 We introduce the following notion of solution:
\begin{defin}
\label{defin}
 Assume that \eqref{cond1} holds true, and that  the data $(a,g,\phi)$ satisfy hypotheses \eqref{a1}-\eqref{condphi1}. We will say that $u$ is a {\it weak solution} of problem \eqref{pb1} if $u$ satisfies 
\begin{equation}\label{form-def}
\begin{cases}\vspace{0.1cm}
\dys u\in \huzu, \phi(u)\in L^2(0,L),  \\ 
\dys -\frac{d}{dx}\left(a(x) \frac{d u}{dx}\right)  =\dys - \frac{d  \phi (u)}{dx}- \frac{d g(x) }{dx} \ \  \text{in} \ \ \mathcal{D}'(0,L).
\end{cases}
\end{equation} \qed
\end{defin}

\begin{remark}
The above Definition~\ref{defin} will be justified by the result presented below in Theorem~\ref{maintheo} (Alternative).

Let us emphasize that there are cases where problem \eqref{pb1} does not have any solution in the sense of Definition~\ref{defin}.

This is for example the case if the nonlinearity $\phi$ satisfies $\phi(0)=+\infty$ (hypothesis \eqref{p3}) and if the source term $g$ is bounded from below (see Theorem~\ref{36} below).

This is also the case if $g$ is arbitrary in $L^2(0,L)$ and if
$$
\int_0^{+\delta} \phi(t) \,dt =+\infty \,\, \mbox{ and }\,\,
\int_{-\delta}^0 \phi(t)\, dt  =+\infty
$$
(see Theorem~\ref{45} below).

For this reason we will state most of the result of the present paper assuming that there exists at least a solution of problem \eqref{pb1} in the sense of Definition~\ref{defin}. We will face the problem of existence of such a solution in Section~\ref{Sec8}.

\qed
\end{remark}

\begin{remark}\label{2.3}
Problem \eqref{form-def} has a precise mathematical meaning, in contrast with problem \eqref{pb1} which has no mathematical meaning, since in \eqref{pb1} the spaces to which $u$ and $\phi(u)$ have to belong to are not specified, and since the mathematical meanings of the two equations in \eqref{pb1} are not specified neither. 

\qed
\end{remark}

\begin{remark}
Since the first line in \eqref{form-def} asserts that $u\in H_0^1(\Omega)$ and $\phi(u)\in L^2(0,L)$ while $a$ and $g$ satisfy \eqref{a1}-\eqref{g1}, the second line of  \eqref{form-def} is equivalent to the {\it variational formulation}
\begin{equation}
\label{vf}
\int_0^L a(x)\frac{du}{dx}\frac{dz}{dx}= \int_0^L \phi(u)\frac{dz}{dx}+\int_0^L g(x)\frac{dz}{dx},\quad \forall z\in H_0^1(\Omega).
\end{equation}\qed
\end{remark}
\begin{proposition}
\label{28bis}
Assume that \eqref{cond1} holds true, and that the data $(a,g,\phi)$ satisfy\break  \eqref{a1}-\eqref{condphi1}. Then every possible weak solution $u$ of problem \eqref{pb1} in the sense of\break Definition~\ref{defin} satisfies 
\begin{equation}
\label{vf2bis}
\dys\int_0^L a(x)\frac{du}{dx}\frac{du}{dx}\, dx= \int_0^L g(x)\frac{du}{dx}\,dx.
\end{equation}

This energy equality in particular implies that

\begin{equation}
\label{propdubis}
\left\|\frac{du}{dx}\right\|_{L^2(0,L)}\leq \frac{1}{\alpha}\|g\|_{L^2(0,L)},
\end{equation}
which in turn implies that
\begin{equation}
\label{propduibis}
\left\|u\right\|_{L^{\infty}(0,L)}\leq \frac{\sqrt{L}}{\alpha} \|g\|_{L^2(0,L)}.
\end{equation}
\end{proposition}

\begin{remark}
\label{28ter}
Proposition~\ref{28bis} asserts that every possible solution of \eqref{form-def} satisfies the a priori estimates \eqref{propdubis}-\eqref{propduibis}, i.e. $\huzu$ and $L^\infty(0,L)$ bounds which depend only on $L$, on the coercivity constant $\alpha$ of $a$, and on $\|g\|_{L^2(0,L)}$, but not on the function $\phi$, which is only assumed to satisfy \eqref{condphi1}. 

\qed

The three results of Proposition~\ref{28bis} immediately follow  from \eqref{vf},  from the coercivity \eqref{a1},  from Morrey's embedding  \eqref{rem210} which is specific to the dimension one, and from Lemma~\ref{prop24ter} below. 
\end{remark}

\begin{lemma}
\label{prop24ter}
Assume that $\phi$ satisfies hypothesis \eqref{condphi1}, and let $z$ be such that  
\begin{equation}
\label{zh}
z\in \huzu \,\,\mbox{ with } \,\, \phi(z)\in L^2(0,L).
\end{equation}
Then one has:
\begin{equation}
\label{intphi}
\int_0^L \phi(z)\frac{dz}{dx}\, dx=0.
\end{equation}

\end{lemma}
\begin{proof}[\bf Proof]
For $n>0$, let  $T_n:\re\to\re$ be the truncation at height $n$ defined by \eqref{tk},
and let $\psi_n:\re\to\re$ be the function defined by 
$$
\psi_n(s)=\int_0^s T_n(\phi(t))\,dt.
$$

\noindent Since $T_n(\phi)\in C_b^0(\re)$, one has $\psi_n\in C^0(\re)$ with $\dys \frac{d\psi_n}{ds}\in C_b^0(\re)$, so that
$$
\forall z\in\huzu,\,\, \psi_n(z)\in\huzu \,\, \mbox{ with }\,\, \frac{d\psi_n(z)}{dx}=(T_n(\phi(z)))\frac{dz}{dx},
$$
and therefore
\begin{equation}
\label{proppsi}
\int_0^L T_n(\phi(z))\frac{dz}{dx}\, dx=\int_0^L \frac{d\psi_n(z)}{dx}\, dx=\psi_n(z(L))- \psi_n(z(0))=0-0=0.
\end{equation}

Since $|T_n(\phi(z(x)))|\leq |\phi(z(x))|$ a.e. $x\in[0,L]$, and since by hypothesis \eqref{zh} $\phi(z)$ belongs to $L^2(0,L)$, while 
$$
T_n(\phi(z(x))) \to \phi(z(x))\,\, \mbox{ a.e.  }\,\, x\in(0,L)\,\, \mbox{ as } \,\,n\to+\infty, 
$$
passing to the limit in \eqref{proppsi} thanks to Lebesgue's dominated convergence theorem implies that
\begin{equation}
\label{intphi3}
\int_0^L \phi(z)\frac{dz}{dx}\, dx=0,
\end{equation}
which proves \eqref{intphi}.

\end{proof}

The proof of the following proposition is straightforward (the last line of \eqref{form-def4}\break below is just  dividing the second line by $a(x)$, integrating  on $(0,L)$, and using that $u(0)=u(L)$).
\begin{proposition}[{{\bf Equivalence}}]
\label{remequiv}
Assume that \eqref{cond1} holds true, and that  the data $(a,g,\phi)$ satisfy  \eqref{a1}-\eqref{condphi1}.  Then $u$ is a weak solution of \eqref{pb1} in the sense of Definition~\ref{defin},  (i.e. $u$ satisfies \eqref{form-def}) if and only if $u$ satisfies 
\begin{equation}\label{form-def3}
\begin{cases}
\dys u\in \huzu,\,\,\, \phi(u)\in L^2(0,L),  \\ 
\dys\exists \,\, c\in\re,\,\, a(x)\frac{du}{dx}=\phi(u)+g+c\ \  \text{in} \ \ \mathcal{D}'(0,L)\,,
\end{cases}
\end{equation}
or, equivalently, if and only if  
\begin{equation}\label{form-def4}
\begin{cases}
\dys u\in \huzu,\,\,\, \phi(u)\in L^2(0,L),  \\ \vspace{0.1cm}
\dys a(x)\frac{du}{dx}=\phi(u)+g+c\ \  \text{in} \ \ \mathcal{D}'(0,L),\\
\dys\mbox{with  }\,\, c=\dys-\frac{\dys\int_0^L \frac{\phi(u)}{a(x)}\,dx+\dys\int_0^L \frac{g(x)}{a(x)}\,dx}{\dys\int_0^L \frac{1}{a(x)}\,dx}.
\end{cases}
\end{equation}

\end{proposition}

\begin{remark}\label{2.4}
Consider a function $\phi:\re\to\re$ which satisfies 
\begin{equation}
\label{phibis}
\phi\in C^0(\re).
\end{equation}
Then the function $\phi$ satisfies \eqref{condphi1} and the notion of weak solution in the sense of Definition~\ref{defin} of problem \eqref{pb1} for this function $\phi$ is defined.

 On the other hand, when $\phi$ satisfies \eqref{phibis}, it is standard to define a {\it classical weak solution of problem} of \eqref{pb1} for such a function $\phi$ as a function $u$ which satisfies
\begin{equation}\label{form-defbis}
\begin{cases}\vspace{0.1cm}
\dys u\in \huzu, \\ 
\dys -\frac{d}{dx}\left(a(x) \frac{d u}{dx}\right)  = \dys- \frac{d \phi (u)}{dx} - \frac{d g(x) }{dx} \ \  \text{in} \ \ \mathcal{D}'(0,L)\,,
\end{cases}
\end{equation} 
where $\phi(u)$ ``automatically" belongs to $L^2(0,L)$, since by Morrey's embedding (see \eqref{rem210}) one has $\huzu\subset L^{\infty} (0,L)$, which implies when $\phi\in C^0(\re)$  that
\begin{equation*}
%\label{e2.5}
 \phi(z)\in L^\infty(0,L)\subset L^2(0,L),\,\, \forall z\in \huzu.
\end{equation*}

When $\phi$ satisfies \eqref{phibis}, the definition of ``weak solution of problem \eqref{pb1} in the sense of Definition~\ref{defin}" coincides with the definition of ``classical weak solution of problem \eqref{pb1}" given by \eqref{form-defbis}. 

Definition~\ref{defin} is actually an extension of definition \eqref{form-defbis} of a ``classical weak" solution of the singular case where $\phi(0)=+\infty$.

 \qed
\end{remark}

\subsection{Existence of a weak solution when $\phi\in C^0(\re)$}

\begin{proposition}\label{propexis}
Assume that \eqref{cond1} holds true, and that  the data $(a,g,\phi)$ satisfy  \eqref{a1}-\eqref{condphi1}. Assume also that
\begin{equation}
\label{phibis2}
\phi\in C^0(\re).
\end{equation}

Then there exists at least one {\rm classical weak solution $u$ of problem} \eqref{pb1} in the sense of \eqref{form-defbis}, or equivalently a function $u$ which satisfies  
\begin{equation}
\label{vf2}
\begin{cases}
\dys u\in \huzu,  \\ 
\dys\int_0^L a(x)\frac{du}{dx}\frac{dz}{dx}\, dx= \int_0^L \phi(u)\frac{dz}{dx}\,dx+\int_0^L g(x)\frac{dz}{dx}\,dx,\quad \forall z\in H_0^1(0,L).
\end{cases}
\end{equation}

This {\rm classical weak solution} is also {\rm a weak solution of problem \eqref{pb1} in the sense of Definition~\ref{defin}.}

Moreover any solution $u$ of \eqref{vf2} satisfies

\begin{equation}
\label{propdu}
\left\|\frac{du}{dx}\right\|_{L^2(0,L)}\leq \frac{1}{\alpha} \|g\|_{L^2(0,L)}\,,
\end{equation}
which in turn implies that 
\begin{equation}
\label{propdui}
\left\|u\right\|_{L^{\infty}(0,L)}\leq \frac{\sqrt{L}}{\alpha} \|g\|_{L^2(0,L)}\,.
\end{equation}
\end{proposition}

\begin{proof}[\rm{\bf{Proof of Proposition~\ref{propexis}}}]

\noindent{\bf Step 1.} Further to \eqref{phibis2}, let us assume in this first  step that 
\begin{equation}
\label{phibis2bis}
\phi\in C^0_b(\re)=C^0(\re)\cap L^\infty(\re).
\end{equation}

For every $\overline{u}\in L^2(0,L)$, define $u$ as  the unique solution of the linear problem  
\begin{equation}\label{form-defbis2}
\begin{cases}\vspace{0.1cm}
\dys u\in \huzu, \\ 
\dys -\frac{d}{dx}\left(a(x) \frac{d u}{dx}\right)  =\dys - \frac{d \phi (\overline{u})}{dx} - \frac{d  g}{dx} \ \  \text{in} \ \ \mathcal{D}'(0,L).
\end{cases}
\end{equation} 
Using $u$ as test function in the variational formulation of \eqref{form-defbis2}, one has 
$$
\int_0^L a(x)\left|\frac{du}{dx}\right|^2 \, dx= \int_0^L \phi(\overline{u})\frac{du}{dx}\, dx +\int_0^L g(x)\frac{du}{dx}\, dx,
$$
which implies that
$$
\alpha\left\|\frac{du}{dx}\right\|_{L^2(0,L)}\leq \sqrt{L}\|\phi\|_{L^\infty(\re)}+  \|g\|_{L^2(0,L)}.
$$
Recalling Poincar\'e's inequality \eqref{Poincare}, 
we get that 
\begin{equation*}
\dys\|u\|_{L^2(0,L)}\dys\leq L\left\|\frac{du}{dx}\right\|_{L^2(0,L)} \leq\frac{L}{\alpha}\left(\sqrt{L}\|\phi\|_{L^\infty(\re)}+\|g\|_{L^2(0,L)}\right).
\end{equation*}

Then Leray-Schauder's fixed point theorem applied to the map $W$ defined by
$$
W:\overline{u}\in L^2(0,L)\to W(\overline{u})=u\in L^2(0,L)
$$
and to the ball $B$ of $L^2(0,L)$ defined by 
$$
B=\left\{z\in L^2(0,L):\, \|z\|_{L^2(0,L)}\leq \frac{L}{\alpha}\left(\sqrt{L}\|\phi\|_{L^\infty(\re)}+\|g\|_{L^2(0,L)}\right)\right\}
$$
implies, using also the Rellich-Kondrachov's compactness theorem, that $W$ has at least a fixed point in $B$. 

When $\phi$ satisfies \eqref{phibis2bis}, this proves that there exists a classical weak solution of problem \eqref{pb1}, i.e. a solution of \eqref{vf2} (or of \eqref{form-defbis}, which is equivalent to \eqref{vf2}). 

Using then $z=u$ as test function in \eqref{vf2} and Lemma~\ref{prop24ter} implies that 
$$
\int_0^L a(x)\frac{du}{dx}\frac{du}{dx}\, dx= \int_0^L g(x)\frac{du}{dx}\,dx,
$$
which immediately implies \eqref{propdu}, which in turns implies \eqref{propdui} using Morrey's inequality \eqref{rem210} (where the latest inequality is specific to the one-dimensional case).

Observe that when $\phi$ satisfies \eqref{phibis2bis} and not only \eqref{phibis2} in Proposition~\ref{propexis}, the existence of (at least) one classical weak solution of \eqref{vf2} and estimate \eqref{propdu}, as well as their proofs, continue to hold true in dimension $N >1$; in contrast estimate \eqref{propdui} is a result specific to the case $N=1$, since it follows from Morrey's inequality \eqref{rem210}.

\noindent{\bf Step 2.}  Let us now consider the case where the sole hypothesis \eqref{phibis2} holds true, i.e. the case where $\phi$ belongs to $C^0(\re)$ and not necessarily  to $C^0_b(\re)$.

In this case, fix $m$ which satisfies
\begin{equation}
\label{defm}
m\geq \frac{\sqrt{L}}{\alpha}\|g\|_{L^2(0,L)},
\end{equation}
and consider the function $\phi_m:\re\to\re$ defined by
$$
\phi_m(s)=\phi(T_m(s)), \quad \forall s\in\re,
$$
where $T_m$ is the truncation at height $m$ defined by \eqref{tk}. Then $\phi_m$ belongs to $C^0_b(\re)$, and Step 1 implies that there exists at least one solution $u_m$ of 
\begin{equation}
\label{vf2m}
\begin{cases}
\dys u_m\in \huzu,  \\ 
\dys\int_0^L a(x)\frac{du_m}{dx}\frac{dz}{dx}\, dx= \int_0^L \phi_m(u_m)\frac{dz}{dx}\,dx+\int_0^L g(x)\frac{dz}{dx}\,dx,\,\,\, \forall z\in H_0^1(0,L).
\end{cases}
\end{equation}

Moreover every solution $u_m$ of \eqref{vf2m} satisfies \eqref{propdu} and \eqref{propdui}, and therefore in view of \eqref{defm} 
$$
\left\|u_m\right\|_{L^{\infty}(0,L)}\leq \frac{\sqrt{L}}{\alpha} \|g\|_{L^2(0,L)}\leq m.
$$
This implies that
$$
T_m(u_m)=u_m, \,\,\mbox{ and }\,\, \phi_m(u_m)=\phi(T_m(u_m))=\phi(u_m),
$$
which in turn implies that $u_m$ is also a solution of \eqref{vf2} for the function $\phi$.

Proposition~\ref{propexis} is then proved in full generality.

\end{proof}

\subsection{Definition of a sequence of reasonable approximations of $\phi$}

 Let us conclude this section by introducing the following definition.
\begin{defin}
\label{reasonable}
Let $\phi$ be a function which satisfies \eqref{condphi1}. We will say that a sequence $\phi_n$ of functions is a {{\it reasonable sequence of approximations} (or simply  a {\it reasonable approximation}) of $\phi$} if the sequence $\phi_n$ satisfies
\begin{equation}
\label{phincond1}
\phi_n \mbox{ satisfies assumption \eqref{condphi1} for every given $n$},
\end{equation}
\begin{equation}
\begin{cases}
\label{phincond2}
\mbox{for every sequence $s_n\in\re$ and every $s\in\re$ such that $s_n\to s$ in $\re$,}\\
\mbox{then } \phi_n(s_n)\to\phi(s) \mbox{ in } \re\cup\{+\infty\}.
\end{cases}
\end{equation} \qed
\end{defin}

\begin{remark}[{{\bf Examples}}]\label{examples}
Let us list, for functions $\phi$ which satisfy \eqref{condphi1} (and possibly \eqref{p3}), five examples of reasonable approximations of $\phi$.

\bigskip
The {\bf first example} is the {\bf approximation by truncation}, which consists in taking, for every function $\phi$ which satisfies \eqref{condphi1}, the sequence of functions $\phi_n$ defined by
\begin{equation}
\label{ex1}
\phi_n(s)=T_n(\phi(s)),\quad \forall s\in\re,\,\, \forall n\in\mathbb{N},
\end{equation}
where $T_n$ is the truncation at height $n$ defined by  \eqref{tk}.

It is easy to prove that $\phi_n\in C^0_b(\re)$, which proves \eqref{phincond1}. It is also easy to prove \eqref{phincond2} when $s\not=0$, as well as when $\phi(0)<+\infty$. 

The case where  $\phi(0)=+\infty$ and where $s_n\to 0$ requires special attention. In this case one has indeed to prove that
$$
\mbox{if } \,\, s_n\to0, \,\, \mbox{ then } \,\, \phi_n(s_n)=T_n(\phi(s_n))\to\phi(0)= +\infty.
$$
This can be done using the facts that $T_n(r)\geq T_m(r)$ for every $n\geq m>0$ and for every $r> 0$.
Indeed, as $s_n\to 0$, one has $\phi(s_n)>0$ for every $n$ sufficiently large. Therefore for every $m>0$ and $n$ sufficiently large, one has
$$
\phi_n(s_n)=T_n (\phi(s_n))\geq T_m (\phi(s_n))\to T_m (+\infty)=m \quad \mbox{for every $m>0$ fixed}.
$$

This completes the proof of the fact that the sequence $\phi_n$ of approximations by truncation defined by \eqref{ex1} is a reasonable sequence of approximations of any function $\phi$ which satisfies \eqref{condphi1}.

Note that the approximations by truncation \eqref{ex1} satisfy 
\begin{equation}
\label{2.61}
\phi_n\in C^0_b(\re),\quad |\phi_n(s)|\leq |\phi(s)|,\quad |\phi_n(s)|\leq n, \quad \forall s\in\re,\,\, \forall n\in\mathbb{N}.
\end{equation}

\bigskip
The {\bf second example} is the {\bf homographic approximation},  which consists in\break taking, for every function $\phi$ which satisfies \eqref{condphi1}, the sequence of functions $\phi_n$ defined by
\begin{equation}
\label{ex2}
\begin{cases}\vspace{0.2cm}
\dys\phi_n(s)=\frac{\phi(s)}{1+\frac 1n |\phi(s)|}, \quad \forall s\in\re,\,\, s\not=0,\\
\dys \phi_n(0)=\begin{cases}
\dys\frac{\phi(0)}{1+\frac 1n |\phi(0)|}, &\quad \mbox{ when } \,\, \phi(0)<+\infty,\\
\dys n,& \quad \mbox{ when } \,\, \phi(0)=+\infty.
\end{cases}
\end{cases}
\end{equation}
It is easy to prove that $\phi_n\in C^0_b(\re)$, which implies \eqref{phincond1}: indeed, when  $\phi(0)=+\infty$, one has, for every fixed $n$,  
$$
\phi_n(s)\sim \frac{\phi(s)}{\frac 1n |\phi(s)|}=n,\quad \mbox{ if } s\to 0,\,\, s\not=0.
$$

Here again the only (small) difficulty in proving \eqref{phincond2} is the case where $\phi(0)=+\infty$ and where $s_n\to 0$. Since $\phi(s_n)>0$ for $n$ sufficiently large, this is done by considering first subsequences $\{n'\}\subset\{n\}$ for which 
$$
\frac {1}{n'} |\phi(s_{n'})|=\frac {1}{n'} \phi(s_{n'})\to c, \quad \mbox{with} \quad  0\leq c<+\infty,\,\,\mbox{ as } n'\to +\infty,  
$$
and then subsequences $\{n'\}\subset\{n\}$ for which 
$$
\frac {1}{n'} |\phi(s_{n'})|=\frac {1}{n'} \phi(s_{n'})\to +\infty, \quad \mbox{ as } n'\to +\infty;  
$$
in the latest case, one has either $\phi_{n'}(s_{n'})=n'$ if $s_{n'}=0$, or if $s_{n'}\not=0$,
$$
\phi_{n'}(s_{n'})=\frac{\phi(s_{n'})}{1+\frac {1}{n'} |\phi(s_{n'})|}\sim \frac{\phi(s_{n'})}{\frac{1}{n'} |\phi(s_{n'})|}= \frac{\phi(s_{n'})}{\frac {1}{n'} \phi(s_{n'})}=n'. 
$$

In any  cases  one has 
$$
\phi_{n'}(s_{n'})\rightarrow +\infty=\phi(0),\quad \mbox{ as } \,\, n'\to+\infty,
$$
which proves \eqref{phincond2}.

This completes the proof of the fact that the sequence $\phi_n$ of homographic approximations defined by \eqref{ex2} is a reasonable approximation of any function $\phi$ which satisfies \eqref{condphi1}.

Note that here again the homographic approximations satisfy \eqref{2.61}.

\bigskip
The {\bf third example}  is the {\bf trivial approximation of the function $\phi$ by itself}, which consists in taking, for every function $\phi$ which satisfies \eqref{condphi1} the sequence of approximations defined by
\begin{equation}
\label{ex3}
\phi_n(s)=\phi(s), \quad \forall s\in\re,\,\, \forall n\in\mathbb{N}.
\end{equation}

In this trivial example, it is clear that \eqref{phincond1}-\eqref{phincond2} hold true, and that the trivial sequence of approximations defined by  \eqref{ex3} is a reasonable sequence of approximations of any function $\phi$ which satisfies \eqref{condphi1}.

Note that in contrast with the first and second examples, one does not have the property $\phi_n\in C^0_b(\re)$ when $\phi\not\in C^0_b(\re)$.

\bigskip
The {\bf fourth example} consists essentially in {\bf approximating $\dys\frac{1}{|s|^\gamma}$ by} $\dys\frac{1}{|s|^{\gamma_n}}$, 
 in the specific case where  $\phi$ is the model example  \eqref{pM} given by 
\begin{equation}
\label{ex4}
%\begin{cases}
\dys\phi_{\gamma} (s)=\frac{c}{|s|^{\gamma}} +\varphi (s),
\mbox{ with } c>0, \,\gamma>0,\, \varphi\in C^0(\re).
%\end{cases}
\end{equation}
In this case one can approximate the function $\phi_{\gamma}$ by the sequence of functions $\phi_n$ given by 
\begin{equation}
\label{2.65}
%\begin{cases}
\dys\phi_n (s)=\frac{c_n}{|s|^{\gamma_n}} +\varphi_n (s),
\mbox{ with } c_n>0, \,\gamma_n>0,\, \varphi_n\in C^0(\re),
%\end{cases}
\end{equation}
where 
\begin{equation}
\label{2.66}
c_n\to c, \quad  \gamma_n\to\gamma,\quad \varphi_n\to\varphi \quad \mbox{ uniformly in } C^0([-R,+R]) \mbox{ for every fixed } R.
\end{equation}

Using in particular the fact that for every $\varepsilon$, with $0<\varepsilon<\min\{\gamma,c\}$, one has 
$$
\phi_n(s_n)=\frac{c_n}{|s_n|^{\gamma_n}}+\varphi_n(s_n)\geq \frac{c-\varepsilon}{|s_n|^{\gamma-\varepsilon}}+\varphi_n(s_n)\to +\infty=\phi(0), \quad \mbox{ as } \,\, s_n\to 0, \,\, s_n\not=0,
$$
allows one to prove that the sequence $\phi_n$ defined by \eqref{2.65}-\eqref{2.66} is a reasonable\break sequence of approximations of the function $\phi$ defined by \eqref{ex4}. Here again one does not have  $\phi_n\in C^0_b(\re)$. 

\bigskip
The {\bf fifth example} concerns the case where $\phi\in C^0 (\re)$ and consists in this case in the classical {\bf approximation by convolution}, namely 
$$
\phi_{\varepsilon}=\phi * \rho_{\varepsilon}\,,
$$
where $\rho_{\varepsilon}$ is a standard sequence of mollifiers, i.e. $\rho_{\varepsilon} =\varepsilon \rho(\varepsilon x)$
with $\rho\in C^{\infty}_{c}(\re)$, $\rho\geq 0$, such that $\int_\re \rho = 1$.  Then  $\phi_{\varepsilon}$ is a reasonable sequence of approximations of $\phi$ since $\phi_{\varepsilon}$ converges locally uniformly to $\phi$. Here one has $\phi_{\varepsilon} \in C^{\infty} (\re)$.

Moreover,  in the case where $\phi\in C^0 (\re)$ is constant at infinity, namely when there exists $R>0$ such that 
$$\phi(s)=
\begin{cases}
\phi(+R) & \text{if}\ s>+ R,\\
\phi(-R) & \text{if}\ s< -R,
\end{cases}
$$
then for every ${\varepsilon} $ the function $\phi_{\varepsilon}  =\phi*\rho_{\varepsilon} $ is Lipschitz continuous on $\re$.  

\qed
\end{remark}

To conclude this section, let us state and prove a characterization of a sequence of reasonable approximations defined by Definition~\ref{reasonable}.
\begin{proposition}\label{rem29}
Assume that $\phi$ satisfies hypotheses \eqref{condphi1}-\eqref{p3}. Then   Definition~\ref{reasonable} is equivalent to assert that the sequence $\phi_n$ satisfies \eqref{phincond1} as well the following two properties 
\begin{equation}
\begin{cases}
\label{phincond3}
\mbox{{\rm for every} $\eta$ {\rm and} $R$, } 0<\eta<R,\\
 \phi_n\to\phi \mbox{ {\rm uniformly in} $C^0([+\eta,+R])$ {\rm and in} $C^0([-R,-\eta])$},
\end{cases}
\end{equation}
\vspace{-0.1cm}
\begin{equation}
%\begin{cases}
\label{phincond4}
\dys\liminf_{n\to\infty}\left(\inf_{t\in[-\eta,+\eta]}\phi_n(t)\right)\to +\infty, \,\,\mbox{ \rm as } \eta\to 0^+.
%\end{cases}
\end{equation}
\end{proposition}

\begin{remark}
The meaning of the two properties \eqref{phincond3}-\eqref{phincond4} is that, {\it in some sense}, the sequence of functions $\phi_n:\re\to\re\cup\{+\infty\}$ (which are assumed to satisfy \eqref{phincond1}) {\it locally uniformly converges} to the function  $\phi:\re\to\re\cup\{+\infty\}$ which satisfies\break \eqref{condphi1}-\eqref{p3}. Note however that the local uniform convergence is usually  defined only for functions from $\re$ into $\re$, while here one has $\phi(0)=+\infty$.

 Property \eqref{phincond3} is indeed nothing but the classical local uniform convergence of the sequence $\phi_n$ to $\phi$ in $\re-\{0\}$, while property \eqref{phincond4} asserts that $\phi_n$ {\it uniformly converges to $+\infty$ around $s=0$}.

When $\phi(0)$ is finite, a variant of  the proof below shows that  it is equivalent for an approximation  $\phi_n$ to  satisfy \eqref{phincond2} or 
to converge locally  uniformly on the whole of $\re$, i.e. in $C^0([-R,+R])$ for every $R\in\re$. Note that when $\phi(0)$ is finite, then necessarily $\phi_n(0)$ is finite for $n$ sufficiently large (take $s_n=0$ in \eqref{phincond2}). 

\qed
\end{remark}

\begin{proof}[\rm{\bf{Proof of Proposition~\ref{rem29}}}]

\noindent{\bf Step 1.} Let us first prove that if $\phi$ satisfies \eqref{condphi1}-\eqref{p3}, and if the sequence $\phi_n$ satifies \eqref{phincond1}, \eqref{phincond3}-\eqref{phincond4}, then the sequence $\phi_n$ satisfies \eqref{phincond2}.

Consider indeed on the first hand a sequence $s_n$ which satisfies 
$$
s_n\to s \,\,\mbox{ as } \,\, n\to+\infty,\,\, s\not=0,
$$
and write
$$
\phi_n(s_n)-\phi(s)=(\phi_n(s_n)-\phi(s_n))+(\phi(s_n)-\phi(s)).
$$
Then using \eqref{phincond3} and \eqref{rem0inf3} proves \eqref{phincond2}.

Consider on the other hand a sequence $s_n$ which satisfies 
$$
s_n\to 0 \,\,\mbox{ as } \,\, n\to+\infty.
$$
Then for every $\eta>0$, one has $|s_n|\leq \eta$ for $n$ sufficiently large, and the inequality $\phi_n(s_n)\geq \inf_{t\in [-\eta,+\eta]} \phi_n(t)$ for $n$ sufficiently large implies that 
$$
\liminf_{n\to+\infty}\phi_n(s_n)\geq \liminf_{n\to+\infty}\inf_{t\in [-\eta,+\eta]} \phi_n(t),
$$
which combined with \eqref{phincond4} proves that 
$$
\phi_n(s_n)\to+\infty=\phi(0), \,\,\mbox{ as } \,\, n\to +\infty.
$$

This completes the proof of \eqref{phincond2}.

 \noindent{\bf Step 2.} Conversely let us prove that if $\phi$ satisfies \eqref{condphi1}-\eqref{p3}, and if the sequence $\phi_n$ satisfies \eqref{phincond1}-\eqref{phincond2}, then the sequence $\phi_n$ satisfies  \eqref{phincond3}-\eqref{phincond4}.  

{\bf As far as \eqref{phincond3} is concerned}, fix $\eta$ and $R$ with $0<\eta<R$ and choose any $s_n\in[\eta,R]$ such that  
\begin{equation}
\label{2.100}
|\phi_n(s_n)-\phi(s_n)|=\max_{t\in [\eta,R)}|\phi_n(t)-\phi(t)|= \sup_{t\in [\eta,R]}|\phi_n(t)-\phi(t)|= \|\phi_n-\phi\|_{C^0([\eta,R])};
\end{equation}
observe indeed that in \eqref{2.100} the supremum on $[\eta,R]$ is actually a maximum. From any subsequence denoted by $\{n'\}$ of $\{n\}$, extract  from ${n'}$ a subsequence $\{n''\}\subset\{n'\}$ such that the subsequence $s_n''\in[\eta,R]$ converges to some $s\in[\eta,R]$, and apply \eqref{phincond2} to $s_n''$; then 
\begin{equation*}
%\label{convphi}
\phi_{n''}(s_{n''})\to\phi(s), \,\, \mbox{ as } \,\, n''\to +\infty.
\end{equation*}
Since 
\begin{align*}
\|\phi_{n''}-\phi\|_{C^0([\eta,R])}=|\phi_{n''}(s_{n''})-\phi(s_{n''})|\leq  |\phi_{n''}(s_{n''})-\phi(s)|+|\phi(s)-\phi(s_{n''})|,
\end{align*}
using \eqref{phincond2} and \eqref{condphi1} implies that 
\begin{equation}
\label{2.101}
\phi_{n''}\to\phi  \,\, \mbox{ uniformly in  } \,\, C^0([\eta,R]), \,\, \mbox{ as } n''\to +\infty.
\end{equation}
The fact that the limit in \eqref{2.101} does not depend on the subsequence $n'$ implies that the convergence \eqref{2.101} takes place for the whole sequence $\{n\}=\mathbb{N}$.

A similar proof implies the similar result in $C^0([-R,-\eta])$, and \eqref{phincond3} is proved.

{\bf As far as \eqref{phincond4} is concerned}, fix $\eta>0$ and choose any $s_n\in [-\eta,+\eta]$ such that 
 \begin{equation}
\label{propphin}
\phi_n(s_n)=\min_{t\in [-\eta,+\eta]} \phi_n(t)= \inf_{t\in [-\eta,+\eta]} \phi_n(t);
\end{equation}
observe indeed that in \eqref{propphin} the infimum on $[-\eta,+\eta]$ is actually a minimum, and that $\phi_n(s_n)$ is finite for every $n$ since, for every $\overline{s}\in  [-\eta,+\eta]$, one has
$$
\phi_n(s_n)=\inf_{|t|\leq \eta} \phi_n(t)\leq  \phi_n(\overline{s}).
$$
Since 
$$
\phi_{n}(\overline{s})\to\phi(\overline{s}), \,\, \mbox{ as } \,\, n\to +\infty,
$$
this implies that for some constant $\overline{C}$, one has
\begin{equation}
\label{2.102}
\phi_{n}(s_n)\leq \overline{C}, \quad \forall n\in\mathbb{N},
\end{equation}
as well as
$$
\liminf_{n}\phi_{n}(s_{n})\leq \lim_{n}\phi_n(\overline{s})=\phi(\overline{s}),\quad \forall \overline{s}\in[-\eta,+\eta],
$$
so that
\begin{equation}
\label{2.105}
\liminf_{n}\phi_{n}(s_{n})\leq \inf_{s\in [-\eta,+\eta]}\phi(s).
\end{equation}

Let us now prove that for some constant $\underline{C}$, one has 
\begin{equation}
\label{2.106}
\phi_{n}(s_n)\geq \underline{C}, \quad \forall n\in\mathbb{N}.
\end{equation}
Indeed, if \eqref{2.106} does not hold true, there exists a subsequence $\{n'\}\subset\{n\}=\mathbb{N}$ such that 
\begin{equation}
\label{2.107}
\phi_{n'}(s_{n'})\to -\infty, \,\, \mbox{ as } \,\, n'\to +\infty.
\end{equation}
Extract from $\{n'\}$ a subsequence $\{n''\}\subset\{n'\}$ such that the subsequence $s_n''\in[-\eta,+\eta]$ converges to some $s\in[-\eta,+\eta]$, and apply \eqref{phincond2} to $s_n''$; then 
\begin{equation*}
%\label{convphi}
\phi_{n''}(s_{n''})\to\phi(s), \,\, \mbox{ as } \,\, n''\to +\infty,
\end{equation*}
in contradiction with \eqref{2.107}, since $\phi:\re\rightarrow\re\cup\{+\infty\}$.

We use the fact that from any sequence $\rho_n$ with $\underline{c}\leq \rho_n\leq \overline{c}$ for some $\underline{c},\overline{c}\in\re$, one can extract a subsequence $\{n'\}\subset\{n\}$ such that 
$$
\lim_{n'} \rho_{n'}=\liminf_{n}\rho_n.
$$

\noindent Using this result  with $\rho_n=\phi_n(s_n)$, which satisfies $\underline{C}\leq \phi_n(s_n)\leq \overline{C}$ in view of \eqref{2.106} and \eqref{2.102}, there exists a subsequence $\{n'\}\subset\{n\}$ such that 
$$
 \lim_{n'}\phi_{n'}(s_{n'})=\liminf_{n}\phi_{n}(s_{n}).
$$
Extract from $\{n'\}$ a subsequence $\{n''\}\subset\{n'\}$ such that the subsequence $s_n''\in[-\eta,+\eta]$ converges to some $s\in[-\eta,+\eta]$, and apply \eqref{phincond2} to $s_n''$; then
$$
\liminf_{n}\phi_{n}(s_{n})= \lim_{n''}\phi_{n''}(s_{n''})=\phi(s)\geq \inf_{t\in [-\eta,+\eta]}\phi(t).
$$

Combining this result with \eqref{2.105}, we have proved that
\begin{equation}
\label{2.110}
\forall \eta>0,\quad \liminf_{n}\left(\inf_{t\in [-\eta,+\eta]}\phi_{n}(t)\right)= \liminf_{n}\phi_{n}(s_{n})=\inf_{t\in [-\eta,+\eta]}\phi(t).
\end{equation}
Since in view of  \eqref{condphi1}-\eqref{p3} one has 
\begin{equation}
\label{2.111}
\inf_{t\in [-\eta,+\eta]} \phi(t)\to\phi(0)=+\infty, \,\,\mbox{ as } \eta\to 0^+,
\end{equation}
we have proved \eqref{phincond4}. 

This completes the proof of Proposition~\ref{rem29}.

\end{proof}

\section{Approximation of problem \eqref{pb1}, a priori estimates, \\and an alternative}\label{Sec3}

As we already said, in order to  (try to) prove the existence of a weak solution of problem \eqref{pb1} in the sense of Definition~\ref{defin}, one proceed as usual by approximation,  finding suitable  priori estimates, and finally passing to the limit.

\subsection{Approximation of problem \eqref{pb1} and the main difficulty}\label{subapprox}

We assume that  hypothesis \eqref{cond1} holds true (so that we are dealing with a one-dimensional problem), and that the data $(a,g,\phi)$ satisfy hypotheses \eqref{a1}-\eqref{condphi1} and \eqref{p3},
and we consider sequences $(a_n, g_n,\phi_n)$ of ``approximated data" which satisfy, for some $\beta> \alpha>0$, and some $c_0>0$ 
\begin{equation}
\label{sequence}
a_n\in L^\infty(0,L),\,\, \exists \,\, \alpha,\beta,\,\, 0<\alpha< \beta,\,\, \alpha\leq a_n(x)\leq  \beta,\,\, a_n(x)\to a(x)\,\, \mbox{ a.e. } x\in(0,L),
\end{equation}
\begin{equation}
\label{propgn}
  g_n\in L^2(0,L),\,\, \,\, \|g_n\|_{L^2(0,L)}\leq c_0,\,\, g_n\rightharpoonup g \,\, \mbox{ weakly in } \, L^2(0,L),
\end{equation}
\begin{equation}
\label{phin}
%\begin{cases}
\phi_n \,\, \mbox{ is a sequence of reasonable approximations of } \,\,  \phi,
%\end{cases}
\end{equation}
(recall Definition~\ref{reasonable} in Section~\ref{sec2}), as well as the further regularity assumption on $\phi_n$
\begin{equation}
\label{phinconv}
\phi_n\in C^0(\re) \,\,\mbox{ for every given } \,\, n.
\end{equation}

Proposition~\ref{propexis} then ensures that for every $n$ there exists at least one classical weak solution of problem \eqref{pb1} for every $(a_n,g_n,\phi_n)$, namely at least one function $u_n$ which satisfies (see \eqref{vf2bis}) the energy equality

\begin{equation}
\label{35bis}
\dys\int_0^L a_n(x)\frac{du_n}{dx}\frac{du_n}{dx}\, dx= \int_0^L g_n(x)\frac{du_n}{dx}\,dx,
\end{equation}
which implies
\begin{equation}
\label{intseq}
\begin{cases}\vspace{0.1cm}
u_n\in \huzu,\\
\dys \int_0^L a_n(x) \frac{d u_n}{dx} \frac{dz}{dx}\,dx=\int_0^L \phi_n(u_n)\frac{dz}{dx}\,dx +\int_0^L g_n \frac{dz}{dx}\, dx, \quad \forall z\in\huzu.
\end{cases}
\end{equation}

Moreover the function $u_n$ satisfies  (see \eqref{propdu})
\begin{equation}
\label{propdun}
\left\|\frac{du_n}{dx}\right\|_{L^2(0,L)}\leq \frac{1}{\alpha} \|g_n\|_{L^2(0,L)}\leq \frac{c_0}{\alpha}.
\end{equation}

One can therefore extract a subsequence, denoted by $n'$, and there exists some\break $u\in\huzu$ such that
\begin{equation}
\label{convweak}
u_{n'}\rightharpoonup u \,\,\mbox{ weakly in } \,\, \huzu \,\,\mbox{ and a.e. }\,\, x\in(0,L), 
\end{equation}
thanks to Rellich-Kondrashov's theorem. One easily passes to the limit in the first and last terms of \eqref{intseq} thanks to the strong convergence of $\dys a_n\frac{dz}{dx}$ and the weak convergence of $g_n$ in $L^2(0,L)$, which result from \eqref{sequence} and \eqref{propgn}, obtaining
\begin{equation}
\label{propintn}
\begin{cases}\vspace{0.1cm}
\dys\int_0^L a_{n'}(x) \frac{d u_{n'}}{dx} \frac{dz}{dx}\,dx\to \int_0^L a(x) \frac{d u}{dx} \frac{dz}{dx}\,dx,\\
\dys\int_0^L g_{n'} \frac{dz}{dx}\, dx\to \int_0^L g \frac{dz}{dx}\, dx, 
\end{cases}
\,\, \mbox{ as } \,\, {n'}\to+\infty.
\end{equation}

 As far as the second term of \eqref{intseq} is concerned, the almost everywhere convergence in $(0,L)$ of $u_{n'}$ stated in \eqref{convweak} and the fact that $\phi_n$ is a sequence of reasonable approximations of $\phi$ immediately imply (see \eqref{phincond2}) that 
\begin{equation}
\label{convphin}
\phi_{n'}(u_{n'}(x))\to \phi(u(x))\,\, \mbox{ a.e. }\,\, x\in(0,L),\,\, \mbox{ as } {n'}\to+\infty.
\end{equation}
Observe however that this almost everywhere convergence does not allow one to pass to the limit in the term 
$$
\int_0^L \phi_{n'}(u_{n'}) \frac{dz}{dx}\,dx,
$$
since the a.e. convergence of $\phi_{n'}(u_{n'})$ is not sufficient to imply the convergence of the integrals.

Observe that up to now, we could have obtained results similar to \eqref{intseq}-\eqref{convphin} in an $N$-dimensional setting.

\subsection{A new a priori estimate due to the one-dimensional setting}
\label{sub32}

We will now prove a new a priori estimate which is specific to the one-dimensional case, see assumption \eqref{cond1}.

\begin{lemma}\label{weakconvergencephin}
Assume that \eqref{cond1} holds true, and that the data $(a,g,\phi)$ and  $(a_n,g_n,\phi_n)$ satisfy hypotheses \eqref{a1}-\eqref{condphi1} and \eqref{p3}, and \eqref{sequence}-\eqref{phin} and \eqref{phinconv}. If a subsequence, denoted by $u_{n'}$, satisfies \eqref{intseq} and \eqref{convweak} for some $u\in\huzu$,  and if 
\begin{equation}
\label{unotzero}
u\not=0,
\end{equation}
then
\begin{equation}
\label{phinbound}
\phi_{n'}(u_{n'}) \,\, \mbox{ is bounded in }\,\, L^2(0,L), 
\end{equation}
and
\begin{equation}
\label{phincweak}
\phi_{n'}(u_{n'})\rightharpoonup  \phi(u) \,\, \mbox{ weakly in  }\,\, L^2(0,L).
\end{equation}
\end{lemma}
\begin{proof}[\bf Proof]

\noindent{\bf Step 1.} We will strongly use in the present proof the assumption that $N=1$ in two ways.  

First by using Morrey's embedding theorem (see \eqref{Morrey}) which asserts that,  when $N=1$, then $H^1_0 (0,L)\subset C^{0,\frac12}([0,L])$, so that,   in view of \eqref{convweak},  that 
 \begin{equation}
\label{convunif}
u_{n'}\to u \,\,\mbox{ uniformly in  }  C^0([0,L]), \,\,\mbox{ as } {n'}\to+\infty.
\end{equation}

And, second,  by using the characterization of a weak solution of problem \eqref{pb1} in the sense of Definition~\ref{defin} given in Proposition~\ref{remequiv}, see \eqref{form-def3}; applied to $u_n$, \eqref{form-def3} implies that 
 \begin{equation}\label{form-defsec}
\begin{cases}
\dys u_n\in \huzu,\,\, \phi_n(u_n)\in L^2(0,L), \\ 
 \dys \exists \,\, c_n\in\re, \,\, a_n(x)\frac{du_n}{dx}=\phi_n(u_n)+g_n(x)+c_n\ \  \text{in} \ \ \mathcal{D}'(0,L)\,;
\end{cases}
\end{equation}
observe that here $\phi_n\in C^0(\re)$,  and that $u_n\in\huzu\subset L^\infty(0,L)$,   so that $\phi_n(u_n)$ ``automatically"  belongs to $L^\infty(0,L)\subset L^2(0,L)$ for each $n$.

\noindent{\bf Step 2.} If  we assume that $u\not=0$ (hypothesis \eqref{unotzero}), there exists at least one\break $x_0\in(0,L)$, such that
 \begin{equation}
 \label{unotzero2}
 u(x_0)\not=0.
 \end{equation}
 Let us assume for a moment that 
  \begin{equation}
 \label{unotzero3}
 u(x_0)>0.
 \end{equation}
 (the proof will be similar in the case where $u(x_0)<0$).
 
 Since $u\in\huzu\subset C^0([0,L])$, \eqref{unotzero3} implies that there exists some $\delta>0$ and $\eta>0$ such that 
  \begin{equation}\label{deltaeta}
 0<x_0-\delta<x_0<x_0+\delta<L,
\mbox{ with } \,\, u(x)\geq \eta, \,\,\,\,\forall x\in[x_0-\delta,x_0+\delta],
\end{equation}
 and the uniform convergence \eqref{convunif} implies that
  \begin{equation}\label{etalarge}
%\begin{cases}
\mbox{for $n'$ sufficiently large, } \,\, 
\forall x\in[x_0-\delta,x_0+\delta] \,\,\mbox{ one has } \,\,  u_{n'}(x)\geq \frac{\eta}{2}.
%\end{cases}
\end{equation}
 
 On the other hand, in view of  the Morrey's inequality \eqref{rem210}, and of  \eqref{propdun} and \eqref{propgn},  
 we  have
 \begin{equation}
 \label{estimun}
 \|u_n\|_{L^\infty(0,L)}\leq \sqrt{L} \left\|\frac{du_n}{dx} \right\|_{L^2(0,L)} \leq \frac{\sqrt{L}}{\alpha} \|g_n\|_{L^2(0,L)}\leq \frac{\sqrt{L}}{\alpha}c_0,
 \end{equation}
therefore one has in view of \eqref{etalarge}
  \begin{equation}
 \label{nlarge}
 %\begin{cases}
\mbox{for $n'$ sufficiently large, } \,\,
\forall x\in[x_0-\delta,x_0+\delta] \quad \frac{\eta}{2}\leq u_{n'}(x)\leq \frac{\sqrt{L}}{\alpha}c_0.
%\end{cases}
 \end{equation}
 Since $\phi_n$, which is a reasonable approximation of $\phi$, satisfies
 \begin{equation}
\label{convphinunif}
\phi_n\to \phi\,\, \mbox{ uniformly in }\,\, C^0\left(\left[\frac{\eta}{2},\frac{\sqrt{L}}{\alpha}c_0\right]\right),
\end{equation}
see \eqref{phincond3}; we deduce that 
 \begin{equation}
 \label{nlargephin}
 %\begin{cases}
\mbox{for $n'$ sufficiently large, } \,\, \phi_{n'}\,\, \mbox{ is bounded in } \,\, L^{\infty}(x_0-\delta,x_0+\delta).
%\end{cases}
\end{equation}

From \eqref{form-defsec}, \eqref{sequence}, \eqref{propdun}, \eqref{propgn}, and \eqref{nlargephin}, we deduce that 
  \begin{equation}
 \label{boundedcn}
% \begin{cases}
\mbox{for $n'$ sufficiently large, } \,\, c_{n'}\,\, \mbox{ is bounded in } \,\, L^{2}(x_0-\delta,x_0+\delta),
%\end{cases}
\end{equation}
 which implies that
  \begin{equation}
 \label{boundedcninr}
% \begin{cases}
\mbox{for $n'$ sufficiently large, } \,\, c_{n'}\,\, \mbox{ is bounded in } \,\, \re.
%\end{cases}
\end{equation}

Turning back to \eqref{form-defsec}, the estimates \eqref{sequence}, \eqref{propdun}, \eqref{propgn}, and \eqref{boundedcninr} together imply that 
 \begin{equation}
 \label{phinboundedr}
 %\begin{cases}
\mbox{for $n'$ sufficiently large, } \,\, \phi_{n'}(u_{n'})\,\, \mbox{ is bounded in } \,\, L^2(0,L).
%\end{cases}
\end{equation}
We have proved that \eqref{phinbound} holds true.

The weak convergence \eqref{phincweak} of $\phi_{n'}(u_{n'})$ to $\phi(u)$ then results from \eqref{phinboundedr} and \eqref{convphin}, since a bounded sequence $z_n$ in $L^p(\Omega)$ which converges a.e. in $\Omega$ to some  $z$ also converges to $z$ weakly in $L^p(\Omega)$ when $1<p<+\infty$ (this results from Vitali's theorem since the sequence $z_n$ is equi-integrable in $L^1(\Omega)$, and therefore it converges strongly in $L^1(\Omega)$). 

Lemma~\ref{weakconvergencephin} is proved.
 
\end{proof}

\begin{remark}
\label{31bis}
As we said at  the beginning of its  Step~1, the proof of Lemma~\ref{weakconvergencephin} strongly uses the assumption $N=1$. 

 On the other hand, the proof given in the second step is very surprising, since it consists to transform the local estimate \eqref{nlargephin}, which is only valid in $(x_0-\delta,x_0+\delta)$, into the global estimate \eqref{phinboundedr}, which is valid in $(0,L)$. This passing from local to global is also specific to the dimension $N=1$.

\qed
\end{remark}

\begin{remark}
\label{31ter}

It is assumed in hypothesis \eqref{propgn} that $g_n$ converges weakly to g in $L^2(0,L)$. If this hypothesis is reinforced in 
\begin{equation}
\label{326bis}
g_n\rightarrow g \,\, \mbox{ strongly in } \, L^2(0,L),
\end{equation}
then the weak convergence \eqref{convweak} is reinforced in  
\begin{equation}
\label{326ter}
u_{n'}\rightarrow u \,\,\mbox{ strongly in } \,\, \huzu. 
\end{equation}

Indeed once a subsequence $n'$ has been extracted for which one has \eqref{convweak} for some $u\in\huzu$, one has 
\begin{equation}
\label{326quarto}
\int_0^L g_{n'} \frac{du_{n'}}{dx}\, dx\to \int_0^L g \frac{du}{dx}\, dx, \,\, \mbox{ as } \,\, {n'}\to+\infty.
\end{equation}
when the strong convergence \eqref{326bis} holds true. Then either $u\equiv 0$, in which case \eqref{35bis} implies that
$$
u_{n'}\rightarrow 0 \,\,\mbox{ strongly in } \,\, \huzu, 
$$
or $u\not\equiv 0$, in which case, in view of Theorem \ref{maintheo} (Alternative) below, $u$ is a weak solution of problem \eqref{pb1} in the sense of Definition~\ref{defin}, which therefore satisfies the energy equality \eqref{vf2bis} in view of Proposition~\ref{28bis}. Passing to the limit in \eqref{35bis} and using \eqref{326quarto}  proves that
$$
\int_0^L a_{n'}(x) \frac{d u_{n'}}{dx} \frac{du_{n'}}{dx}\,dx\to \int_0^L a(x) \frac{d u}{dx} \frac{du}{dx}\,dx, \,\, \mbox{ as } \,\, {n'}\to+\infty.
$$
which with the weak convergence \eqref{convweak} implies the strong convergence \eqref{326ter} by passing to the limit in 
\begin{equation*}
\begin{split}
\dys \alpha \int_0^L \left| \frac{d u_{n'}}{dx}-\frac{d u}{dx}\right|^2 dx&\leq \int_0^L a_{n'}(x)\left( \frac{d u_{n'}}{dx}-\frac{d u}{dx}\right)\left( \frac{d u_{n'}}{dx}-\frac{d u}{dx}\right)dx=\\
\dys&= \int_0^L a_{n'}(x)\frac{d u_{n'}}{dx}\frac{d u_{n'}}{dx} \, dx - 2 \int_0^L a_{n'}(x)\frac{d u}{dx}\frac{d u_{n'}}{dx} \, dx+\\& + \int_0^L a_{n'}(x)\frac{d u}{dx}\frac{d u}{dx} \, dx. 
\end{split}
\end{equation*}

\qed
\end{remark}

\subsection{An alternative}

From the results obtained in Subsection~\ref{subapprox}, and from Lemma~\ref{weakconvergencephin} of Subsection~\ref{sub32}, we deduce that we are in front of {\it an alternative}:
\begin{theorem}[{{\bf Alternative}}]\label{maintheo}
Assume that hypothesis \eqref{cond1} holds true, and that the data $(a, g, \phi)$  satisfy hypotheses \eqref{a1}-\eqref{condphi1} and \eqref{p3}. Consider approximations $(a_n, g_n, \phi_n)$ which satisfy \eqref{sequence}-\eqref{phin} and \eqref{phinconv}. 

Then for every $n\in \mathbb{N}$  there exists at least one function $u_n$ which satisfies \eqref{intseq} and \eqref{propdun}. If one extract a subsequence, denoted by $u_{n'}$, such that \eqref{convweak} holds for some $u\in \huzu$, then one has the alternative: 

\begin{itemize}
\item either $u\equiv 0$,
\item or $u$ is a weak solution of problem \eqref{pb1} in the sense of Definition~\ref{defin}. 
\end{itemize}

\end{theorem}

\begin{remark}
Let us emphasize that in the whole of the present section, and in particular in Theorem~\ref{maintheo}, we have assumed hypothesis \eqref{p3}, namely $\phi(0)=+\infty$. If we do not make this hypothesis, but $\phi(0)<+\infty$, we have $\phi\in C^0(\re)$ and we are in the hypotheses of Proposition~\ref{propexis} with reasonable approximations $\phi_n$ which converge uniformly on $C^0([-R,+R])$ for every $R<+\infty$. In this classical   setting,  $u\equiv0$ can be  a solution\footnote{As a side remark, note that when $\phi\in C^0(\re)$ (or in other terms when $\phi(0)<+\infty$), $u\equiv 0$ is a classical weak solution of problem \eqref{pb1} if and only if $g$ is constant (see \eqref{form-defbis}).}, and there is no alternative: all the converging subsequences $u_{n'}$ converge to a classical weak solution, see the proof of Proposition~\ref{propexis}.

In Theorem~\ref{maintheo} the alternative is indeed due to the fact that $\phi$ is singular in $s=0$.

\qed
\end{remark}

Let us complete Theorem~\ref{maintheo} by a result which characterizes the behaviour of the constant $c_n$ which appears in \eqref{form-defsec}, and also the behaviour of $\phi_n(u_n)$.

To this aim  observe that, since $c_n\in\re$ for any given $n$, but without any bound on $|c_n|$, we are at liberty to extract from $n'$ a further subsequence denoted by $n''$ such that
{\small\begin{equation}
\label{3.100}
\exists \,\tilde{c}\in \re\cup\{+\infty\}\cup\{-\infty\}\,\mbox{ such that } \, c_{n''}\to \tilde{c}\,\, \mbox{ in } \,\re\cup\{+\infty\}\cup\{-\infty\}, \, \mbox{ as } n''\to+\infty. 
\end{equation}}

We then have the following result, which describes the links between the possible limits of $u_n$, $\phi_n(u_n)$, and $c_n$.
\begin{proposition} 
\label{prop35}
Assume that hypothesis \eqref{cond1} holds true, and that the data $(a, g, \phi)$  satisfy hypotheses \eqref{a1}-\eqref{condphi1} and \eqref{p3}. Consider approximations $(a_n, g_n, \phi_n)$ which satisfy \eqref{sequence}-\eqref{phin} and \eqref{phinconv}.  For $\tilde{c}$ and $n''$ defined by \eqref{3.100}, we have the following equivalences:
\begin{equation}
\label{3.101}
\begin{cases}
\dys u=0\Leftrightarrow \tilde{c}=-\infty\Leftrightarrow\\
\dys \Leftrightarrow   \forall M\in \re,\,\, \phi_{n''}(u_{n''}(x))\geq M, \,\, \forall x\in(0,L),\,\, \mbox{ for $n''$ sufficiently large} \dys \Leftrightarrow\\
\dys \Leftrightarrow \phi_{n''}(u_{n''}(x))\to +\infty \,\,\mbox{ uniformly in  }  [0,L], \,\,\mbox{ as } {n''}\to+\infty.
\end{cases}
\end{equation}
\begin{equation}
\label{3.102}
\begin{cases}
\dys u\not=0\Leftrightarrow -\infty<\tilde{c}<+\infty\Leftrightarrow\\
\dys \Leftrightarrow    \phi_{n''}(u_{n''})\rightharpoonup\phi(u) \,\, \mbox{ weakly in } L^2(0,L), \,\,\mbox{ as } \,\, n''\to +\infty.
\end{cases}
\end{equation}
\end{proposition}

\begin{remark}
Since the limit $u$ of a subsequence $u_{n'}$ can only be equal to $u=0$, or to $u\not=0$, taking into account the  equivalences in \eqref{3.101} and in \eqref{3.102}, one sees that the limit $\tilde{c}$ of a subsequence $c_{n''}$ can never be equal to $+\infty$ (but  only  be either   finite or equal to $-\infty$), and that the limit $\phi(u)$ of a subsequence $\phi_{n''}(u_{n''})$ can only be equal to $+\infty$, or to $\phi(u)$ for $u$ a weak solution of problem \eqref{pb1} in the sense of Definition~\ref{defin}.

\qed
\end{remark}

\begin{proof}[\rm{\bf{Proof of Proposition~\ref{prop35}}}]

\noindent{\bf Step 1: The case $u\not =0$.} In this case, we have proved  in Step 2 of the proof of Lemma~\ref{weakconvergencephin} that (see \eqref{boundedcninr}-\eqref{phinboundedr})
$$
u\not=0\Rightarrow c_{n''}\,\, \mbox{ is bounded in} \,\, \re\Rightarrow \phi_{n''}(u_{n''})\,\, \mbox{ is bounded in } \,\, L^2(0,L),
$$
which implies that (see the last paragraph of the proof of Lemma~\ref{weakconvergencephin})
\begin{equation}
\label{3.103}
\begin{cases}
\dys u\not=0\Rightarrow -\infty<\tilde{c}<+\infty\Rightarrow\\
\dys \Rightarrow    \phi_{n''}(u_{n''})\rightharpoonup\phi(u) \,\, \mbox{ weakly in } L^2(0,L), \,\,\mbox{ as } \,\, n''\to +\infty.
\end{cases}
\end{equation}

\noindent{\bf Step 2: The case $u= 0$.} In view of \eqref{convunif} $u_{n''}$ converges to $0$ uniformly in $C^0([0,L])$, and therefore 
$$
\forall \eta>0,\,\,\,\,|u_{n''}(x)|\leq \eta,\,\,\,\, \forall x\in[0,L]  \,\,\,\,   \mbox{ for $n''$ sufficiently large, } \,\, 
$$
which implies that
\begin{equation}
\label{3.106}
\forall \eta>0,\,\,\,\, \phi_{n''}(u_{n''}(x))\geq \inf_{t\in[ -\eta,+\eta]} \phi_{n''}(t) \,\,\,\, \forall x\in[0,L]  \,\,   \mbox{ for $n''$ sufficiently large}.
\end{equation}

On the other hand, in the second part  of Step 2 of the proof of Proposition~\ref{rem29}, we have proved, see \eqref{2.110}, that 
\begin{equation*}
\forall \eta>0,\,\,\,\, \liminf_{n''}\left\{\inf_{t\in[ -\eta,+\eta]} \phi_{n''}(t)\right\}= \inf_{t\in[ -\eta,+\eta]} \phi(t), 
\end{equation*}
which implies that for every $\dys k<\inf_{t\in[ -\eta,+\eta]}\phi(t)$, one has
$$
 \inf_{t\in[ -\eta,+\eta]} \phi_{n''}(t)\geq k, \,\,\,\,   \mbox{ for $n''$ sufficiently large}.
$$
Since it results from \eqref{condphi1}-\eqref{p3} that
$$
 \inf_{t\in[ -\eta,+\eta]} \phi(t)\to \phi(0)=+\infty, \quad \mbox{ as } \eta\to 0,
$$
we have proved that 
\begin{equation}
\label{3.110}
 \forall M\in \re,\,\,  \inf_{t\in[ -\eta,+\eta]}\phi_{n''}(t)\geq M, \,\, \mbox{ for $n''$ sufficiently large}.
\end{equation}

Combining \eqref{3.106} and \eqref{3.110} we have proved that
\begin{equation}
\label{3.111}
\dys u=0\Rightarrow   \forall M\in \re,\,\, \phi_{n''}(u_{n''}(x))\geq M, \,\, \forall x\in(0,L),\,\, \mbox{ for $n''$ sufficiently large}.
\end{equation}

On the other hand, in view of \eqref{form-defsec} one has 
$$
c_{n''}+\phi_{n''}(u_{n''})=a_{n''}(x)\frac{du_{n''}}{dx}-g_{n''}(x) \quad \mbox{ in } L^2(0,L),
$$
which combined with \eqref{3.111} implies that 
$$
\forall M\in \re,\quad c_{n''}+M\leq a_{n''}(x)\frac{du_{n''}}{dx}-g_{n''}(x) \quad \mbox{ in } L^2(0,L) ,\,\, \mbox{ for $n''$ sufficiently large}.
$$
Since the right-hand side of this inequality is bounded in $L^2(0,L)$ in view of \eqref{sequence}-\eqref{propgn} and \eqref{propdun}, integrating on $(0,L)$ and dividing by $L$ implies that there exists a constant $c_0<+\infty$ such that
$$
\forall M\in \re,\quad c_{n''}+M\leq c_0,\,\, \mbox{ for $n''$ sufficiently large},
$$
which implies that $c_{n''}\rightharpoonup-\infty$ as $n''\to +\infty$, or in other terms that $\tilde{c}=-\infty$.

We have proved that 
\begin{equation}
\label{3.112}
\begin{cases}
\dys u=0\Rightarrow \\\dys \Rightarrow  \forall M\in \re,\,\, \phi_{n''}(u_{n''}(x))\geq M, \,\, \forall x\in(0,L),\,\, \mbox{ for $n''$ sufficiently large}\Rightarrow\\
\Rightarrow \tilde{c}=-\infty.
\end{cases}
\end{equation}

\noindent{\bf Step 3: Proof of the two equivalences \eqref{3.101} and \eqref{3.102}.} We will deduce \eqref{3.101} and \eqref{3.102} from the two results \eqref{3.103} and \eqref{3.112}, and from the fact that when a subsequence $u_{n''}$ converges weakly to $u$ in $H_0^1(0,L)$ (see \eqref{convweak}), then one has the dicotomy ``either $u\not =0$ or $u=0$".

Indeed, when considering $\tilde{c}$, one deduces from \eqref{3.101} and \eqref{3.102} and from  the dicotomy ``either $u\not =0$ or $u=0$", that 
$$
\mbox{either $\tilde{c}$ is finite or $\tilde{c}=-\infty$,}
$$
and that one can not have $\tilde{c}=+\infty$.

Consider first the case when $\tilde{c}$ is finite; then use the dicotomy ``either $u\not=0$ or $u=0$": if $u=0$, then by \eqref{3.112} $\tilde{c}=-\infty$, which is not the case; therefore 
$$
\tilde{c} \mbox{ finite }\Rightarrow u\not=0.
$$
Consider then the case where $\tilde{c}=-\infty$; then use the dicotomy ``either $u\not=0$ or $u=0$": if $u\not=0$, then by \eqref{3.103} $\tilde{c}$ is finite, which is not the case; therefore 
$$
 \tilde{c}=-\infty \Rightarrow u=0.
$$

As far as $\phi_{n''}(u_{n''})$ is concerned, one deduces from \eqref{3.101} and \eqref{3.102} and from the dicotomy ``either $u\not =0$ or $u=0$" that, as $n''\to+\infty$,
$$
\mbox{either } \phi_{n''}(u_{n''})\rightharpoonup  \phi(u) \,\, \mbox{ weakly in  }\,\, L^2(0,L) \mbox{ or }  \phi_{n''}(u_{n''})\to+\infty \mbox{ uniformly in } [0,L].
$$

A proof similar to the proof made just above for $\tilde{c}$ leads to 
$$
\phi_{n''}(u_{n''})\rightharpoonup  \phi(u) \,\, \mbox{ weakly in  }\,\, L^2(0,L), \mbox{ as } n''\to+\infty\Rightarrow u\not=0,
$$
and to
$$
 \phi_{n''}(u_{n''}(x))\to+\infty \mbox{ uniformly in } [0,L],  \mbox{ as } n''\to+\infty\Rightarrow u=0.
$$

This completes the proof of the two equivalences \eqref{3.101} and \eqref{3.102} and of Proposition~\ref{prop35}.

\end{proof}
\bigskip
At the end of this section, one could think that, except maybe in some very special cases, every limit $u$ of weak solutions of approximations of problem \eqref{pb1} which satisfy \eqref{sequence}-\eqref{phin} and \eqref{phinconv} is always a weak solution of problem \eqref{pb1} in the sense of Definition~\ref{defin}.

We will see in Section \ref{sec4} below that this is not the case, and that for a large class of functions $g\in L^2(0,L)$ (see Theorem 4.1), and for a large class of functions $\phi$ (see Theorem \ref{45}), every limit  of approximations is $u\equiv0$. This is unexpected.

We will also see in Section~\ref{Sec8} below that for an another large class of functions $\phi$ and for another large class of functions $g$, there exists at least one weak solution of problem \eqref{pb1} in the sense of Definition~\ref{defin}. This will be also unexpected.

\section{Non-existence results}\label{Sec4}\label{sec4}

\indent In this section we give two results of non-existence of a  weak solution of problem \eqref{pb1} in the sense of Definition  \ref{defin}.

Our first  non-existence result states, in particular, that there is  no weak solution of problem \eqref{pb1} in the sense of Definition  \ref{defin} when $g\in L^{\infty}(0,L)$. This result is obtained independently of the nonlinearity $\phi$, provided $\phi(0)=+\infty$, i.e  \eqref{p3} holds true.

\begin{theorem}[{{\bf Non-existence when $g$ is bounded from below}}]\label{36}
Assume that hypothesis \eqref{cond1} holds true, and that the data $(a, g, \phi)$  satisfy hypotheses \eqref{a1}-\eqref{condphi1} and \eqref{p3}. Assume moreover that exists  $M>0$ such that 
\begin{equation} \label{glim} g(x)\geq -M\,\ \ \text{for a.e.}\ x\in (0,L).\end{equation} 
\indent  Then it does not exist any  weak solution of problem \eqref{pb1} in the sense of Definition~\ref{defin}.  
\end{theorem}

\begin{remark}
Observe that Theorem \ref{36} implies that if, for a given nonlinearity $\phi$, $\hat{u}$ is a weak solution of problem \eqref{pb1} in the sense of Definition~\ref{defin} corresponding to a source term $g$ (we will see in Section \ref{Sec8} below that there exist many  data $(a,g,\phi)$ for which there exist weak solutions of problem \eqref{pb1} in the sense of Definition~\ref{defin}), one can not hope to approximate $\hat{u}$ by approximating $\hat{g}$ by any sequence $\hat{g}_n$ which approximate $\hat{g}$ (weakly or strongly) in $L^2(0,L)$: in view of Theorem \ref{36},  it is indeed sufficient to approximate $\hat{g}$ by a sequence $\hat{g}_n\in L^{\infty}(0,L)$ which converges to $g$ (weakly or even strongly) in $L^2 (0,L)$, since for those $\hat{g}_n$ there is no weak solution $u_n$ of problem \eqref{pb1} in the sense of Definition  \ref{defin} with the source term $g_n$. 

\end{remark}
\qed

The proof of Theorem \ref{36} will use the following  Lemma:
\begin{lemma}[{\bf{The forbidden region}}\rm]\label{lemma36}
Assume that hypothesis \eqref{cond1} holds true, and consider data $(a, l, \phi)$ which satisfy hypotheses \eqref{a1}-\eqref{condphi1} and \eqref{p3}. Let $w$ which satisfies 
\begin{equation}\label{formLemma}
\begin{cases}
\dys w\in H^1(0,L)\,,\,\,\, \phi(w)\in L^2(0,L),  \\ 
\dys a(x)\frac{dw}{dx}=\phi(w)+l\ \  \text{in} \ \ \mathcal{D}'(0,L)\,.
\end{cases}
\end{equation}
Let $A,B$ and $x_0$ be such that 
\begin{equation}\label{abx}
0\leq A<B\leq L,\,\  x_0\in [A,B], \,\  w(x_0)=0\,.  
\end{equation}
If $l$ satisfies 
\begin{equation} \label{gsat} \exists \, M>0, \, \, \ l(x)\geq -M\,\ \ \text{for a.e.}\ x\in [A,B]\,,\end{equation}
then one has 
\begin{equation}\label{FC}
\begin{cases}
\dys {\forall \ k>0}, \ \ \exists \ \delta>0 \ \ \text{such that}\\ 
\dys\frac{dw}{dx}\geq k, \ \text{ a.e.}\ x\in [x_0-\delta, x_0+\delta]\cap[A,B], \\
\dys w(x)\geq k(x-x_0), \ \forall \ x\in [x_0, x_0+\delta]\cap[A,B], \\
\dys w(x)\leq k(x-x_0), \ \forall\ x \in [x_0-\delta, x_0]\cap[A,B]. 
\end{cases}
\end{equation}
\end{lemma}

Observe that formula   \eqref{FC} implies that the graph of the function $w$ can not enter in the {\it forbidden region} colored  red in  Figure \ref{conef},  when $x_0$,  which is a  zero of $w$,  can be either  an interior point  of $[A,B]$ or   an extremity of $[A,B]$, namely  $x_0=A$ or $x_0=B$. 
\bk 

\begin{figure}[htbp]\centering
\includegraphics[width=1.4in]{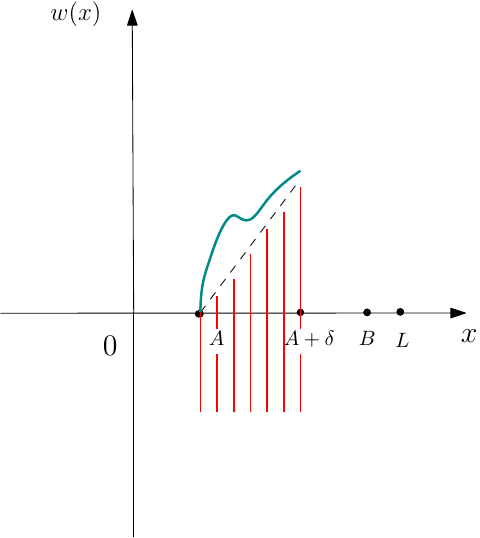}\ \ \includegraphics[width=1.9in]{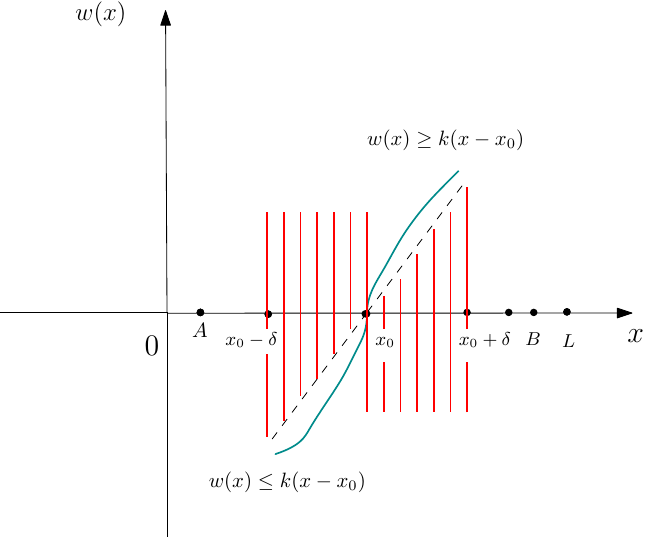}\ \ \includegraphics[width=1.3in]{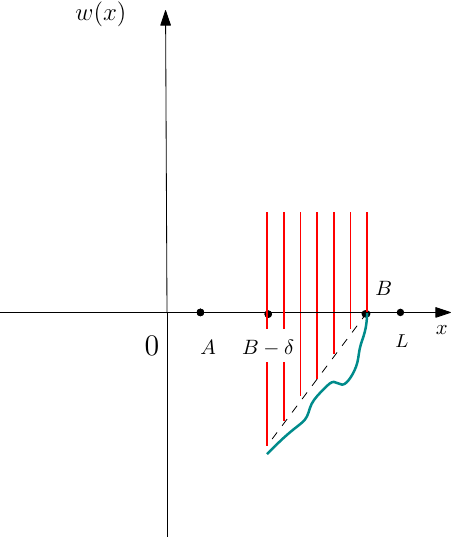}
\caption{Visualizing the statement of Lemma \ref{lemma36}: \newline
the extreme case $x_0=A$, the case $x_0\in (A,B)$   and the extreme case $x_0=B$.}\label{conef}
\end{figure}
\begin{proof}[\rm{\bf{Proof of Lemma \ref{lemma36}}}]
The main idea of the proof is that  the facts that  $\phi(w(x_0))=+\infty$ and $l\geq -M$ and the equation of the second line of \eqref{formLemma},  formally imply that $\dys\frac{d w }{dx}(x_0)=+\infty$, and therefore that $\dys\frac{d w }{dx}(x)$ is very large when $x$ is close to  $x_0$. Let us write this idea in a correct mathematical form. 

 Because $N=1$ one has $H^1 (A,B)\subset C^{0}([A,B])$ by Morrey's theorem (see \eqref{Morrey}). Since $w(x_0)=0$,  
then for every fixed $\vare>0$ there exists $\delta>0$, $\delta=\delta(\varepsilon)$,  such that
$$
\forall\, x\in\mathcal{V}_\delta=[x_0-\delta,x_0+\delta]\cap[A,B], \,\,\,\,|w(x)|\leq \vare.
$$
Then, 
 $$
\forall\,  x\in\mathcal{V}_\delta\,\,\,\,|\phi(w(x))|\geq \inf_{t\in[-\vare,+\vare]}\phi(t),
$$
so that in view of \eqref{gsat} and \eqref{formLemma}, one has 
\begin{equation}\label{nobeta}
a(x)\frac{dw}{dx}\geq \inf_{t\in[-\vare,+\vare]}\phi(t) - M, \mbox{for a.e. }\, x\in\mathcal{V}_\delta.
\end{equation}

Since $\dys \inf_{t\in[-\vare,+\vare]}\phi(t)\to +\infty $ as $\vare\to 0$,  the left hand side of \eqref{nobeta} is non-negative for $\vare$ sufficiently small, so that $\dys\frac{dw}{dx}\geq 0$ and, by \eqref{a1}, for $\vare$ sufficiently small, one has 

$$
 \beta\frac{dw}{dx}\geq  a(x)\frac{dw}{dx}\geq \inf_{t\in[-\vare,+\vare]}\phi(t) - M, \,\,\,\,\mbox{for a.e. }\, x\in\mathcal{V}_\delta .
$$

Dividing by $\beta$ and choosing $\vare$ sufficiently  small, this implies that, for every $k>0$, one has for some $\delta>0$  
\begin{equation}\label{beta}
\frac{dw}{dx}\geq {k}\ \ \text{a.e. in}\ \  \mathcal{V}_\delta\,,
\end{equation}
which immediately implies, using also  $w(x_0)=0$, that also the two latest lines of \eqref{FC} hold true. Lemma \ref{lemma36} is proved.

\end{proof}

\begin{proof}[\rm{\bf{Proof of Theorem \ref{36}}}]

Assume by contradiction that there exists some $u$ which is a weak  solution of problem \eqref{pb1} in the sense of  Definition \ref{defin}.  Then,  by Proposition \ref{remequiv}, $u$ satisfies
$$\begin{cases}
u\in H^1(0,L), \ \ \phi(u)\in L^2(0,L),\\
u(0)=u(L)=0,\\
\dys a (x) \frac{du}{dx}  = \phi (u) + g+ c\,,
\end{cases}$$
where $c$ is given in the last line of \eqref{form-def4}.

\indent Apply Lemma~\ref{lemma36} with $w=u$, $l=g+c$, $A=0$, $B=L$,  and $x_0=0$. Fixing any $k>0$, one obtains that for some $\delta_0>0$ with $\delta_0< L$, 
\begin{equation}
\label{4.100}
u(x)\geq k{x}>0,\ \   \forall\, x\in(0,\delta_0]\,.
\end{equation}

Let us define the set $X$ and the number $y$ by 
$$
X=\{x:\,x\in[\delta_0,L],\,u(x)=0\}
$$
$$
y=\inf_{x\in X}x.
$$

The set $X$ is non-empty since $L\in X$; on the other hand, one has $L\subset[\delta_0,L]$; therefore $y$ is correctly defined. Let $x_n$ be a minimizing sequence, i.e. a sequence which satisfies 
$$
x_n\in X, \mbox{ i.e. } x_n\in[\delta_0,L],\, u(x_n)=0,\,\mbox{ and } x_n\to y.
$$
Then $y\in [\delta_0,L]$ and $u(y)=0$ since $u$ is continuous. Since $u(\delta_0)\geq k\delta_0>0$ by \eqref{4.100}, one has $L\geq y>\delta_0>0$.

Then  apply again  Lemma~\ref{lemma36}, now  with $w=u$, $l=g+c$, $A=\delta_0$,   $B=y$, and $x_0=y$. Fixing any $k>0$,  one obtains that   for some $\delta_y>0$ with $y-\delta_y\geq \delta_0$, 
\begin{equation}\label{beta4}
u(x)\leq k(x-y)<0,\ \ \forall\, x\in[y-\delta_y,y)\,. 
\end{equation}

Now observe that $u(\delta_0)>0$ and $u(y-\delta_y)<0$. Since $u$ is continuous there exists some  $y_0$  such that $\delta_0<y_0<y-\delta_y <y$ with $u(y_0)=0$,  which contradicts the definition of $y$. 
Theorem \ref{36} is proved.

\end{proof}

Our second non-existence result is obtained instead  independently of the source term $g$. It asserts that when the singularity of $\phi$ at $s=0$ is too strong, and, more precisely, when $\phi$ is not integrable both in $0^+$ and $0^-$, then it does not exist any weak solution of problem \eqref{pb1} in the sense of Definition  \ref{defin}.

\begin{theorem}[{{\bf Non-existence when the singularity is too strong}}]\label{45}
Assume that hypothesis \eqref{cond1} holds true, and that the data $(a, g, \phi)$  satisfy hypotheses \eqref{a1}-\eqref{condphi1} and \eqref{p3}. Assume moreover that $\phi$ satisfies
\begin{equation}\label{p+}
\int_0^{+\delta} \phi(t) \,dt =+\infty\,, \,\, \forall \delta,\,\, 0<\delta<1,
\end{equation}
and 
\begin{equation}\label{p-}
\int_{-\delta}^0 \phi(t)\, dt  =+\infty,\, \,\, \forall \delta,\,\, 0<\delta<1. 
\end{equation}
 \indent Then  it does not exist  any weak  solution of problem \eqref{pb1} in the sense of Definition  \ref{defin}. \end{theorem}

\begin{remark}
In the model  case \eqref{pM} where the function $\phi$ is given by $\phi =\phi_\gamma$ defined by 
\begin{equation}
\label{pMbis}
%\begin{cases}
\dys\phi_\gamma (s)=\frac{c}{|s|^{\gamma}} +\varphi (s),
\mbox{ with } c>0, \,\gamma>0,\, \varphi\in C^0(\re),
%\end{cases}
\end{equation}

\noindent hypotheses \eqref{p+} and \eqref{p-} are satisfied if and only if $\gamma\geq 1$. 

\indent In this case the proof of Theorem \ref{45} is very simple. Assume indeed by contradiction that  $u$ is a weak solution of problem \eqref{pb1} in the sense of Definition \ref{defin},  and let $x_0\in [0,L]$  be such that $u(x_0)=0$. Recalling Morrey's embedding  $\huzu\subset C^{0,\frac12}([0,L])$ (see \eqref{Morrey}), one has 
 $$
 |u(x)|=|u(x)-u(x_0)|\leq \|u\|_{C^{0,\frac12}([0,L])}|x-x_0|^{\frac{1}{2}}.
 $$
Using \eqref{pMbis} and recalling that  $u\in H^1(0,L)$ and therefore $\varphi(u)$ is bounded,  one  has 
$$
\phi_\gamma(u(x))= \frac{c}{|u(x)|^\gamma} + \varphi(u(x))\geq \frac{c}{ \|u\|^{\gamma}_{C^{0,\frac12}([0,L])}|x-x_0|^{\frac{\gamma}{2}}} + \inf_{|s|\leq \|u\|_{L^\infty(0,L)}}|\varphi(s)|\,,
$$
from which one deduces that $\phi_\gamma(u(x))\not\in L^2(0,L)$ if $\gamma\geq 1$, a contradiction.

Observe that this proof continues to hold in the case where  \eqref{pM} is only assumed to be in force on a neighborhood of $s=0$. 

\qed
\end{remark} 

\begin{remark}
Let us remark that, when $\phi$ satisfies hypothesis \eqref{condphi1}, it is equivalent to make hypotheses \eqref{p+} and \eqref{p-} for every $\delta$, $0<\delta<1$,  or to assume that there exists some $\delta_0>0$ such that hypothesis \eqref{p+} and \eqref{p-} hold true for this fixed $\delta_0$.

\qed
\end{remark}

 Theorem \ref{45} immediately follows from the following proposition,   which has its own interest. 

\begin{proposition}
\label{46}
Assume that hypothesis \eqref{cond1} holds true, and that the data $(a, g, \phi)$  satisfy hypotheses \eqref{a1}-\eqref{condphi1} and \eqref{p3}. Assume moreover that $\phi$ satisfies  \eqref{p+}.  Then every possible weak solution of problem \eqref{pb1} in the sense of Definition  \ref{defin} is non-positive. 

Similarly, assume that $\phi$ satisfies \eqref{p-}. Then every possible weak solution of problem \eqref{pb1} in the sense of Definition  \ref{defin} is non-negative. 

\end{proposition}

Proposition~\ref{46} is itself an immediate consequence of the following result, where the set $\mathcal{U}$ is defined by 
\begin{equation}\label{setU}\mathcal{U}=\{{u}\in H^1_0 (0,L) \ \text{such that}\ \  \phi({u})\in L^2(0,L)\};\end{equation}
observe that every solution $u$ of problem  \eqref{pb1} in the sense of Definition  \ref{defin} belongs\break to $\mathcal{U}$,  while  $0\not\in \mathcal{U}$.  
\begin{proposition}
\label{47bis}
Assume that hypothesis \eqref{cond1} holds true, let $\phi$ be a nonlinearity satisfying \eqref{condphi1} and \eqref{p3}, and let $u\in \mathcal{U}$. Then: 
if $\phi$ satisfies \eqref{p+},  one has 
\begin{equation}
\label{412<0}
u(x) \leq 0,\bk \quad \forall x\in[0,L];
\end{equation}
if $\phi$ satisfies \eqref{p-},  one has 
\begin{equation}
\label{412>0}
u(x) \geq 0,\bk  \quad \forall x\in[0,L].
\end{equation}
\end{proposition}

\begin{proof}[\bf Proof of Proposition \ref{47bis}]
We will prove that hypothesis \eqref{p+}  implies \eqref{412<0}. In  the case where \eqref{p-} is assumed instead of \eqref{p+}, the proof of \eqref{412>0} is similar.

Assume that \eqref{p+} holds true and assume by contradiction that  $u\in \mathcal{U}$ is such that  for some $x_0\in (0,L)$ one has $u(x_0)>0$.

Define $y_0$ by 
\begin{equation}
\label{411bis}
y_0=\inf\{x\in [0,x_0]:  u(x)>0\}
\end{equation}
(one could also  consider $y_1=\sup\{x\in [x_0,L]:  u(x)>0\}$). Observe that $y_0$ is well defined and that 
\begin{equation}
\label{414bis}
u(y_0)=0;
\end{equation}
indeed, since $u(x)>0$ for every $x\in (y_0,x_0]$, one has $u(y_0)\geq 0$; but if $u(y_0)>0$, there exists $\delta>0$ such that $u(x)>0$ for every $x\in [y_0-\delta,y_0+\delta]$, a contradiction with the definition \eqref{411bis} of $y_0$.

On the other hand, fix $\delta>0$ and define $\psi_\delta:]0,+\infty[\to\re$ by
$$
\psi_\delta(s)=\int_s^\delta \phi(t) dt,\quad \forall s>0,
$$
or equivalently by
$$
\psi_\delta(\delta)=0,\quad \psi'_\delta(s)=-\phi(s),\quad \forall s>0.
$$
Then \eqref{p+} is equivalent to 
\begin{equation}
\label{451}
\psi_\delta(s)\to+\infty \quad \mbox{ as } s\to 0,\,s>0.
\end{equation}

Recall that $u\in H_0^1(\Omega)\subset\subset C^0([0,L])$, and for $\eta$ such that  $0<\eta<x_0-y_0$, define the two real numbers $\underline{u}_\eta$ and $\overline{u}$ by
$$
\underline{u}_\eta=\min_{x\in[y_0+\eta,x_0]}u(x),\quad \overline{u}=\max_{x\in[y_0,x_0]}u(x),
$$
and observe that 
$$
\forall \eta,\quad 0<\eta<x_0-y_0, \mbox{ one has } \,\, 0<\underline{u}_\eta\leq u(x)\leq \overline{u}<+\infty,\quad \forall x\in [y_0+\eta,x_0].
$$
Define also the two real numbers
$$
\underline{\phi}=\min_{s\in[0,\overline{u}]}\phi(s),\quad \overline{\phi}_\eta=\max_{s\in [\underline{u}_\eta,\overline{u}]}\phi(s),
$$
and observe that $ \overline{\phi}_\eta$ is finite for every $\eta$, $0<\eta<x_0-y_0$, even if it is unbounded as $\eta\to 0$.

Then since $u\in H^1(y_0+\eta,x_0)$ and since 
$$
\phi \in C^0_b([\underline{u}_\eta,\overline{u}]), \quad \mbox{ which implies that } \quad \psi_\delta\in C^1([\underline{u}_\eta,\overline{u}]),
$$
one has the chain rule 
$$
\phi(u)\frac{du}{dx}=-\psi'_\delta(u)\frac{du}{dx}=-\frac{d\psi_\delta(u)}{dx} \quad \mbox{ in } \,\, L^2(y_0+\eta,x_0),
$$
and therefore
$$
\int_{y_0+\eta}^{x_0}\phi(u)\frac{du}{dx}dx=\int_{y_0+\eta}^{x_0}-\frac{d\psi_\delta(u)}{dx} dx=\psi_\delta(u(y_0+\eta))-\psi_\delta(u(x_0)), \quad \forall \eta,\quad 0<\eta<x_0-y_0.
$$
Now  $\psi_{\delta}(u(x_0))$ is finite, while in view of \eqref{414bis} and \eqref{451} one has
$$
u(y_0+\eta)\to 0 \quad \mbox{and} \quad \psi_{\delta}(u(y_0+\eta))\to +\infty, \quad \mbox{as } \quad \eta\to 0,\quad \eta>0,
$$
which implies that 
\begin{equation}
\label{452}
\int_{y_0+\eta}^{x_0} \phi(u)\frac{du}{dx} dx \to +\infty, \quad  \mbox{as } \eta\to 0,\,\eta>0,
\end{equation}
{a contradiction since $u \in \mathcal{U}$}. This proves Proposition~\ref{47bis}.

\end{proof}

\section{Studying an (associated) ODE}\label{Sec5}

In this section we will study an Ordinary Differential Equation (ODE) that for the moment we formally write as
\begin{equation}\label{50}
\begin{cases}
\dys a (x) \frac{dv}{dx}   =  {\phi} (v) + h(x) & \text{in}\;(0,L)\,,\\
\dys v(0)=0. & 
\end{cases}
\end{equation}
Under convenient hypotheses, we will prove an existence result, an a priori estimate, and two stability results  (Subsection~\ref{existenceODE}),  a positivity result, a comparison result, and an uniqueness result  (Subsection~\ref{uniquenessODE}). For the sake of exposition these results are summarized in the brief Subsection~\ref{sub53}.

The ODE \eqref{50} is clearly strongly related to problem \eqref{pb1}, see e.g. \eqref{form-def3} in  Proposition \ref{remequiv} above.  In the present  section we will nevertheless study the ODE \eqref{50} for  itself, and we will not try to make connections between the ODE \eqref{50} and problem \eqref{pb1}. The results of the present section will be exploited in the following sections, and in particular in Section~\ref{mult} to obtain multiplicity results for problem \eqref{pb1}. 

In order to emphasize the difference between the present study of ODE \eqref{50} and  the study of problem \eqref{pb1}, we will denote by $v$ (and not by $u$ as in problem \eqref{pb1}) the solution of the ODE \eqref{50}. We will also denote by $(a,h,\phi)$ (and not by $(a,g,\phi)$) the data for ODE \eqref{50}.

In the whole  section we will assume that the data $(a,h,\phi)$ satisfy the following hypotheses (see Remark~\ref{rem51} below for a comparison with the hypotheses \eqref{a1}-\eqref{condphi1}  on the data $(a,g,\phi)$ for problem \eqref{pb1}):
\begin{equation}\label{52n}
a\in L^{\infty}(0,L),\,\, \exists \,\, \alpha,\beta,\,\, 0<\alpha<\beta, \ \alpha\leq a(x)\leq  \beta\, \,\, \text{ a.e.}\,\, x\in(0,L),
\end{equation}
\begin{equation}\label{53n}
h\in L^2(0,L), 
\end{equation}
\begin{equation}
\label{55n}
\begin{cases}
\phi:\re\mapsto\re\cup\{+\infty\},
\,\,\phi\,\, \mbox{ is continuous with values in } \re\cup\{+\infty\},\\
 \phi(s)<+\infty, \ \ \forall s\in\re,\,\, \ s\neq 0\,,
\end{cases}
\end{equation} 
 
\begin{equation}\label{57n}
\int_0^{+\delta}\phi(t)\,dt<+\infty,\quad \int_{-\delta}^{0} \phi(t)\, dt  <+\infty,\, \,\, \forall \delta, \,\, 0<\delta<1,
\end{equation}

\begin{equation}\label{58n}
\phi\in C^0_b(\re\backslash(-\delta, \delta)),  \ \   \forall \ \delta>0.
\end{equation}

\begin{remark}
\label{rem50} 
\indent When  $\phi(0)<+\infty$,  then  \eqref{57n} is automatically satisfied, and, due to \eqref{58n}, one has $\phi\in C^0_b(\re)$. In this case the results of the present section are classical. But the interest is actually in the case 
$$
\phi(0)=+\infty,
$$
 where the results of this section are new.  
 
 \qed
\end{remark}

\begin{remark}
\label{rem51}
Hypothesis \eqref{52n} on $a$, \eqref{53n} on $h$, and \eqref{55n} on $\phi$ are identical to hypotheses \eqref{a1} on $a$, \eqref{g1} on $g$, and \eqref{condphi1} on $\phi$ made in Section~\ref{sec2} above. 

\indent In contrast, hypotheses  \eqref{57n} and \eqref{58n} on $\phi$ are new and restrictive in comparison with the hypotheses made in Section~\ref{sec2} above.

\indent Hypothesis \eqref{57n} it is quite natural in this context;  recall in fact  that it is proved in Theorem~\ref{45} above that if $\phi$ satisfies \eqref{p+} and \eqref{p-}, namely if 
\begin{equation}
\label{5.6bis}
\int_0^{+\delta}\phi(t)\,dt=+\infty,\quad \int_{-\delta}^{0} \phi(t)\, dt  =+\infty,\, \,\, \forall \delta, \,\, 0<\delta<1,
\end{equation}
(compare with hypothesis \eqref{57n}), then problem \eqref{pb1} does not have any weak solution in the sense of Definition~\ref{defin}.

\indent Concerning  hypothesis \eqref{58n}, observe that this new hypothesis  impose a restriction on the function $\phi$ in comparison with the hypotheses  made in Section~\ref{sec2}: indeed hypothesis \eqref{58n} impose that $\phi$ is bounded at $s=-\infty$ and $s=+\infty$, or in other terms that $\phi\in C^0_b(\re\backslash(-\delta, +\delta))$ for every $\delta>0$ (compare with \eqref{rem0inf3}).  However this restriction    can be considered as  tolerable due to the uniform boundedness of every possible solutions of problem \eqref{50} (see Remark \ref{importante2} above).\bk

\qed
\end{remark}
\begin{remark}
Observe that, when $\phi$ satisfies hypothesis \eqref{55n}, it is equivalent to make hypothesis \eqref{57n} for every $\delta$, $0<\delta<1$, and hypothesis \eqref{58n} for every $\delta$, $\delta>0$, or to assume that there exists some $\delta_0>0$ such that hypotheses \eqref{57n} and \eqref{58n} holds true for this $\delta_0$.

\qed
\end{remark}

\begin{remark}
In the ODE \eqref{50}  we have assumed that the (Cauchy) initial condition is $v(0)=0$. Results similar to the ones stated and proved in the present section could be obtained for any arbitrary initial condition $v(0)=v_0\in\re$. We do not consider this possibility here since our interest in the present paper  is only  in the case where $v(0)=0$.
 
\qed
\end{remark}

\subsection{Existence of a solution of  the Cauchy problem \eqref{50}}

\label{existenceODE}

\indent The mathematical (correct) formulation of initial value problem associated to the  ODE in \eqref{50} that we will use in this paper is  the following: we look for a function $v$ which satisfies
\begin{equation}\label{formh}
\begin{cases}
 v\in H^1(0,L), \,\,\phi(v)\in L^2(0,L) \,,\\
\dys a (x) \frac{dv}{dx}   =  {\phi} (v) + h \ \  \text{in} \ \ \mathcal{D}'(0,L),\\
\dys v(0)=0\,. & 
\end{cases}
\end{equation}
 
In this subsection we will prove the following existence result:

\begin{theorem}[{{\bf Existence}}]
\label{shoot} Assume that \eqref{cond1} holds true, and that  the data $(a,h,\phi)$ satisfy hypotheses \eqref{52n}-\eqref{58n}. Then there exists at least a solution of problem \eqref{formh}.
\end{theorem}

The proof of the existence Theorem~\ref{shoot} is based on the two propositions \ref{54} and \ref{reason} below. Before  stating and proving these two propositions, let us state and prove a lemma which looks natural but is not so easy to obtain due to the possible singularity of the function $\phi$.

 Let us define the function  $\psi:\re\to\re$ by 
\begin{equation}\label{5102}
\psi (s)=\int_0^s \phi(t)dt,\ \ \forall\,s\in\re; 
\end{equation}
note that in view of  hypothesis  \eqref{57n} the function $\phi$ is integrable both  in $0^+$ and $0^-$, and that  the function  $\psi$ therefore satisfies
\begin{equation}\label{5103}
\psi \in W^{1,1}_{\mbox{\tiny loc}}(\re)\subset C^0 (\re)\ \ \text{with}\ \ \psi(0)=0\,.
\end{equation}

\begin{lemma}
\label{2.5bis}
Assume that \eqref{cond1} holds true, and that $\phi$ satisfies hypotheses \eqref{55n}-\eqref{58n}. Let $z$  satisfying  
\begin{equation}
\label{5.15bis}
z\in H^1(0,L) \mbox{ with } \phi(z)\in L^2(0,L).
\end{equation}
Then the function  $\psi$ defined by \eqref{5102} satisfies 
\begin{equation}
\label{5.15ter}
\psi(z)\in W^{1,1}(0,L) \mbox{ with }  \frac{d\psi(z)}{dx}= \phi(z) \frac{dz}{dx} \ \  \text{in} \ \ \mathcal{D}'(0,L), 
\end{equation}
a result which in particular implies that
\begin{equation}
\label{514bis}
\int_0^L \phi(z)\frac{dz}{dx} dx = \psi(z(L))-\psi(z(0)).
\end{equation}
\end{lemma}

\begin{remark}
\label{rem2.5ter}
When $\phi(0)<+\infty$, the result \eqref{5.15ter} is classical since then $\phi \in C^0_b(\re)$ in view of  \eqref{55n} and \eqref{58n}. Indeed in this case one has $\psi'=\phi \in C^0_b(\re)$, which implies that $\psi\in C^1(\re)\cap Lip(\re)$; the classical chain rule in $H^1(0,L)$ then implies that 
$$
\psi(z)\in H^1(0,L) \,\, \mbox{ with }\,\, \frac{d\psi(z)}{dx}= \phi(z) \frac{dz}{dx} \ \  \text{in} \ \ \mathcal{D}'(0,L),
$$
which implies \eqref{5.15ter}.

\qed
\end{remark}
\begin{remark}
If $z\in\huzu$, the result \eqref{514bis} implies that 
\begin{equation*}
\int_0^L \phi(z)\frac{dz}{dx} dx = \psi(0)-\psi(0)=0-0=0.
\end{equation*}
This result is nothing but Lemma~\ref{prop24ter} above, which we recall was proven under the sole hypothesis \eqref{condphi1} (identical to \eqref{55n}) without making hypotheses \eqref{57n} and \eqref{58n} on $\phi$.

\qed
\end{remark}

\begin{proof}[\bf Proof of Lemma~\ref{2.5bis}]

\noindent{\bf Step 1.} Fix $R>0$. Then,  one has 
$$\psi(s)=\int_0^s \phi(t)dt=\int_0^{+R}  \phi(t)dt +\int_{+R}^s  \phi(t)dt, \quad \forall s\geq +R,
$$
which, thanks to \eqref{58n},  implies that 

\begin{equation}
\begin{cases}
\begin{split}
\dys |\psi(s)| &\leq \int_{0}^{+R}|\phi(t)| dt + \|\phi\|_{L^{\infty}(+R, +\infty)} \,\,(s-R)\leq \,\\ \\
\dys& \leq  \|\phi\|_{L^1 (0,+R)} +\|\phi\|_{L^{\infty}(+R, +\infty)} \,\, |s|\,,\ \forall\, s\geq +R,
\end{split}
\end{cases}
\end{equation}
a result which in fact holds true for every $s\geq 0$. 

Similarly,  one has, for $s\leq -R$ 
$$\psi(s)=\int_0^s \phi(t)dt=\int_0^{-R}  \phi(t)dt +\int_{-R}^s  \phi(t)dt, \quad \forall s\leq  -R,
$$
which implies that 

\begin{equation}
\begin{cases}
\begin{split}
\dys |\psi(s)| &\leq \int_{-R}^{0}|\phi(t)| dt + \|\phi\|_{L^{\infty}(-\infty, -R)} \,\,(-R-s) \leq \,\\ \\
\dys & \leq  \|\phi\|_{L^1 (-R,\, 0)} +\|\phi\|_{L^{\infty}(-\infty,-R)}\, |s|\,,\ \forall\, s\leq -R,
\end{split}
\end{cases}
\end{equation}
a result which in fact holds true for every  $s\leq 0$.

\indent These two results imply that when $\phi$ satisfies  \eqref{55n}-\eqref{58n}, one has 
\begin{equation}\label{5106}
 \dys |\psi (s)|\leq \|\phi\|_{L^1(-R,+R)}+\|\phi\|_{L^{\infty}(\re\backslash [-R,+R])} \,\, |s| \,,\ \  \dys \forall\, R>0,\, \ \ \forall\, s\in \re.\
\end{equation}

\noindent{\bf Step 2.} Let $T_n$ be the truncation at height $n$ defined by \eqref{tk}. Then $T_n(\phi)\in C^0_b(\re)$. 

Defining the function $\psi_n$ by 
\begin{equation}
\label{defpsin}
\psi_n(s)=\int_0^s T_n(\phi(t))\, dt,\quad \forall s\in\re,
\end{equation}
one has $\psi_n\in C^1 (\re)\cap Lip(\re)$, so that as observed in the Remark~\ref{rem2.5ter} where  $\phi\in C^0_b(\re)$, one has 
\begin{equation}
\label{513bis}
\dys\psi_n(z)\in H^1(0,L)\,\, \mbox{ with } \,\, \frac{d\psi_n(z)}{dx}= T_n(\phi(z)) \frac{dz}{dx} \ \  \text{in} \ \ \mathcal{D}'(0,L).
\end{equation}

Let us now use the fact that   $
\phi(z)\in L^2(0,L)
$.

Since
$$
|T_n(\phi(z(x)))|\leq |\phi(z(x))|, \,\mbox{ a.e. } \, x\in(0,L) \, \mbox{ with } \, T_n(r)\stackrel{n\to\infty}{\longrightarrow} r,\, \forall r\in\re,
$$
Lebesgue's dominated convergence theorem implies that
\begin{equation}
\label{5101b}
 T_n(\phi(z))\frac{dz}{dx} \to   \phi(z)\frac{dz}{dx}\,\, \mbox{ strongly in } \,\, L^1(0,L).
\end{equation}

On the other hand, in view of \eqref{5106}, one has

\begin{equation}
\begin{split}
\dys |\psi_n(s)|  & \leq \|T_n(\phi)\|_{L^1(-R,+R)} + \|T_n(\phi)\|_{L^\infty(\re\setminus [-R,+R])}|s|\leq \,\\ \\
\dys & \leq  \|\phi\|_{L^1 (-R,+R)} +\|\phi\|_{L^{\infty}(\re\setminus [-R,+R])}\, |s|, \quad \forall k>0,\,\,\forall s\in\re.
\end{split}
\end{equation}

Since $z\in H^1(0,L) \subset L^\infty(0,L)$, this implies that 
\begin{equation}
\label{517bis}
\psi_n(z) \,\,\mbox{ is bounded in }\,\, L^\infty(0,L).
\end{equation}

Also, for every $s\in\re$, one has by \eqref{57n}
$$
\phi\in L^1(0,s) \, \mbox{ if } \, s>0, \, \,\mbox{ and } \, \phi\in L^1(s,0) \, \mbox{ if } \, s<0,
$$
while 
\begin{equation*}
\begin{cases}
|T_n(\phi(t))|\leq |\phi(t)|\,, \forall\, t\in (0,s) \, \mbox{ if } \, s>0, \, \,\mbox{and } \, \forall t\in(s,0) \, \mbox{ if } \, s<0,\nonumber\\
\mbox{with} \, T_n(r)\stackrel{n\to\infty}{\longrightarrow} r,\, \forall r\in\re, 
\end{cases}
\end{equation*}
so that Lebesgue's dominated convergence theorem (in $L^1(0,s)$ and in $L^1(s,0)$) implies that 
\begin{equation}
\label{5102b}
\psi_n(s)=\int_0^s T_n(\phi(t)) dt\to \int_0^s \phi(t)dt=\psi(s),\quad \forall s\in\re.
\end{equation}

From \eqref{517bis} and \eqref{5102b} one deduces (using again Lebesgue's dominated convergence theorem) that 
\begin{equation}
\label{521bis}
\psi_n(z)\to \psi(z) \quad \mbox{strongly in } \quad L^p(0,L) \quad \forall p, \,\, 1\leq p< +\infty.
\end{equation}

This fact implies that $\dys\frac{d\psi_n(z)}{dx}\stackrel{n\to\infty}{\longrightarrow} \frac{d\psi(z)}{dx}$ in $\mathcal{D}'(0,L)$ and also, by \eqref{513bis} and \eqref{5101b}, strongly in $L^1(0,L)$; this  implies that $\dys\frac{d\psi(z)}{dx}= \phi(z) \frac{dz}{dx}$, and therefore that $\psi(z)\in W^{1,1}(0,L)$.

Lemma~\ref{2.5bis} is proved.

 \end{proof}

 \begin{proposition}[{{\bf A priori estimate}}]
\label{54}
Assume that \eqref{cond1} holds true, and that  the data $(a,h,\phi)$ satisfy hypotheses \eqref{52n}-\eqref{58n}.  If $v$ is any solution of the problem \eqref{formh}, then,  for any given $R>0$,  $v$ satisfies 
\begin{equation}\label{stimah}
\|v\|_{H^1(0,L)}\leq C_R\,,
\end{equation}
where  $C_R$ is given by 
\begin{equation}\label{520}
C_R= (L+ 1)\left(\frac{\sqrt{L}\|\phi\|_{L^{\infty}(\re\backslash [-R,+R])} + \|h\|_{L^2(0,L)}}{\alpha} +\frac{\sqrt{\|\phi\|_{L^1(-R,+R)}}}{\sqrt{\alpha}}\right)\,,
\end{equation}
which   depends only on $R$, ${L}$, $\alpha$, $\|h\|_{L^{2}(0,L)}$,   $ \|\phi\|_{L^{1}(-R,+R)} $,  and $\| \phi \|_{L^{\infty}(\re\backslash (-R,+R))} $. 
\end{proposition}

\begin{proof}[\bf Proof of Proposition~\ref{54}] 

 Multiplying pointwise the second line of \eqref{formh} by $\displaystyle\frac{d v}{dx}$ and  integrating between $0$ and $L$,  we get
\begin{equation}\label{5101}
\int_0^{L} a(x)\left|\frac{d v}{dx}\right|^2 dx= \int_0^{L}\phi (v)\frac{d v}{dx}  dx+ \int_0^{L} h (x)\frac{dv}{dx} dx\,.
\end{equation}

Using in \eqref{5101} the coercivity of $a$ (see \eqref{52n}), the result \eqref{514bis}, $\psi(v(0))=\psi(0)=0$, and the  Cauchy-Schwartz inequality implies that

\begin{equation}\label{estima}
\dys \alpha\int_0^L \left|\frac{d v}{dx}\right|^2  dx \leq  |\psi (v (L))|  +\displaystyle\|h\|_{L^2(0,L)} \left\|\frac{d v}{dx}\right\|_{L^2(0,L)}.
\end{equation}

Since $v\in H^1 (0,L)$ with $v(0)=0$, \eqref{5106} combined with Morrey's estimate \eqref{rem210} (which continues to hold true for $z\in H^1(0,L)$ with $z(0)=0$ without assuming that $z(L)=0$),   yields 

\begin{equation}\label{5107}
\begin{split}
|\psi(v(L))| & \leq  \|\phi\|_{L^1(-R,+R)}+\|\phi\|_{L^{\infty}(\re\backslash [-R,+R])} \,\,|v(L)| \leq  \\ \\\dys & \leq    \|\phi\|_{L^1(-R,+R)}+\|\phi\|_{L^{\infty}(\re\backslash [-R,+R])}\sqrt{L}\left\|\frac{dv}{dx}\right\|_{L^2(0,L)}\,,\,\ \forall\, R>0.
\end{split}
\end{equation}

Turning back to \eqref{estima} we have proved that, for every $R>0$ 
\begin{eqnarray}\label{5108}
\dys \alpha\left\|\frac{dv}{dx}\right\|^2_{L^2(0,L)}   \leq \left(\sqrt{L}\|\phi\|_{L^{\infty}(\re\backslash [-R,+R])}+ \|h\|_{L^2(0,L)}\right) \left\|\frac{d v}{dx}\right\|_{L^2(0,L)}+  \|\phi\|_{L^1(-R,+R)}.\nonumber
\end{eqnarray}

From this inequality, using the fact that for  $\alpha>0$, $B>0$, $\Gamma> 0$, one has 
$$
\dys \alpha X^2 \leq BX+\Gamma, \ \  X\geq 0 \Longleftrightarrow 0\leq X\leq \frac{B+\sqrt{B^2 +4\alpha\Gamma}}{2\alpha}$$
and then the inequality
$$ \dys  \frac{B+\sqrt{B^2 +4\alpha\Gamma}}{2\alpha}\leq \frac{B+B +2\sqrt{\alpha \Gamma}}{2\alpha}=\frac{B}{\alpha} +\frac{\sqrt{\Gamma}}{\sqrt{\alpha}}, 
$$ 
one deduces that $v$ satisfies 
\begin{equation}\label{5109}
\left\|\frac{dv}{dx}\right\|_{L^2(0,L)}  \leq \frac{\sqrt{L}\|\phi\|_{L^{\infty}(\re\backslash [-R,+R])} + \|h\|_{L^2(0,L)}}{\alpha} +\frac{\sqrt{\|\phi\|_{L^1(-R,+R)}}}{\sqrt{\alpha}}\,, \forall \, R>0\,.
\end{equation}
Combined with the Poincar\'e inequality \eqref{Poincare} (which continues to hold true for $z\in H^1(0,L)$ with $z(0)=0$ without assuming that $z(L)=0$), namely 
\begin{equation}\label{5109b}
\|z\|_{L^2(0,L)}\leq {L}\left\|\frac{dz}{dx}\right\|_{L^2(0,L)},
\end{equation}
formula \eqref{5109} gives the desired a priori estimate \eqref{stimah} with $C_R$ given by \eqref{520}.

\end{proof}

\begin{remark}
Observe that as far as Lemma ~\ref{54} is concerned it is  an a priori estimate for any possible solution of the problem \eqref{formh}.

Observe that the proof of this result  is a proof in the spirit of an a priori estimate for a PDE, rather then for an ODE.

\qed
\end{remark}

\begin{proposition}[{{\bf Passage to the limit}}]\label{reason}
Assume that \eqref{cond1} holds true, and that  the data  $(a,h,\phi)$ satisfy  hypotheses \eqref{52n}--\eqref{58n}. Let $h_k$ be a sequence which satisfies 
\begin{equation}
\label{5.29bis}
 h_k\in L^2(0,L), \,\,\,\, h_k\to h \,\,\,\, \mbox{weakly in } \,\,L^2(0,L). 
 \end{equation}
  Let also $\phi_k$ be a sequence of reasonable approximations of $\phi$ which satisfies  \eqref{55n}--\eqref{58n} for every $k$. Assume that  there exists some $R^*>0$ and  some $C^*>0$ such that, one has 
{\small\begin{equation}\label{525}
 (L+ 1)\left(\frac{\sqrt{L}\|\phi_k\|_{L^{\infty}(\re\backslash [-R^*, +R^*])} + \|h_k\|_{L^2(0,L)}}{\alpha} +\frac{\sqrt{\|\phi_k\|_{L^1(-R^*,+R^*)}}}{\sqrt{\alpha}}\right)\leq C^*, \quad \forall k\,.
\end{equation}}

Consider a sequence of solutions $v_k$ of the ODE  problem
\begin{equation}\label{happrox}
\begin{cases}
 v_k\in H^1(0,L), \,\,\phi_k(v_k)\in L^2(0,L) \,,\\
\dys a (x) \frac{dv_k}{dx}   =  {\phi_k} (v_k) + h_k \ \  \text{in} \ \ \mathcal{D}'(0,L),\\
\dys v_k(0)=0\,. & 
\end{cases}
\end{equation}

Then there exists a function $v$ and a subsequence, denoted by $v_{k'}$, such that \begin{equation}\label{525b}v_{k'}\rightharpoonup v \ \ \text{weakly in}\ \ H^1(0,L),\end{equation} where $v$ is a solution of the ODE
\begin{equation}
\label{lem5733}
\begin{cases}
 v\in H^1(0,L), \,\,\phi(v)\in L^2(0,L) \,,\\
\dys a (x) \frac{dv}{dx}   =  {\phi} (v) + h \ \  \text{in} \ \ \mathcal{D}'(0,L),\\
\dys v(0)=0\,. & 
\end{cases}
\end{equation}

Moreover, if further hypothesis to \eqref{5.29bis}, the sequence $h_k$ satisfies 
\begin{equation}
\label{5.32bis}
h_k\rightarrow h \,\,\,\, \mbox{ strongly in } \,\,L^2(0,L),
\end{equation}
then, further to \eqref{525b}, one has 
\begin{equation}
\label{5.32ter}
v_{k'}\rightarrow v \,\,\,\, \mbox{ strongly in }\,\, \huzu.
\end{equation}

\end{proposition}

\begin{proof}[\bf Proof]

In view of \eqref{stimah}, \eqref{520}, and \eqref{525} we have 
$$
\|v_k\|_{H^1(0,L)}\leq C^*\,. 
$$
Therefore,  there exists  a $v\in H^1(0,L)$ and a subsequence, denoted by $v_{k'}$, such that
$$
v_{k'}\to v\   \text{weakly in} \  H^1(0,L),\ \text{strongly in }  L^2(0,L), \text{and a.e. } \mbox{ in } (0,L). 
$$
Moreover, 
$$
 \phi_k (v_k) = a (x) \frac{d v_k}{dx} - h_k \ \ \text{is bounded in }  L^2(0,L),
$$
so that, since $\phi_k$ is a  sequence of reasonable approximations of $\phi$, one gets 
$$
\phi_{k'} (v_{k'}\bk) \rightharpoonup \phi(v) \ \ \text{ a.e. in $(0,L)$ and weakly in } L^2(0,L),
$$
which allows one to pass to the limit in  \eqref{happrox} and to obtain \eqref{lem5733}.  

Assume now hypothesis \eqref{5.32bis}, namely that the sequence $h_k$ converges strongly in $L^2(0,L)$. Multiplying pointwise the equation in \eqref{happrox} by  $\dys\frac{dv_k}{dx}$ and integrating on $(0,L)$ we get
\begin{equation*}
\int_0^{L} a(x) \frac{d v_k}{dx} \frac{d v_k}{dx} dx= \int_0^{L}\phi_k (v_k)\frac{d v_k}{dx}  dx+ \int_0^{L} h_k (x)\frac{dv_k}{dx} dx\,.
\end{equation*}

Define now in this proof the function $\psi_k$ by 
$$
\psi_k(s)=\int_0^s \phi_k(t)\,dt,\,\,\forall k\in\re,
$$
(please do not make confusion with the function $\psi_n$ defined by \eqref{defpsin}), and using formula \eqref{5106} and $\psi_k(v_k(0))=\psi_k(0)=0$, we get 
\begin{equation}
\label{5.5555}
\int_0^{L} a(x) \frac{d v_k}{dx} \frac{d v_k}{dx} dx= \psi_k(v_k(L))+ \int_0^{L} h_k (x)\frac{dv_k}{dx} dx\,.
\end{equation}
Since the subsequence $v_{k'}$ converges weakly in $\huzu$ (see \eqref{525b}) and therefore by Morrey's embedding in $C^0([0,L])$ strongly, since $\phi_k$ is a sequence of reasonable approximations of $\phi$, and since the sequence $h_k$ is assumed to converge strongly to $h$ in $L^2(0,L)$ (see \eqref{5.29bis}), the right hand side of \eqref{5.5555} converges to  
$$
\psi(v(L))+\int_0^{L} h (x)\frac{dv}{dx} dx,
$$
which is nothing but 
$$
\int_0^{L} a(x) \frac{d v}{dx} \frac{d v}{dx} dx
$$
(multiply pointwise the second line of \eqref{formh} by $\dys \frac{dv}{dx}$ and integrate between $0$ and $L$), we have proved that 
\begin{equation}
\label{5.5556}
\int_0^L a(x) \frac{d v_{k'}}{dx} \frac{d v_{k'}}{dx}\,dx\to \int_0^L a(x) \frac{d v}{dx} \frac{dv}{dx}\,dx. 
\end{equation}
Passing to the limit in 
$$
\int_0^L a(x) \left(\frac{d v_{k'}}{dx} -\frac{d v}{dx}\right)\left(\frac{d v_{k'}}{dx} -\frac{d v}{dx}\right)\,dx\geq \alpha \int_0^L \left(\frac{d v_{k'}}{dx} -\frac{d v}{dx}\right)^2\,dx
$$
with the help of \eqref{525b} and \eqref{5.5556} proves \eqref{5.32ter}.

Proposition~\ref{reason} is proved.

\end{proof}
\begin{proof}[\bf Proof of Theorem \ref{shoot}] 
In order to prove the existence theorem \ref{shoot} we will  apply three times Proposition \ref{reason} to the following sequences of  approximations: 

$\bullet$ Firstly,  for any $n\in \mathbb{N}$,  we define $\phi_n: \re\to\re$  by  
$$
\phi_n(s)= T_n(\phi(s))\,,
$$
where $T_n$ is the truncation at level $n$ defined by \eqref{tk}. 

$\bullet$ Then, for $n$ fixed,  for any $m\in \mathbb{N}$,  we define $\phi_{n,m}: \re\to\re$  by
$$
\phi_{n,m} (s) = 
\begin{cases}
\phi_n(-m) & \text{if}\ s<-m, \\
\phi_n(s) & \text{if}\ |s|\leq m,\\
\phi_n(m) & \text{if}\ s>m\,.
\end{cases}
$$

$\bullet$ Finally, for $n$ and $m$ fixed,  for any $\varepsilon>0$,  we define  $\phi_{n,m, \varepsilon}: \re\to\re$  by
$$
\phi_{n,m,\varepsilon}=  \phi_{n,m}* \rho_{\varepsilon}, 
$$
where $ \rho_{\varepsilon} $ is a standard sequence of mollifiers.

We will first pass to the limit in $\varepsilon$ for $n$ and $m$ fixed.

Recalling the fifth example in Remark~\ref{examples}, one observes that the function  $\phi_{n,m}\in C^0(\re)$ and that the function $\phi_{n,m,\varepsilon} \in Lip(\re)$. Therefore there exists a unique solution $v_{n,m, \varepsilon}$ of problem \eqref{formh} for the function $\phi_{n,m,\varepsilon}$.  For $n$ and $m$  fixed,  the sequence $\phi_{n,m,\varepsilon}$ is a sequence of reasonable approximations of $\phi_{n,m}$, which satisfies 
$$
\|\phi_{n,m,\varepsilon}\|_{L^{\infty}(\re)}\leq \|\phi_{n,m}\|_{L^{\infty}(\re)}\leq n. 
$$
Therefore, for $n$ and $m$ fixed,  \eqref{525} is satisfied with a constant $C^*$ given by  
$$(L+ 1)\left(\frac{\sqrt{L}n + \|h\|_{L^2(0,L)}}{\alpha} +\frac{\sqrt{ 2R^*n}}{\sqrt{\alpha}}\right).$$ 

\noindent An application of  Proposition~\ref{reason} proves the existence of a solution $v_{n,m}$ of problem \eqref{formh} for the function $\phi_{n,m}$.

We will then pass to the limit in $m$ for $n$ fixed.

For $n$ fixed,  the sequence  $\phi_{n,m}$ is a sequence of  reasonable approximations of $\phi_{n}$, which satisfies, for $m>R^*$,  
$$
\phi_{n,m}=\phi_n \,\,\mbox{ in } \,\, (-R^*,+R^*),
$$
$$
\|\phi_{n,m}\|_{L\infty(\re\setminus[-R^*,+R^*])} \leq  \|\phi_n\|_{L^\infty(\re\setminus[-R^*,+R^*])}. 
$$

Therefore, for $n$ fixed and $m>R^*$, \eqref{525} is satisfied with a constant $C^*$ given by
$$
(L+ 1)\left(\frac{\sqrt{L}\|\phi_n\|_{L^{\infty}(\re\backslash [-R^*,+R^*])} + \|h\|_{L^2(0,L)}}{\alpha} +\frac{\sqrt{\|\phi_n\|_{L^1(-R^*,+R^*)}}}{\sqrt{\alpha}}\right)\,.
$$

\noindent An application of Proposition~\ref{reason} proves the existence of a solution $v_{n}$ of problem \eqref{formh} for the function $\phi_{n}$.

We will finally pass to the limit in $n$.

The sequence $\phi_{n}$ is a sequence of reasonable approximations of $\phi$ (recall the first example in Remark~\ref{examples}), which satisfies 
$$
 |\phi_{n}(s)|\leq |\phi(s)|, \quad \forall s\in\re.  
$$

Therefore \eqref{525}  is satisfied with a constant $C^*$ given by 
$$
 (L+ 1)\left(\frac{\sqrt{L}\|\phi\|_{L^{\infty}(\re\backslash [-R^*,+R^*])} + \|h\|_{L^2(0,L)}}{\alpha} +\frac{\sqrt{\|\phi\|_{L^1(-R^*,+R^*)}}}{\sqrt{\alpha}}\right)\,.
$$

\noindent An application of  Proposition~\ref{reason} proves the existence of a solution $v$ of problem \eqref{formh} for the function $\phi$, namely the   Theorem~\ref{shoot}.

\end{proof}

 \subsection{Positivity, comparison and  uniqueness}
 
 \label{uniquenessODE}
 
 In this subsection, further to hypotheses \eqref{52n}--\eqref{58n}  on the data $(a,g,\phi)$, we shall  assume that 
 \begin{equation}
 \label{533bis}
 \phi(0)=+\infty,
 \end{equation}
 i.e. that $\phi$ is singular in $s=0$, and that 
 \begin{equation}\label{gf}
\forall \ \eta\in (0,L),\ \ \exists \, M_\eta >0, \ \text{such that}\ \ h(x)\geq -M_\eta\ \ \forall\  x\in [0,L-\eta];
\end{equation}
 i.e. that $h$  may  blow down to $-\infty$ only for  $x=L$.

\begin{remark}
Observe that  hypothesis \eqref{gf} looks similar to hypothesis \eqref{glim} of Theorem~\ref{36}, but actually implies very different consequences. Indeed in \eqref{gf} 
 one assumes that $\eta>0$  while, in some sense in  Theorem~\ref{36}, one assumes $\eta=0$ (an assumption that implied  that 
 problem \eqref{pb1} does not have any weak solution in the sense of Definition~\ref{defin}).
 
 \qed
\end{remark}

\begin{proposition}[{{\bf Positivity}}]\label{coneL}
Assume that \eqref{cond1} holds true, and that  the data  $(a,h,\phi)$ satisfy  hypotheses \eqref{52n}--\eqref{58n}  and \eqref{533bis}--\eqref{gf}. 

Then any solution $v$  of the ODE problem \eqref{formh} satisfies 
\begin{equation}
\label{5.1000}
v(x)>0,\quad  \forall x\in(0,L). 
\end{equation}

\end{proposition}

\begin{figure}[htbp]\centering
\includegraphics[width=4.5in]{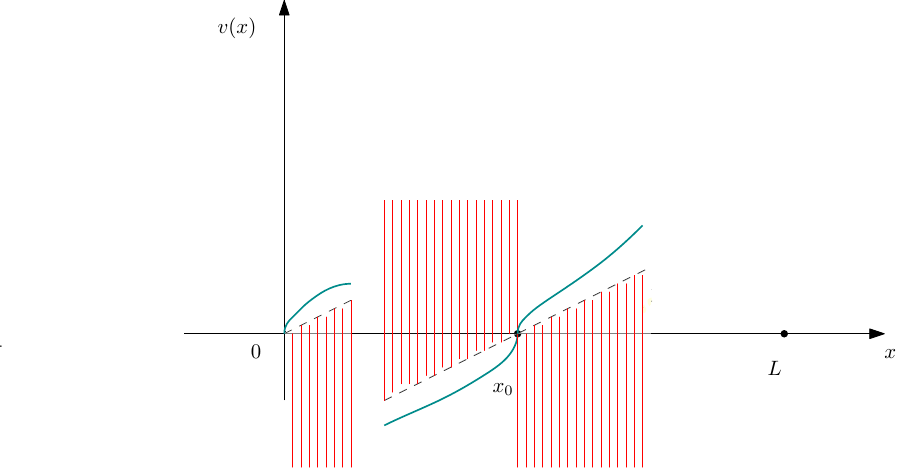}
\caption{The forbidden region in red}
\end{figure}

\begin{proof}[\bf Proof] 
Since all the hypotheses of the existence Theorem~\ref{shoot} are assumed in the statement of Proposition~\ref{coneL}, the ODE problem \eqref{formh} has at least a solution $v$.

Apply Lemma~\ref{lemma36} with   $w=v$, $l=h$, $A=0$, $B=L $,  and $x_0=0$; this is licit in view of hypothesis \eqref{gf}. Fixing any $k>0$, one obtains that for some $\delta_0>0$ with $\delta_0< L$, one has 
\begin{equation}
\label{5.101}
v(x)\geq k{x}>0,\ \   \forall\, x\in(0,\delta_0]\,.
\end{equation}

The  function $v\in C^0([0,\delta_0])$ satisfies \eqref{5.101}. Then one has the following alternative:  either 
\begin{equation}
\label{5.102}
v(x)>0,\ \   \forall\, x\in(0,L),
\end{equation}
or
\begin{equation}
\label{5.103}
\exists \hat{x},\quad \hat{x}\in(\delta_0,L),\quad v(\hat{x}) = 0.
\end{equation}

In the second case, we define the set $X$ and the number $y$ by 
$$
X=\{x:\,x\in[\delta_0,L],\,u(x)=0\}
$$
$$
y=\inf_{x\in X}x.
$$

The set $ {X}$ is non-empty, since $\hat{x}\in {X}$; therefore ${y}$ is correctly defined, $y>\delta_0$, and reasoning as in the proof of Theorem \ref{36}, $y$ is actually a minimum, i.e. $v(y)=0$.

Apply now  Lemma~\ref{lemma36} with  $w=v$, $l=h$, $A=0$, $B=L$,  and $x_0={y}$. Fixing any $k>0$, one obtains that for some $\delta_{ {y}}>0$, one has 
\begin{equation*}
v(x)\leq k(x- {y})<0,\ \   \forall\, x\in( {y}-\delta_{ {y}}, {y})\,.
\end{equation*}
a contradiction with the definitions of the set ${X}$ and of ${y}$.

Therefore the second possibility \eqref{5.103} is impossible and so  \eqref{5.102} holds, i.e. 
$$
v(x)>0,\ \   \forall\, x\in(0,L ).
$$ 

This  completes the proof of Proposition~\ref{coneL}.

\end{proof}

\begin{remark}\label{511} 
Proposition \ref{shoot} can be  analogously proved   for the "backward" case (starting from $L$) by performing the change of variable $y=L-x$. In this case, a corresponding result of the one in Lemma \ref{coneL} can be proved provided one assume \eqref{gf}  in $[\eta,L]$ instead of $[0,L-\eta]$; as an application of Lemma~\ref{lemma36} in this case the solution is negative and it cannot be zero up to possibly $x=0$.  \qed
\end{remark}

In the rest of this section, further to hypotheses \eqref{52n}--\eqref{58n} and  \eqref{533bis}--\eqref{gf}, we also assume that $\phi(s)$ is monotone non-increasing for  $s>0$, i.e. that: 
\begin{equation}
\label{monot}
\phi(s)\geq \phi(t) \ \ \text{for} \ \ \ 0\leq s\leq t\,.
\end{equation}

\begin{proposition}[{{\bf Comparison and uniqueness}}]\label{compar}
Assume that \eqref{cond1} holds true, and that  the data  $(a,h_1,\phi)$ and $(a,h_2,\phi)$ satisfy  hypotheses \eqref{52n}--\eqref{58n},  \eqref{533bis}--\eqref{gf}, and \eqref{monot}. Let $v_1$ and $v_2$ be solutions of problem \eqref{formh} for the data $h_1$ and $h_2$. Assume that  \begin{equation}\label{h1h2}h_1\leq h_2\,,\end{equation}
then one has  \begin{equation}v_1\leq v_2.\end{equation}

This comparison result immediatly implies that the solution of ODE problem \eqref{formh} is unique.
\end{proposition}

\begin{proof}[\bf Proof]
By Theorem \ref{coneL}  $v_1$ and $v_2$ are positive. We take the difference of the second lines of the formulations 
$$
\begin{cases}
 v_1\in H^1(0,L), \,\,\phi(v_1)\in L^2(0,L) \,,\\
\dys a (x) \frac{dv_1}{dx}   =  {\phi} (v_1) + h_1 \ \  \text{in} \ \ \mathcal{D}'(0,L),\\
\dys v_1(0)=0\,, & 
\end{cases}
$$
$$
\begin{cases}
 v_2\in H^1(0,L), \,\,\phi(v_2)\in L^2(0,L) \,,\\
\dys a (x) \frac{dv_2}{dx}   =  {\phi} (v_2) + h_2 \ \  \text{in} \ \ \mathcal{D}'(0,L),\\
\dys v_2(0)=0\,, & 
\end{cases}
$$
and we multiply this difference pointwise by $(v_1 - v_2)^+$. Since $(v_1 - v_2)^+\in H^1(0,L)$, using \eqref{monot} and \eqref{h1h2}, one has 
$$
\begin{cases}
 \dys \frac{a(x)}{2} \frac{d}{dx}   ((v_1 - v_2)^+)^2 = (\phi (v_1) - \phi(v_2))(v_1 - v_2)^+ + (h_1 - h_2)(v_1 - v_2)^+\leq 0,\\
 (v_1-v_2)^+(0)=0,
 \end{cases}
$$
which easily implies the comparison result.

\end{proof}

\subsection{Synthesis of  the results on ODE problem \eqref{formh}} 
\label{sub53}

To conclude this section, we synthesize the results that we have proved concerning the ODE problem  \eqref{formh}, namely
\begin{equation*}
\begin{cases}
 v\in H^1(0,L), \,\,\phi(v)\in L^2(0,L) \,,\\
\dys a (x) \frac{dv}{dx}   =  {\phi} (v) + h \ \  \text{in} \ \ \mathcal{D}'(0,L),\\
\dys v(0)=0\,. & 
\end{cases}
\end{equation*}

If we assume that the data satisfy hypotheses  \eqref{52n}--\eqref{58n}, then the  ODE problem \eqref{formh} has at least one solution (Theorem~\ref{shoot}). Moreover, all the possible solutions of \eqref{formh} satisfy an a priori estimate (Proposition~\ref{54}). These solutions enjoy a stability property (Proposition~\ref{reason}), namely the fact that from a sequence of solutions of problems   \eqref{formh} relative to a  sequence of reasonable approximations $\phi_k$ of $\phi$, one can extract a subsequence which converges to a solution of problem   \eqref{formh} relative to $\phi$, weakly in $H^1(0,L)$ if the sequence of source terms converges weakly in $L^2(0,L)$, and strongly in $H^1(0,L)$ if  the sequence of source terms converges strongly in $L^2(0,L)$.

If further to hypotheses \eqref{55n}--\eqref{58n}, we assume that the data satisfy hypotheses \eqref{533bis} ($\phi(0)=+\infty$) and \eqref{gf} ($h(x)\geq M_\eta$ on $(0,L-\eta)$ for every $\eta>0$), then every solution of the ODE problem \eqref{formh} is positive in $(0,L)$ (Proposition~\ref{coneL}).

If in addition to hypotheses \eqref{55n}--\eqref{58n} and \eqref{533bis}--\eqref{gf}, we assume hypothesis \eqref{monot} ($\phi$ monotone non-increasing for $s>0$), then the solutions of the ODE problem \eqref{formh} satisfy a comparison principle and the solution is unique (Proposition~\ref{compar}). In this latest case, the stability property described above become the continuity (in the weak and in the strong topologies) of the application which from the source term $h$ provides the (unique) solution of the ODE problem \eqref{formh}.

\section{An unexpected  multiplicity result}\label{Sec6}
\label{mult}

In this section  we  show that, under suitable assumptions on $\phi$, if there exists a solution of \eqref{pb1} in the sense of Definition \ref{defin} then there   are infinitely many solutions of \eqref{pb1} in the sense of Definition \ref{defin}. 
Here we will assume that $g$ satisfies \eqref{gf} with $h$ replaced by $g$, i.e. 
\begin{equation}\label{gf6}
\forall \ \eta\in (0,L),\ \ \exists \, M_\eta >0, \ \text{such that}\ \ g(x)\geq -M_\eta\ \ \forall\  x\in [0,L-\eta]. 
\end{equation}

Here is the main result of this section:
\begin{theorem}\label{une} \label{cara}
Assume that \eqref{cond1} holds true, and that  the data $(a,g,\phi)$ satisfy hypotheses \eqref{a1}--\eqref{p3},  \eqref{57n}-\eqref{58n},  \eqref{monot} and  \eqref{gf6}.

 \indent Assume  that there exists a solution $\ou$ of problem \eqref{pb1} in the sense of Definition \ref{defin}. Then there exist infinitely many solutions of \eqref{pb1} in the sense of Definition \ref{defin}.

More precisely, there  exists a critical value $c^*$ satisfying 
\begin{equation}\label{cstar}
c^* \leq \frac{1}{\sqrt{L}}\left(\frac{\beta}{\alpha}+1\right)\|g\|_{L^2(0,L)}-\inf_{s\in\re} \phi (s)\,,
\end{equation}
and a function $U$
\begin{equation}\label{U}
U: {c\in (-\infty,c^*] \longrightarrow U(c)\in \huzu}\, \ \ \text{weakly  continuous in $\huzu$}\,,
\end{equation}
which satisfies 
\begin{equation}\label{58}
 \text{$U(c)$ is a solution of problem \eqref{pb1} in the sense  of Definition \ref{defin}}\,,
\end{equation}
\begin{equation}\label{posi}
U(c)\geq 0\,,
\end{equation}
\begin{equation}\begin{cases}\label{511}
\begin{array}{l}
 \displaystyle\text{for any}\ \  c_1, c_2\in (-\infty, c^*]\ \text{such that $c_1<c_2$}, \\ \displaystyle\text{then  $U({c_1})(x)< U({c_2})(x)$}\ \ \text{for any}\ x\in (0,L)\,,
\end{array}\end{cases}\end{equation}

\begin{equation}\label{uniform}
 U(c) \to 0\ \ \text{ weakly   in $\huzu$ (and uniformly on}\ \  [0,L]) \ \text{when $c\to -\infty$}\,.
\end{equation}
\indent Moreover 
\begin{equation}\begin{cases} \label{anyu}
 \text{for any solution $u$ of problem \eqref{pb1} in the sense  of Definition  \ref{defin} for the data $(a,g,\phi)$, }
    \\ \text{there exists $c\in (-\infty, c^*]$ such that  $u=U(c)$}\,.
\end{cases}
\end{equation}
\end{theorem}

\begin{proof}[Proof of Theorem \ref{une}]

 \bk 
We assumed that there exists a solution $\overline{u}$  of problem \eqref{pb1} in the sense  of Definition \ref{defin}.  Then by Proposition \ref{remequiv}, $\ou$ also solves 

\begin{equation}\label{ubarra}
\begin{cases}
\dys \ou\in \huzu, \,\,\,  \phi(\ou)\in L^2(0,L),  \\
\dys a (x) \frac{d \ou}{dx}   =  \phi (\ou) + g + \overline{c}\,\,\,\,  \text{in}\ \mathcal{D}'(0,L),
\end{cases}
\end{equation}
for 
\begin{equation}
\displaystyle \overline{c}= \dys -  \frac{\int_0^L \frac{\phi(\ou)}{a(x)} dx +\int_0^L\frac{g}{a(x)} dx}{\int_0^L\frac{1}{a(x)} dx}\,. 
\label{oc}\end{equation}

On the other hand, for every  $c\in\re$,  by Theorem \ref{shoot} with $h=g+c$ and Proposition \ref{compar} there exists a unique solution $v_c$ of 
\begin{equation}\label{formc}
\begin{cases}
v_c\in H^1(0,L),  \,\,\,    \phi(v_c)\in L^2(0,L), \\
\dys a (x) \frac{d v_c}{dx}   =  \phi (v_c) + g+c \,\,\,\, \text{in}\ \mathcal{D}'(0,L),\\
\dys v_c(0)=0. 
\end{cases}
\end{equation}
This allows us to  define a function $V$ by 
$$
V: {c\in \re \longrightarrow V(c)=v_c\in H^1(0,L)}\,.
$$
   \indent Since,  by Lemma \ref{coneL}, $V(c)=v_c>0$ in $(0,L)$,  one has 
   \begin{equation}\label{vcp}
   V(c)(L)\geq 0\,\,\,\, \forall\,\, c\in\re.  
   \end{equation}
   
 Observe that   $V(c)$  is a solution of problem \eqref{pb1} in the sense of Definition \ref{defin} if and only if $V(c)(L)=0$.

Moreover, as a consequence of  Proposition \ref{compar},  $\overline{u}=V({\overline{c}})$, and, in particular,  $V({\overline{c}})(L)=0$.

 For any $c<\overline{c}$ one has, again by   Proposition \ref{compar},  $V(c)\leq V({\overline{c}})$.  Therefore, by \eqref{vcp},  we have  $V(c)(L)=0$  and  $V(c)$ is a solution of  \eqref{pb1} in the sense of Definition \ref{defin}.

\bk

For any $c\in \re$ we consider the solution $V(c)$ of \eqref{formc} and  we define
$$
c^*=\sup\{c\in\re: V(c) \ \text{satisfies}\  V(c)(L)=0\}
$$
Observe that $c^*\geq \overline{c}$ where $\overline{c}$ is defined in \eqref{oc}.

On the other hand, if $c>c^*$ then $V(c) (L)>0$ and so, by uniqueness of solutions of \eqref{formc}, no solutions of \eqref{pb1} do exist.

We show that $c^*<+\infty$. In fact, if $u$ is a solution for \eqref{pb1} in the sense of Definition \ref{defin} then by \eqref{propdubis} one has 
\begin{equation}\label{alp}
\left\|\frac{d u}{dx}\right\|_{L^{2}(0,L)}\leq \frac{1}{\alpha}\|g\|_{L^{2}(0,L)}\,.
\end{equation}

Hence, by Proposition \ref{remequiv},    there exists $c\leq c^*$
$$
 c= a (x) \frac{d u}{dx}   -  \phi (u) - g \leq  \beta \left|\frac{d u}{dx}\right|   -\inf_{s\in\re} \phi (s)  + |g| \,.
$$
Therefore, using \eqref{alp} 
$$
c\leq \frac{1}{\sqrt{L}} \left(\frac{\beta}{\alpha}+1\right)\|g\|_{L^{2}(0,L)} -\inf_{s\in\re} \phi (s) \,,
$$
this, recalling Proposition \ref{remequiv} and the Definition of $c^*$,  implies that 
$$
c^*\leq \frac{1}{\sqrt{L}}\left(\frac{\beta}{\alpha}+1\right)\|g\|_{L^{2}(0,L)} -\inf_{s\in\re} \phi (s) <+\infty;
$$
that is   
 \eqref{cstar}.

Let us  show that $c^*$ is actually a maximum. Let $c_n\nearrow c^*$ so that  $V({c_n})\in H^1_0 (0,L)$. As ${c_n}$ are bounded, using   \eqref{stimah} one  gets
$$
\|V({c_n})\|_{H^1_{0}(0,L)}\leq C\,.
$$

Reasoning as before one can pass to the limit in 
$$
 a (x) \frac{d V({c_n})}{dx}   =  \phi (V({c_n})) + g+c_n\,\,\,\, \text{in}\ \mathcal{D}'(0,L)
$$
to obtain, using also  Proposition \ref{compar}, that 
$$
 a (x) \frac{d V({c^*})}{dx}   =  \phi (V({c^*})) + g+c^*\,\,\,\, \text{in}\ \mathcal{D}'(0,L). 
$$

The map $U=V_{\vert (-\infty,c^*]}$ is then well defined by $$U(c)=V(c) \ \text{for any}\ \  c\in (-\infty,c^*];$$ as there is no ambiguity in view of Proposition \ref{remequiv},  we will denote  by  $U(c)$  the function $V(c)$ once referring to the solution of \eqref{pb1} in the sense of Definition \ref{defin}. Hence \eqref{58} and  \eqref{posi} are  proven and the weak continuity of $U$ in $\huz$ is straightforward.  

\bigskip
 
Let us prove  \eqref{511}.  \newline By   Proposition \ref{compar}  one has that $U({c_1})\leq U({c_2})$. Assuming by contradiction that there exists  $x_0 \in (0,L)$ be such that $U({c_1}) (x_0) = U({c_2}) (x_0)$ then we define $$\sigma(x)=a(x)\frac{d (U({c_1})-U({c_2}))}{dx}(x)$$ so that 
$$
\sigma(x)=\phi(U({c_1})(x)) - \phi(U({c_2})(x)) +c_1 - c_2;
$$
in particular, since $\sigma$ is continuous on $(0,L)$ and  since 
\begin{equation}\label{c1c2}
\sigma(x_0)=c_1-c_2<0,  
\end{equation}

One has  $$\sigma(x) < \frac{c_1-c_2}{2} \ \ \forall x\,\,\text{in}\ \ (x_0 -\delta, x_0 +\delta)$$ for some $\delta>0$. This implies that 
\begin{equation} \label{c1c22}
\frac{d (U({c_1})-U({c_2}))}{dx}(x) \leq \frac{c_1-c_2}{2\beta}<0 \ \ \text{for a.e. }\ \ x\in (x_0 -\delta, x_0 +\delta), 
\end{equation}
a contradiction with the fact that $U({c_1})\leq U({c_2})$ as \eqref{c1c22}  gives that $U({c_1})-U({c_2})$ is strictly decreasing around $x_0$.  In fact, 
let $x\in (x_0 -\delta, x_0)$ and denote by $w=U({c_1}) - U({c_2})$; we have $w(x)\leq 0$ and $w(x_0)=0$. On the other hand  
$$
w(x)=\int_{x_0}^{x}\frac{d w}{dx}(s)\ ds =-\int_{x}^{x_0}\frac{d w}{dx}(s)\ ds \stackrel{\eqref{c1c22}}{>}0\,.
$$

\bigskip 

Let us formally prove  \eqref{uniform}.  Let $c_n<c^*$ such that $c_n\searrow- \infty$. As before 
$$
\left\|U({c_n})\right\|_{H^1_0(0,L)}\leq \frac{1}{\alpha}\|g\|_{L^{2}(0,L)}\,.
$$
so that, up to subsequences, $U({c_n})\rightharpoonup u$ in $H^1_0(0,L) $ and uniformly (by compact embedding).  We can apply Theorem \ref{maintheo}: either $u\equiv 0$ or assume by contradiction that  there exists $x_0\in (0,L)$ such that 
$$
u(x_0)>0,
$$
and so,   in  a small neighborhood of $x_0$ of the form $(x_0 -\delta, x_0 +\delta)$, one has 
$$
c_n=a(x) \frac{d U({c_n})}{dx}-\phi(U({c_n})) -g\,. 
$$
i.e., using Proposition \ref{prop35} one has that  $|c_n|$ is bounded, that is a contradiction. 

\bigskip 

In order to conclude the proof of Theorem \ref{cara} we are left with the proof of \eqref{anyu} which easily follows by Proposition \ref{compar}: in fact, if $u$ is a solution of problem \eqref{pb1} in the sense of Definition \ref{defin},  then $u=U(c)$ by uniqueness of the  solution of problem \eqref{formh}. 

\end{proof}

\section{Existence of solutions in the sense of Definition \ref{defin}}
\label{secexistence}\label{Sec8}

Up to now we explored the consequences of Theorem \ref{maintheo} (Alternative) by assuming the existence of a solution of problem \eqref{pb1} in the sense of  Definition \ref{defin}, but we never proved the existence of such a solution.  

In this section we fix $a$ and $\phi$ satisfying \eqref{a1} and \eqref{condphi1} as well as \eqref{p+} and \eqref{p-} (to be compared with  Theorem~\ref{45}),  and we consider $g$ satisfying \eqref{g1} as a parameter.  We aim at constructing two  large sets of  $g$'s for which a solution in the sense of Definition~\ref{defin} does exist.

\subsection{A first remark about the set of solutions of problem \eqref{pb1} in the sense of Definition~\ref{defin}}

Let us begin with the following remark, which is very simple, but essential.

\begin{remark}

Define 
  the set $\mathcal{G}$ of "good data`` by 
 \begin{equation}\label{cappg}
 \mathcal{G}=\{g\in L^{2}(0,L): \text{$\exists$ at least one solution $u$ of  \eqref{pb1} in the sense of Definition \ref{defin}} \},
 \end{equation}
and, as in \eqref{setU},  the set 
\begin{equation}
\label{81bis}
\mathcal{U}=\{\hat{u}\in H^1_0 (0,L) \ \text{such that}\ \  \phi(\hat{u})\in L^2(0,L)\}\,. 
\end{equation}

It is clear that every $u$ which is a solution of problem \eqref{pb1} in the sense of Definition~\ref{defin} is an element of $\mathcal{U}$.

On  the other side, for every $u\in\mathcal{U}$,   taking  any $c\in \re$, and setting 
 \begin{equation}\label{ggs}
g= a(x)\frac{du}{dx} -\phi(u)-c\,,
 \end{equation}
it is clear that $u$ is a solution of problem \eqref{pb1} in the sense of  Definition \ref{defin} for the source term $g$. Actually, every $g\in \mathcal{G}$ is of the form \eqref{81bis} for some $u\in  \mathcal{U}$ and $c\in\mathbb{R}$.

This establishes a very strong relation between the two sets $ \mathcal{G}$ and $ \mathcal{U}$. In particular  $\mathcal{G}$ is non-empty if and only if $ \mathcal{U}$ is non-empty.

Therefore the study of the set $ \mathcal{U}$ is essential when studying the existence of solutions of problem \eqref{pb1} in the sense of Definition~\ref{defin}

 \qed
\end{remark}

\begin{remark}[Model example]
\label{rem80}
In the model case where the nonlinearity $\phi$ is given by  
$$
\phi(s)=\frac{c}{|s|^{\gamma}}\,\, \mbox{ with } \,\,c>0 \,\,\mbox{ and }\,\, 0<\gamma<1,
$$
it is easy to see that $u$ defined by 
\begin{equation}
\label{83}
u(x)=Kx^{\lambda}(L-x)^{\lambda} \,\,\mbox{ with }\,\, \frac 12<\lambda<\frac{1}{2\gamma},\,\,\forall K\in\mathbb{R},\,\, K\not=0,
\end{equation}
belongs to $\mathcal{U}$ (since then $\dys\frac{du}{dx}\in L^2(0,L)$ and $\dys\frac{1}{|u|^\gamma}\in L^2(0,L)$ because we have $0<\gamma<1$). Therefore $u$ is a solution of problem \eqref{pb1} in the sense of Definition~\ref{defin} for the source term
\begin{equation}
\begin{split}
\label{83bis}
g(x)&=a(x)\frac{du}{dx}-\phi(u)-c= \\&=a(x)\left(K\lambda x^{\lambda-1}(L-x)^{\lambda} - K\lambda x^{\lambda}(L-x)^{\lambda-1}\right)- \frac{c}{|K|^{\lambda} x^{\lambda \gamma}(L-x)^{\lambda\gamma}} -c,\,\,\\\nonumber \forall c\in\mathbb{R}.
\end{split}
\end{equation}

This is a first example of a solution of problem \eqref{pb1} in the sense of Definition~\ref{defin}, which will be a model for the whole of the present section.

 \qed
\end{remark}

\subsection{A first large class of good data}

 Starting from the idea of the example presented in  Remark~\ref{rem80}  we will show that it is always possible to construct explicit local  solutions $w^r$   of problem \eqref{pb1} emerging from a point $\overline{x}\in [0,L)$ towards the right side, 
 and $w^\ell$ coming backward  from a point $\overline{y}\in (0,L]$ to the left side provided the datum $g$ is chosen accordingly.

 For the sake of exposition we start by showing how  these solutions can be constructed   in the model case   \begin{equation}\label{modelphi}
\phi(s)=\frac{c}{|s|^{\gamma}}\,\, \mbox{ with } \,\,c>0 \,\, \text{ and }\ \  0<\gamma<1
.\end{equation}
\bk 
Define for $y\in [0,L)$ and for some $K^r\in \mathbb{R}$, $K^r\not=0$, $\lambda^r>0$ and $\delta>0$, the function $w^r$ by 
$$
w^r(x)=K^r(x-y)^{\lambda^{r}}\ \ \text{for}\  y\leq x\leq y+\delta\,.
$$
 Since
$$
\frac{dw^{r}}{dx} = K^r \lambda^{r}(x-y)^{\lambda^{r} -1}\ \ \text{and}\ \ \phi(w^{r})= \frac{c}{(|K^{r}| (x-y)^{\lambda^{r}} )^\gamma } \ \ \text{for}\  y\leq x\leq y+\delta,
$$
we have $\displaystyle\frac{dw^{r}}{dx}$ and  $\phi(w^{r})$ in $L^2(y,y+ \delta)$ if and only if 
\begin{equation}\label{lamb}
\frac12 <\lambda^r<\frac{1}{2\gamma},
\end{equation}
this choice is possible since $0<\gamma<1$.

Reasoning in the same way, define for $y\in(0,L]$ and for some $K^\ell\in \mathbb{R}$, $K^\ell\not=0$, $\lambda^\ell>0$ and $\delta>0$, the function $w^\ell$ by

$$
w^{\ell}(x)=K^{\ell}(y-x)^{\lambda^{\ell}}\ \ \text{for}\  y-\delta\leq x<y\,,
$$
for which we have $\displaystyle\frac{dw^{\ell}}{dx}$ and  $\phi(w^{\ell})$ in $L^2(y-\delta,y)$ if and only if 
$$\frac12< {\lambda^{\ell}}<\frac{1}{2\gamma}\,.$$

  Now,  we show how given any $x_1$ and $x_2$ with $0\leq x_1<x_2\leq L$ and $\delta>0$ with $x_1+\delta<x_2-\delta$, one is able to construct a solution of problem \eqref{pb1} in the sense of Definition \ref{defin} in any interval of the form $[x_1,x_2]$ and not only on $[0,L]$. In fact,  take any  function $w^{int}(x)\in H^1(x_1 +\delta,x_2-\delta) $ such  that $w^{int}(x_1 +\delta)=w_1(x_1 +\delta)$ and $w^{int}(x_2-\delta)=w_{2}(x_2-\delta)$. We also request that, for some $\eta>0$, $w^{int}(x)\geq \eta$ in $x\in(x_1 +\delta,x_2-\delta)$. \bk Now we define
\begin{equation}\label{glue}
w(x) =
\begin{cases}
 w_1^r & x\in[x_1,x_1+\delta),\\
w^{int}& x\in [x_1 +\delta,x_2-\delta],\\
  w_2^l & x\in (x_2-\delta, x_2].
\end{cases}
\end{equation}
Then, if we set 
$$
g=a(x)\frac{ dw}{dx}-\frac{c}{|w|^\gamma}\,, \ x\in (x_1,x_2),  
$$
it is easy  to check that $g\in L^2(x_1,x_2)$, 
 and that $w$ is a solution of
$$
\begin{cases}
\dys w\in H^1_0 (x_1,x_2), \,\,\,\, \frac{c}{|w|^\gamma}\in L^2(x_1,x_2), \\
\dys -\frac{ d}{dx}\left(a(x)\frac{ dw}{dx}\right)  = - \frac{ d}{dx}\left(\frac{c}{|w|^\gamma}\right)
 - \frac{ d g}{dx}  \,\,\,\, \text{in}\,\,[x_1,x_2]\,.
\end{cases}
$$

 By modifying the value of $\lambda$, with   $\frac{1}{2}<\lambda<\frac{1}{2\gamma}$,  the value and the sign   of $K$,   and reasoning around a finite number of points $x_0=0<x_1<...<x_{n} <x_{n+1}=L$,   we can construct a bunch  of  functions which  behave as ${w}$ between  $x_i$ and $x_{i+1}$,   and so a large class of data for which there  exists a solution of problem \eqref{pb1} in the sense of Definition \ref{defin}. 
 
  \begin{figure}[htbp]\centering
\includegraphics[width=5in]{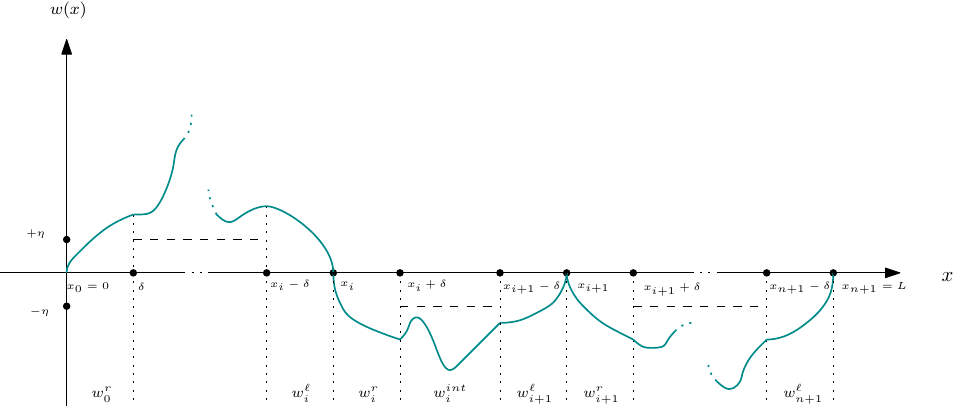}
\caption{Building up the function $w(x)$}
\end{figure}
 
  More precisely let $x_0=0<x_1<... <x_{n} <x_{n+1}=L$ (i.e. $x_1, ..., x_n$ are the internal points) and let  $\delta>0$ such that 
  $$
  x_i+\delta< x_{i+1} -\delta\,, \ \ \text{for} \ 0\leq i\leq n\,. 
  $$
Moreover  for the internal points $x_i$,  $i=1, ..., n$,  let us consider $\lambda_i^{r}$, $\lambda_i^{\ell}$ satisfying \eqref{lamb}, and let  $K_i^{r}$, $K_i^{\ell}$ in $\re\setminus\{0\}$; also consider $\lambda_0^{r}$ and $\lambda_{n+1}^{\ell}$ in $\re$ satisfying \eqref{lamb},  and  $K_{0}^{r}$ and $K_{n+1}^{\ell}$ in $\re\setminus\{0\}$ for the extremal points $0$ and $L$.  We assume 
 $$
 K_i^{r} K_{i+1}^{\ell}>0,\ \  \text{for}\ i=0, ..., n.
 $$
 
Now, around each internal point $x_i$, $i=1, ..., n$, we set

 \begin{equation}
 \begin{split}
\label{wominus} w_i^{\ell}(x)&= K_i^{\ell} (x_i-x)^{\lambda^{\ell}_i} \ \text{for}\ x_i-\delta<x <x_i\,,\\
 w_i^{r}(x)&= K_i^{r} (x-x_i)^{\lambda^{r}_i} \ \text{for}\ x_i<x <x_i +\delta,
 \end{split}
 \end{equation}
and for the extremities (i.e. $i=0$ and $i=n+1$)
 \begin{equation}
 \begin{split}
 \label{wominus2}
w_0^{r}(x)&= K_0^{r} x^{\lambda^{r}_0} \ \text{for}\ 0<x <\delta, \\
w_{n+1}^{\ell}(x)&= K_{n+1}^{\ell} (L-x)^{\lambda^{\ell}_{n+1}} \ \text{for}\ L-\delta<x <L\,.
 \end{split}
 \end{equation}
 In the remaining intervals,  which are  of the form $(x_i+\delta,x_{i+1}-\delta)$, $i=0, ..., n$, we   define   $w_{i}^{int}(x)$ as any  function in $H^1(x_i+\delta,x_{i+1}-\delta)$ which is continuously joined  with the functions $w_i^{r}$ and $w_{i+1}^{\ell}$ defined in \eqref{wominus} and \eqref{wominus2}, i.e. 
 $$
 w_i^{int} (x_i+\delta)=  w_i^{r}(x_i+\delta) \ \ \text{and} \ \ w_{i}^{int} (x_{i+1}-\delta)=  w_{i+1}^{\ell}(x_{i+1}-\delta)\ \ \forall \ i=0,..., n\,, 
 $$
and which satisfies,  for some $\eta>0$
 $$
 |w_{i}^{int}(x)|\geq \eta>0, \ \ \text{for any}\ \ x\in (x_i+\delta,x_{i+1}-\delta).
 $$

Summarizing  we have defined a function  $w\in H^1_0 (0,L)$ that in any interval $(x_i,x_{i +1})$, $i=0, ..., n$,  is given by
 \begin{equation}\label{glue3}
w(x) =
\begin{cases}
 K_{i}^r (x-x_i)^{\lambda_{i}^r},  & x\in (x_i,x_i+\delta],\\
  w_{i}^{int} (x),  & x\in (x_i+\delta,x_{i+1}-\delta),\\
K_{i+1}^{\ell} (x_{i+1}-x)^{\lambda_{i+1}^{\ell}},& x\in [x_{i+1}-\delta,x_{i+1}).
\end{cases}
\end{equation}
 
 Finally we define the function $g$ by
\begin{equation}\label{deduce}\dys g=\frac{ dw}{dx}-\phi(w)\ \text{in}\ \ \mathcal{D}'(0,L),\end{equation}
 and we observe that   $w\in H^1_0 (0,L)$ is a  solution of problem \eqref{pb1} in the sense of Definition \ref{defin} with $g$ as datum.

One then has the following result.   
 \begin{proposition}\label{bunch}
 Assume \eqref{cond1}--\eqref{a1} and \eqref{pM}. For $n\in \mathbb{N}$,  fix $n$ points $x_1, ..., x_n$ such that  $0=x_0<x_1<...x_{n} <x_{n+1}=L$. Also consider  for the internal points $x_i$,  $i=1, ..., n$,  let us consider $\lambda_i^{r}$, $\lambda_i^{\ell}$ satisfying \eqref{lamb}, and let  $K_i^{r}$, $K_i^{\ell}$ in $\re\setminus\{0\}$; also consider $\lambda_0^{r}$ and $\lambda_{n+1}^{\ell}$ in $\re$ satisfying \eqref{lamb},  and  $K_{0}^{r}$ and $K_{n+1}^{\ell}$ in $\re\setminus\{0\}$ for the extremal points $0$ and $L$.  We assume 
 $$
 K_i^{r} K_{i+1}^{\ell}>0,\ \  \text{for}\ i=0, ..., n.
 $$
 Then there exists $w\in H^1_0 (0,L)$ solution of \eqref{pb1} in the sense of Definition \ref{defin} with datum $\dys g=\frac{ dw}{dx}-\phi(w)$ and such that $w(x_i)=0$ for $i=1, ..., n$. 
 \end{proposition}

\begin{remark}
Observe that, by the arguments in the proof of Theorem \ref{36},   data  $g$ belonging to $\mathcal{G}$ are necessarily unbounded around each zero's of the functions $w$ through which it is defined; in fact, near each zero's 
$$
\frac{dw}{dx} \sim |x-x_i|^{\lambda_i -1}\ \ \text{and}\ \ \phi({w})\sim  |x-x_i|^{-\lambda_i\gamma},
$$
and $g$ is defined to compensate them.  The threshold value $\lambda_i=\frac{1}{1+\gamma}$  is the one in which the two terms are in balance. 

Also observe that in this subsection for simplicity we choose to present the construction for the model case of $\phi$ given by \eqref{modelphi}. As a matter of fact  this construction can be done for a general $\phi$ satisfying 
\eqref{condphi1}--\eqref{p3},  \eqref{57n} and \eqref{58n} using the idea presented  in Theorem \ref{t85} of Subsection \ref{s83}. 

 \qed 
\end{remark}

\subsection{Obtaining solutions for any  datum  $g$ by modifying it on $[L-\delta,L]$} \label{s83}

In the previous subsection  we constructed a bunch of  solutions of  problem \eqref{pb1} in the sense of Definition \ref{defin}. From this construction we obtained by \eqref{deduce}  a large class of data $g$ for which there exist a solution of  problem \eqref{pb1} in the sense of Definition \ref{defin}.

In this subsection, we change   viewpoint and  we   construct, for any fixed $g\in L^{2}(0,L)$ and for any $\delta$ with  $0<\delta<L$,  a datum $\hat{g}$ which coincides with $g$ on $[0,L-\delta]$  for which problem \eqref{pb1} admits a  solution in the sense of Definition \ref{defin}. This construction uses  an   idea similar to the one used in the previous subsection, but now exploits  a shooting argument for backward  solutions starting from $x=L$  on $[L-\delta,L]$.  Moreover, this argument  is presented here   in the case of a general    $\phi$ satisfying  
\eqref{condphi1}--\eqref{p3},  \eqref{57n} and \eqref{58n}, and not only in the model case of $\phi$ given by   \eqref{modelphi}. 

\begin{theorem}\label{t85} Assume that \eqref{cond1} holds true, and that  the data $(a,\phi)$ satisfy hypotheses \eqref{a1},  \eqref{condphi1}--\eqref{p3},   \eqref{57n}-\eqref{58n},  and \eqref{533bis}--\eqref{gf}.  
Then for any $g\in L^2(0,L)$ and $\delta$ with $0<\delta<L$, there exist  $\hat{g}$ in $\mathcal{G}$ such that $g=\hat{g}$  in $[0,{L-\delta}]$. 
\end{theorem}

\begin{proof}[\bf Proof]
By Theorem \ref{shoot}, there exists $v$ such that 
$$
\begin{cases}
v\in H^1 (0,L),\,\,\,\, \phi(v)\in L^2(0, L),  \\
\dys a(x)\frac{dv}{dx}  =  \phi (v) + g  \,\, \text{in}\,\,(0,L),\\
v(0)=0\,. & 
\end{cases}
$$

By possibly choosing a smaller $\delta$ we can assume that $v(L-\delta)\neq 0$. Our aim is to define the solution $\hat{u}$ of problem \eqref{pb1} in the sense of Definition \ref{defin} relative to a suitably chosen $\hat{g} $
$$
\hat{u}(x)=\begin{cases}
 v(x)& \text{in}\ [0,L-\delta),\\
 w(x)&  \text{in}\ [L-\delta,L]\,,\\ 
\end{cases}
$$
where  $w(x)$ satisfies 
\begin{equation}\label{vdelta}
\begin{cases}
w\in H^1 (L-\delta,L),\,\,\,\, \phi(w)\in L^2(L-\delta, L),  \\
w(L-\delta)=v(L-\delta)\,,\\
w(L)=0\,.
\end{cases}
\end{equation}
In fact, it is not difficult to check that, if such a $w$ exists then $\hat{u}$ is a solution of problem \eqref{pb1} in the sense of Definition \ref{defin} for the datum
$$
\hat{g}(x)=\begin{cases}
 g(x)& \text{in}\ [0,L-\delta),\\
 a(x)\displaystyle\frac{dw}{dx} -\phi(w)&  \text{in}\ [L-\delta,L]\,.\\ 
\end{cases}
$$

Therefore, to conclude it suffices to  construct a function $w$ satisfying \eqref{vdelta}.

In order to simplify the argument we observe that, by mean of the change of variable $y=L-x$, the construction of  a function $w$ satisfying \eqref{vdelta} is equivalent to the construction of a function $\overline{w}$ satisfying 
\begin{equation}\label{ovdelta}
\begin{cases}
\overline{w}\in H^1 (0,\delta),\,\,\,\, \phi(\overline{w})\in L^2(0,\delta),  \\
\overline{w}(\delta)=v(L-\delta)\,,\\
\overline{w}(0)=0\,.
\end{cases}
\end{equation}

Assume first that $v(L-\delta)>0$. We define the following function
\begin{equation}\label{p>1}
\phi^{\oplus}(s)=\phi(s) - \inf_{s\in\re} \phi(s)+1\,;
\end{equation}
it is easy to check that $\phi^{\oplus}$ also satisfies \eqref{condphi1}--\eqref{p3},  \eqref{57n} and \eqref{58n}, moreover \begin{equation}\label{phip}\phi^{\oplus}(s)\geq 1\ \ \text{for any}\ s\in\re.\end{equation}

Now, for a fixed $K>0$ to be chosen later, one has, by Theorem \ref{shoot} (observe that since $K>0$ then $K \phi^{\oplus} (0)=+\infty$), that there exists at least a  solution $w^{\oplus}$ to  
$$
\begin{cases}
w^{\oplus}\in H^1 (0,\delta),\,\,\,\, \phi^{\oplus}(w^{\oplus})\in L^2(0, \delta),  \\
\dys \frac{dw^{\oplus}}{dx}  = K \phi^{\oplus} (w^{\oplus})  \,\, \text{in}\,\, (0,\delta),\\
w^{\oplus}(0)=0,\,
\end{cases}
$$
and by and Proposition \ref{coneL} $w^{\oplus}>0$ on $(0,\delta)$.

First observe that  $\phi(w^{\oplus})\in L^{2}(0,\delta)$ if and only if $\phi^{\oplus}(w^{\oplus})\in L^{2}(0,\delta)$. 

Now,  in view of \eqref{phip}, we  have 
\begin{equation}\label{divid}
\frac{1}{\phi^{\oplus} (w^{\oplus})}\frac{dw^{\oplus}}{dx}  = K \,,\ \ \text{for any}\ x\in (0,\delta). 
\end{equation}
Now observe that 
\begin{equation}\label{1phip}0\leq \frac{1}{\phi^{\oplus}(s)}\leq  1\ \ \text{for any}\ s\in\re. \end{equation}

Hence,   we define
$$
\zeta (s)=\int_0^s{\frac{1}{\phi^{\oplus} (r)}}dr,  
$$
 and, using \eqref{phip},  $\zeta$ is strictly increasing on $\re^+$, by \eqref{p3} one has $\zeta (0)=0$, 
 and, recalling \eqref{58n},   for any $\eta>0$ there exists $c_\eta>0$ such that 
$$
\frac{1}{\phi^{\oplus} (r)}\geq \frac{1}{c_\eta}>0, \ \ \text{for all $r$ in}\  [\eta,+\infty),
$$
 so that we deduce that 
 $$
 \lim_{s\to +\infty} \zeta (s)=+\infty, 
 $$
yielding in particular that  $\zeta: \re^+\to\re^+$  is a bijection.  
 
Ultimately, by \eqref{divid} one has 
$$
\begin{cases}
\dys \frac{d}{dx} \zeta (w^{\oplus}) = K \  & \text{in}\;(0,\delta),\\
\zeta(0)=0\,, & 
\end{cases}
$$
 i.e., 
 $$
 \zeta (w^{\oplus}(x))=Kx\, \ \ \ \text{in}\;(0,\delta), 
 $$
 that is, choosing 
 $$
 K=\frac{\zeta({v(L-\delta)})}{\delta}, 
 $$
 one easily check that $w^\oplus$ satisfies \eqref{ovdelta} and we conclude.

In the case where   $v(L-\delta)<0$,   we define 
$$
\phi^{\ominus}(s)= \phi^{\oplus}(-s),
$$
 and  we reason as before and we may pick a positive $K>0$ and  a solution $\underline{w}$  of 
 $$
\begin{cases}
\underline{w}\in H^1 (0,\delta),\,\,\,\, \phi^{\ominus}(\underline{w})\in L^2(0, \delta),  \\
\dys \frac{d\underline{w}}{dx}  = K \phi^{\ominus} (\underline{w})  \,\, \text{in}\,\,(0,\delta),\\
\underline{w}(0)=0\,, 
\end{cases}
$$
 such that $\underline{w}(\delta)=-v(L-\delta)$.  To conclude we  define $w^{\ominus}= - \underline{w}$ and observing that 
$$
 \phi (w^{\ominus}) \in L^{2}(0,\delta) \iff \phi^{\oplus} (w^{\ominus})\in L^{2}(0,\delta) \iff \phi^{\ominus} (\underline{w})\in L^{2}(0,\delta)
$$
we get that $w^\ominus$ satisfies \eqref{ovdelta}. 

 \begin{figure}[htbp]\centering
\includegraphics[width=3in]{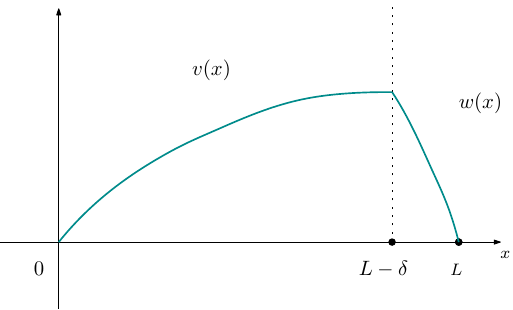}
\caption{Gluing up $v(x)$ and $w(x)$}
\end{figure}

\end{proof}

\section{Stability and instability of the solutions of  approximate equations}\label{Sec7}

The following  proposition shows that any solution of \eqref{pb1} in the sense of Definition \ref{defin} is  not isolated, or, in other terms, can be obtained as a limit of solutions of convenient {approximating} problems, each of those problems being different from the limit one. Let us stress that in these approximating  problems the sequence $\phi_n$  can be any reasonable approximation of $\phi$, but that the sequence $g_n$ has to be chosen accordingly.  
\begin{proposition}\label{stab}
Assume that \eqref{cond1} holds true, and that  the data $(a,g,\phi)$ satisfy hypotheses \eqref{a1}--\eqref{condphi1}.  Assume also   that there exists a solution $u\in \huzu$ of problem \eqref{pb1} in the sense of Definition \ref{defin}, i.e. 
\begin{equation}
\begin{cases}\label{71}
u\in \huzu, \, \,\,\,  \phi(u)\in L^2(0,L), \\\\
\dys -\frac{d}{dx}\left(a (x) \frac{du}{dx}\right)  = - \dys\frac{d \phi (u)}{dx} - \frac{dg}{dx} \ \  \text{in}\,\;\mathcal{D}'(0,L)\,.
\end{cases}
\end{equation}
\indent Fix any  reasonable approximation $\phi_n$ of $\phi$,  with $\phi_n\in C_{0}^{b} (\re)$,   $\phi_n\neq \phi$ for any $n\in\mathbb{N}$.  Then there exists a sequence  $$\text{$g_n\in L^2(0,L)$   with  $g_n\to g $ strongly in $L^2(0,L)$,}$$ such that there exists at least a    classical weak  solution $u_n$ of 
\begin{equation}
\begin{cases}\label{trivial}
u_n\in \huzu, &  \\\\
\dys -\frac{d}{dx}\left(a (x) \frac{du_n}{dx}\right)  = - \dys\frac{d \phi_n (u_n)}{dx} - \frac{d g_n}{dx} \ \  \text{in}\,\,\mathcal{D}'(0,L)\,,
\end{cases}
\end{equation}
such that 
$$u_n \rightarrow u \ \ \text{strongly in}\ \ H^1_0 (0,L)\,.$$
\end{proposition}
\begin{proof}[\bf Proof]
As $u$ is a solution of problem \eqref{pb1} in the sense of Definition \ref{defin},  \eqref{71} is equivalent to the existence of some $c\in\re$ such that 
 $$
a (x) \frac{du }{dx}  = \phi (u) + g+ c\,,
$$
in view of Proposition \ref{remequiv}. \newline 
\indent Let $\phi_n$  be any reasonable approximation of $\phi$,  with $\phi_n\in L^\infty (\re)$, $\phi_n\neq \phi$.   Define 
$$
g_n = \phi(u)-\phi_n(u) +g;
$$
then $u_n=u$ is a solution of \eqref{trivial}, and the theorem is proved.

\end{proof}

\begin{remark}
The reader can be surprised  by the fact $u_n$ is   equal to $u$ for every $n$. This  can be amended in the following way: take any interval $[A,B]$ with  $0<A<B<L$ for which $u(x)\geq \delta$ on $[A,B]$ (or for which $u(x)\leq -\delta$ on $[A,B]$) for a certain $\delta>0$, and replace $u$ by 
$$
u_n=\begin{cases}
u& \text{in}\ [0,L]\backslash (A,B), \\
v_n &\text{in}\ [A,B]\,,
\end{cases}
$$
 where $v_n \in H^{1}(A,B)$ is any sequence such that $(v_n- u)\to 0$ in $H^{1}_{0}(A,B)$, with $v_n\geq {\eta}$ (or with $v_n\leq {-\eta}$) for a certain $\eta>0$. 
 
\qed
\end{remark}

As a counterpart of the previous result,  we present the following  one,  which shows that, if $\phi(0) =+\infty$,  then the choice of the sequence $g_n$ is crucial in order to obtain that the solutions of the approximating problems can converge to a given solution as in Proposition \ref{stab}. In fact we will show that a bad choice of $g_n$ can lead to a  sequence of "approximating problems" whose solutions converge to $u=0$.

\begin{proposition}
Assume that \eqref{cond1} holds true, and that  the data $(a,g,\phi)$ satisfy hypotheses \eqref{a1}--\eqref{p3}.  Assume also  that there exists a solution $u\in \huzu$ of problem \eqref{pb1} in the sense of Definition \ref{defin}, i.e. 
\begin{equation}
\begin{cases}
u\in \huzu, \,\,\,\, \phi(u)\in L^2(0,L), \\\\
\dys -\frac{d}{dx}\left(a (x) \frac{du}{dx}\right)  = - \dys\frac{d \phi (u)}{dx} - \frac{dg}{dx} \ \  \text{in}\,\,\mathcal{D}'(0,L)\,.
\end{cases}
\end{equation}
\indent Fix any reasonable approximation $\phi_n$ of $\phi$, with  $\phi_n\in C_{0}^{b} (\re)$\bk. Then one can extract  a subsequence  $\phi_{n'}$ and  find  a sequence (indexed by $n'$) $g_{n'}$  such that   $${g}_{n'} \in L^2(0,L)  \ \ \   {{g}_{n'}}\to g \ \ \text{strongly in}\ \  L^2(0,L), $$ such that  for any (classical weak)  solution $u_{n'}$, granted by Proposition \ref{propexis}, of the "approximating problems"
\begin{equation}
\begin{cases}\label{trivial2}
u_{n'}\in \huzu,   \\\\
\dys -\frac{d}{dx}\left(a (x) \frac{du_{n'}}{dx}\right)  = - \dys\frac{d \phi_{n'} (u_{n'})}{dx} - \frac{d g_{n'}}{dx} \ \  \text{in}\,\,\mathcal{D}'(0,L)\,,
\end{cases}
\end{equation}
 one has  $$u_{n'} \rightharpoonup 0 \ \ \text{in}\ \ H^1_0 (0,L)\,.$$ 
\end{proposition}

\begin{proof}[\bf Proof]
Consider  $\overline{g}_n$ a given sequence such that 
$$
\overline{g}_n \in L^\infty (0,L)\ \ \text{and}\ \ \overline{g}_n\to g \ \text{strongly in}\ L^2(0,L)\,. 
$$
Now, fix $n\in\mathbb{N}$ and consider the (classical weak) solutions   $v_n^k$ of 
\begin{equation}
\begin{cases}\label{vkn}
v_n^k\in \huzu,  \\\\
\dys -\frac{d}{dx}\left(a (x) \frac{d v_n^k}{dx}\right)  = - \dys\frac{d \phi_k ( v_n^k)}{dx} - \frac{d \overline{g}_n}{dx} \ \  \text{in}\,\,\mathcal{D}'(0,L)\,.
\end{cases}
\end{equation}

As $\phi_k$ is in $C_{0}^{b} (\re)$ then by Proposition \ref{propexis} such a $v_n^k$ does exist. Moreover
\begin{equation}
\label{propdun2}
\left\|\frac{d v_n^k}{dx}\right\|_{L^2(0,L)}\leq \frac{1}{\alpha} \|\overline{g}_n\|_{L^2(0,L)}\leq C, 
\end{equation}
where $C$ does not depend on $k$ and $n$.

 As $n$ is fixed and $\overline{g}_n \in L^\infty (0,L)$ then, using Theorem \ref{36},  one has that 
\begin{equation}
 v_n^k \rightharpoonup 0\ \ \text{weakly in $\huzu$}\ \ \text{as}\ \ k\to\infty\,,
\end{equation}
and strongly in $L^2 (0,L)$. In particular, if $\vare_n$ is a sequence of positive vanishing constant,   for fixed $n$ one can pick ${k^*}(n)$ such that 
\begin{equation}\label{68}
\| v_n^{{k}}\|_{L^2(0,L)}\leq \vare_n\ \ \text{for any}\ \ k\geq k^*(n)\,.
\end{equation}
It is clear that, as $n$ is fixed, one can choose a strictly increasing  function $\overline{k}:\mathbb{N}\mapsto\mathbb{N}$ such that
$$
\overline{k}(n)\geq k^*(n), \ \ \text{for any}\ n\in\mathbb{N}\,. 
$$

Let us indicate by $\mathbb{J}=\overline{k}(\mathbb{N})$ the image of such bijection. Now, with the following notations 
 $$
 n'=\overline{k}(n) \ (\text{i.e. $n=\overline{k}^{-1} (n')$})\ \ \,  \text{and}\ \ {g}_{n'}=\overline{g}_{\overline{k}^{-1}(n')}\,,
 $$ 
as $u_{n'}=v_n^{\overline{k}(n)}$, then \eqref{vkn} (at $k=\overline{k}(n)$) reads as
 \begin{equation}\label{65j}
\begin{cases}
u_{n'} \in \huzu,  \\\\
\dys -\frac{d}{dx}\left(a (x) \frac{d u_{n'}}{dx}\right)  = - \dys\frac{d \phi_{n'} (u_{n'})}{dx} - \frac{d {g}_{n'}}{dx} \ \  \text{in}\,\,\mathcal{D}'(0,L)\,.
\end{cases}
\end{equation}
Now, using \eqref{68}, one has that
$$
\|u_{n'}\|_{L^2(0,L)}\leq \vare_{\overline{k}^{-1} (n')}=\eta_{n'}\,,
$$
where $\eta_{n'}\to 0$ as $n'\to \infty$,  i.e. 
$$
u_{n'} \rightharpoonup 0\ \ \text{weakly in $\huzu$}\ \ \text{as}\ \ n'\to\infty\,,
$$
and the proof is concluded. 
\end{proof}
\begin{remark}
In the proof of Proposition \ref{stab} we  saw that  the choice of $g_n$ is determined by the choice of the approximating sequence $\phi_n$.  In contrast it is not clear how to choose a suitable reasonable approximation $\phi_n$ when a sequence $g_n$ is given. \newline\indent  In other terms,   when approximating  problem \eqref{pb1} the choice of $g_n$ has to be conveniently made  in order to  balance  the velocity of approximations of $\phi_n$  if one does not want to end up  with $u=0$. 

 \qed
\end{remark}

\vspace{2cm}
\section*{Acknowledgements}
\smallskip
\noindent The authors were partially supported by  the Gruppo Nazionale per l'Analisi Matematica, la Probabilit\`a e le loro Applicazioni (GNAMPA) of the Istituto Nazionale di Alta Matematica (INdAM). The second author was supported by the Ministerio de Ciencia, Innovaci\'on y Universidades (MCIU), Agencia Estatal de Investigaci\'on (AEI), and Fondo Europeo de Desarrollo Regional (FEDER) under Research Project PGC2018-096422-B-I00, Junta de Andaluc\'ia, Consejer\'ia de Transformaci\'on Econ\'omica, Industria, Conocimiento y Uni\-ver\-si\-dades-Uni\'on Europea grant UAL2020-FQM-B2046 and  FQM-116.

\end{document}